\documentclass[11pt]{amsart}

\newif\iflongversion
\longversiontrue
\usepackage{geometry}

\usepackage{xcolor}
\definecolor{amaranth}{rgb}{0.9, 0.17, 0.31}
\usepackage[colorlinks=true,citecolor=amaranth,linkcolor=black]{hyperref}

\usepackage[T1]{fontenc}
\usepackage{graphicx}
\usepackage{amsmath}
\usepackage{amssymb}
\usepackage{mathrsfs}
\usepackage{amsthm}
\usepackage{thmtools}
\usepackage{cleveref}

\usepackage{tikz}
\usepackage{tikz-cd}
\usepackage{todonotes}
\usepackage{cite}
\usepackage{booktabs}
\usepackage{colortbl}
\usepackage{colonequals} % for rendering := via \colonequals
\usepackage[shortlabels]{enumitem} % custom label enumeration
\ifcsname theorem\endcsname{}\else\declaretheorem[parent=section]{theorem}\fi
\ifcsname corollary\endcsname{}\else\declaretheorem[sibling=theorem]{corollary}\fi
\ifcsname lemma\endcsname{}\else\declaretheorem[sibling=theorem]{lemma}\fi
\ifcsname proposition\endcsname{}\else\declaretheorem[sibling=theorem]{proposition}\fi
\ifcsname conjecture\endcsname{}\else\fi
\ifcsname problem\endcsname{}\else\fi
\ifcsname question\endcsname{}\else\fi
\ifcsname definition\endcsname{}\else\declaretheorem[sibling=theorem, style=definition]{definition}\fi
\ifcsname exercise\endcsname{}\else\fi
\ifcsname example\endcsname{}\else\fi
\ifcsname remark\endcsname{}\else\declaretheorem[sibling=theorem, style=remark]{remark}\fi

\providecommand {\Z}{{\mathbb Z}}
\providecommand {\Q}{{\mathbb Q}}
\providecommand {\R}{{\mathbb R}}
\providecommand {\C}{{\mathbb C}}

\renewcommand {\P}{{\mathbb P}}

\providecommand{\GL}{\operatorname{GL}}

\providecommand {\from}{{\colon}}

\providecommand{\spec}{\operatorname{Spec}}

\providecommand{\Hom}{\operatorname{Hom}}

\providecommand{\Aut}{\operatorname{Aut}}

\providecommand{\Pic}{\operatorname{Pic}}

\providecommand{\rk}{\operatorname{rk}}

\providecommand{\im}{\operatorname{im}}

\newcommand{\D}{\mathbb D}
\providecommand{\ord}{\operatorname{ord}}
\newcommand{\E}{\mathcal E}

\newcommand{\Y}{\mathscr Y}

\newcommand{\delpezzo}{\rm dP}
\newcommand{\F}{\mathbb F}
\DeclareMathOperator{\Bl}{Bl}
\DeclareMathOperator{\br}{br}
\DeclareMathOperator{\Jac}{Jac}
\renewcommand{\O}{\mathcal O}

\usetikzlibrary{shapes}
\date{\today}

\title[Compact moduli of K3 surfaces with a nonsymplectic automorphism]{Compactifications of moduli spaces of K3 surfaces with a higher-order nonsymplectic automorphism}

\author{Valery Alexeev}
\email{valery@math.uga.edu}
\address{Department of Mathematics, University of Georgia, Athens, GA, USA}

\author{Anand Deopurkar}
\email{anand.deopurkar@anu.edu.au}
\address{Mathematical Sciences Institute, The Australian National University, Canberra, ACT, Australia}

\author{Changho Han}
\email{changho\_han@korea.ac.kr}
\address{Department of Mathematics, Korea University, Seoul, Republic of Korea}

\begin{document}
\begin{abstract}
   We describe Baily-Borel, toroidal, and geometric --- using the KSBA stable pairs --- compactifications of some moduli spaces of K3 surfaces with a nonsymplectic automorphism of order $3$ and $4$ for which the fixed locus of the automorphism contains a curve of genus $\ge2$. For order $3$, we treat all the maximal-dimensional such families.  We show that the toroidal and the KSBA compactifications in these cases admit simple descriptions in terms of certain $ADE$ root lattices.
\end{abstract}
\maketitle

\setcounter{tocdepth}{1}
\tableofcontents

\section{Introduction} \label{sec:intro}

\subsection*{Background}
There has been significant recent progress on modular interpretations of Hodge-theoretic compactifications of moduli spaces of K3 surfaces \cite{ale.eng:23,ale.eng.han:24}.
Building on this, we completely describe the Baily-Borel, toroidal, and modular (= semi-toroidal) compactifications for some moduli spaces of K3 surfaces with a non-symplectic automorphism of degree 3 or 4.
The degree $N=2$ case was treated in \cite{ale.eng:22}.
If $N\not\in\{2,3,4,6\}$ then the moduli spaces of $ADE$ K3 surfaces are already compact by \cite{ale.eng.han:24}.

Let \(X\) be a smooth K3 surface and \(\sigma \colon X \to X\) an automorphism of degree \(N > 1\).
We say that \(\sigma\) is \emph{purely non-symplectic} if \(\sigma^{*}\) acts on \(H^{2,0}(X)\) by multiplication by \(\zeta_N=\exp(2\pi i/N)\).
Let \(L\) be the K3 lattice.
The automorphism \(\sigma\) induces an order \(N\) automorphism \(\rho\) of \(L\).
Let \(T_{\rho} \subset L\) be the primitive sub-lattice---the maximal sublattice such that the eigenvalues of \(\sigma\) on \(T_{\rho} \otimes \C\) are primitive \(N\)-th roots of unity.
%We can associate to \((X, \sigma)\) a period in a Hermitian symmetric domain \(\D_{\rho} \subset \P(T_{\rho} \otimes \C)\).
We can associate to \((X, \sigma)\) a period in a Hermitian symmetric domain \(\D_{\rho} \subset \P(T_{\rho,\C}^{\zeta_N})\), where \(T_{\rho,\C}^{\zeta_N}\) is the \(\zeta_N\)-eigenspace of \(\rho\) on \(T_{\rho} \otimes \C\).
The Torelli theorem identifies the moduli space of \((X,\sigma)\) with a dense subset \(F_{\rho}^{\rm sep}\)  of the quotient of \(\D_{\rho}\) by an arithmetic group \(\Gamma_{\rho}\) (see \Cref{sec:moduli_K3s}).

Assume that \(N > 2\).
Then \(\D_{\rho}\) is of type I (i.e. a complex hyperbolic ball).
It has several standard compactifications: the Baily-Borel, the toroidal (unique in this case), and semi-toroidal that interporate between the two.
The boundary of the Baily-Borel compactification consists of finitely many points (``cusps''), each corresponding to a \(\Gamma_{\rho}\)-orbit of a primitive \(\rho\)-invariant isotropic plane \(J \subset T_{\rho}\).
The semi-toroidal compactifications are determined by a semifan \(\mathfrak F\), which in this case is simply a primitive \(\rho\)-invariant sublattice \(\mathfrak F_J \subset J_{T_\rho}^{\perp}/J \colonequals (J^\perp \cap T_\rho)/J\) for every cusp \(J\). Here, $J^\perp$ is taken in $L$.

Assume the following condition.
\begin{equation}
  \label{eq:g2}
  \tag{$\exists g\ge2$}
  \text{The fixed locus }
  {\rm Fix}(\sigma)
  \text{ contains a smooth curve $C_g$ of genus } g\ge2.
\end{equation}
Then \(C_{g} \subset X\) is a semi-ample divisor.
Let \(X \to \overline X\) be the contraction defined by \(C_g\) and \(\overline C_{g} \subset \overline X\) the image of \(C_g \subset X\).
Then \(\overline X\) is a K3 surface with ADE singularities and for a small enough positive \(\epsilon\), the pair \((\overline X, \epsilon \overline C_g)\) is KSBA stable.
Let \(\overline F_{\rho}^{\rm KSBA}\) be (the normalization of) the KSBA compactification of such pairs.
By \cite[Theorem~3.26]{ale.eng.han:24}, this compactification is isomorphic to a semi-toroidal compactification of \(\D_{\rho}/\Gamma_{\rho}\).
When \(N=3\), this quotient is in fact rational by \cite[Theorem 1.1]{ma.oha.tak:15}. 

\subsection*{Results}
We explicitly identify the semi-fan describing the semi-toroidal compactification \(\overline F_{\rho}^{\rm KSBA}\) for certain moduli spaces.
To be more precise, recall that Artebani and Sarti classify all possible automorphisms \(\rho\) of the K3 lattice of order 3 that arise from a non-symplectic automorphism \(\sigma\) \cite{art.sar:08}.
The \(\rho\) is determined (up to conjugation) by the number \(n\) of isolated fixed points of \(\sigma\) and the number \(k\) of fixed curves of \(\sigma\).
Of all the possibilities, 11 satisfy \eqref{eq:g2}.
Of these, the maximal four (per genus of \(C_g\)) correspond to \((n,k) = (0,2), (0,1), (1,1),\) and \((2,1)\).
(The others arise from one of these four by specialization.)

\begin{table}
\begin{tabular}[ht]{llllll}
  \toprule
  $g$ &
  \((n,k)\) & \multicolumn{4}{l}{
              \parbox{24em}{Root lattice of \(J^{\perp}_{T_\rho}/J\) at the cusps.
              The saturation of parenthesized sublattice is the KSBA semi-fan \(\mathfrak F_J\) (if non-zero).}
              }\\
  \midrule
  $5$ & \((0,0)\) & \(E_8^{\oplus 2}\) \\
  \addlinespace
  $4$ & \((0,1)\) & \(E_6^{\oplus 2} \oplus \left(A_2^{\oplus 2}\right)\)
            & \(E_8 \oplus E_6 \oplus \left(A_2\right)\)
            & \(E_8^{\oplus 2}\) \\
\addlinespace
  $3$ & \((1,1)\) & \(E_6 \oplus \left(A_2^{\oplus 4}\right)\)
            & \(E_8\oplus \left(A_2^{\oplus 3}\right)\)
            &\(E_6^{\oplus 2} \oplus \left(A_2\right)\)
            & \(E_8 \oplus E_6\)\\
  \addlinespace
  $2$ & \((2,1)\) & \(\left(A_2^{\oplus 6}\right)\) 
            & \(E_6\oplus \left( A_2^{\oplus 3} \right)\)
            &\(E_6\oplus E_6\) 
            &\(E_8 \oplus \left(A_2^{\oplus 2}\right)\)\\
  \bottomrule\\
\end{tabular}
\caption{The semi-fans corresponding to the KSBA compactification of the moduli space of K3 surfaces with a non-symplectic automorphism of order 3 with \(n\) fixed points and \(k\) fixed curves.}
\label{tab:3}
\end{table}

\begin{theorem} \label{thm:mainthm}
  For each of the four maximal families of K3 surfaces with a non-symplectic automorphism of degree 3 satisfying \eqref{eq:g2}, the compactification \(\overline F_{\rho}^{\rm KSBA}\) is a semi-toroidal compactification of \(\D_{\rho}/\Gamma_{\rho}\) given by the explicit semifan \(\mathfrak{F}\) called the KSBA semi-fan described in \Cref{tab:3}.
\end{theorem}
In the main text, \Cref{thm:mainthm} is split into \Cref{thm:ksba02}, \Cref{thm:ksba01}, \Cref{thm:ksba11}, and \Cref{thm:ksba21} by the four cases.
The corresponding sections also describe the KSBA stable surfaces that appear on the boundary.

For \(N = 4\), we describe the moduli space of cyclic quadruple covers of \(\P^2\) branched at a smooth quartic (studied in \cite{kon:00, hac:04, art:09} from other points of view).
In this case, the Baily-Borel compactification has a unique cusp.
\begin{theorem}\label{thm:4}
  The compactification \(\overline F_{\rho}^{\rm KSBA}\) of the space of cyclic quadruple covers of \(\P^2\) branched at a quartic is the semi-toroidal compactification where the semi-fan \(\mathfrak F_J\) is the \(A_1^{\oplus 2}\) summand of \(J_{T_\rho}^{\perp}/J = D_4^{\oplus 2} \oplus A_1^{\oplus 2}\).
\end{theorem}
In the main text, \Cref{thm:4} is \Cref{thm:ksba4}. 

\subsection*{Strategy of the proof}
We know that there is a regular map
\[ \overline{\D_{\rho}/\Gamma_{\rho}}^{\rm tor}   \to \overline F_{\rho}^{\rm KSBA}\]
from the toroidal to the KSBA compactification \cite[Theorem~3.26]{ale.eng.han:24}.
For \(N=3\) or \(4\), let \(E\) be the unique elliptic curve with an action of \(\Z/N\Z\).
Over a cusp corresponding to \(J \subset T_{\rho}\), the toroidal compactification is (a finite quotient of) an abelian variety \(\mathcal{A}_J\), which is isogeneous to \(J_{T_\rho}^{\perp}/J \otimes_{\Z[\zeta_N]} E\).
The map to the KSBA compactification contracts precisely the translates of the abelian subvariety of \(\mathcal{A}_J\) isogeneous to \(\mathfrak F_J \otimes_{\Z[\zeta_N]} E \subset J_{T_\rho}^{\perp}/J \otimes_{\Z[\zeta_N]} E\).

Recall that \(\Hom(J^{\perp}/J, E) = J^{\perp}/J \otimes E\) is the period domain of Kulikov degenerations over the cusp \(J\) \cite{kon:85}.
In the \(\rho\)-equivariant setting, the period domain turns out to be \(J^{\perp}_{T_{\rho}}/J \otimes_{\Z[\zeta_N]} E\) (see \Cref{subsec:limiting_periods}).
On the other hand, the points of the KSBA compactification correspond to stable pairs.
So it suffices to identify which Kulikov degenerations yield the same stable pair; this leads to the computation of \(\mathfrak F_{J}\).
%To identify \(\mathfrak F_{J}\), it suffices to identify which Kulikov degenerations yield the same stable pair.
We do so in all the cases treated in \Cref{thm:mainthm} and \Cref{thm:4}.
The \(N = 4\) case is easy to do by ad-hoc methods.
For the \(N = 3\) cases, we devise a construction called the \emph{triple Tschirnhausen construction}, that takes an admissible triple cover as input and constructs a Kulikov model as output.
It turns out to give essentially all Kulikov degenerations.
We then run an MMP to construct the stable models.

\subsection*{About the lattices and the semi-fans}
The lattice \(J^{\perp}/J\) arises from the middle cohomology lattice of the Kulikov degenerations.
Paired with the action of \(\rho \in O(L)\), the lattice \(J^{\perp}_{T_\rho}/J\) is the intersection of \(J^{\perp}/J\) and \(T_\rho/J\) in \(L/J\).
The components of the Kulikov degenerations turn out to be blow-ups of del Pezzo surfaces.
The \(E_6, E_7, E_8\) lattices come from these del Pezzos; for example, if \(N=3\) then \(E_8\), resp. \(E_6\), comes from a degree \(1\), resp. degree \(3\), del Pezzo surface with a \(\Z/3\)-action.
The \(A_n\)-lattices, on the other hand, come from the exceptional loci of the blow-ups of aforementioned surfaces along some $\Z_N$-orbits of points. 
The passage from the Kulikov model to the stable model (roughly speaking) undoes the blow-ups, and therefore, the translates of the \(A_n\)-lattices are identified in the KSBA compactification.

\subsection*{Organization}
In \Cref{sec:moduli_K3s} we recall the Hodge-theoretic preliminaries for K3 surfaces with a purely non-symplectic automorphism, including the Baily-Borel, toroidal, and semi-toroidal compactifications.
In \Cref{sec:Kulikov_surfaces}, we recall Kulikov models, their modifications, and the associated periods.
In \Cref{sec:moduli_quartic_P2} we finish the (easier) \(N = 4\) case mentioned in \Cref{thm:4}.
From then on, we fix \(N = 3\).
In \Cref{sec:K3_aut3_Eisenstein}, we recall Eisenstein lattices, and using them we describe the Baily-Borel cusps of \(\D_{\rho}/\Gamma_{\rho}\).
In \Cref{sec:Kulikov_trigonal}, we describe the triple Tschirnhausen construction, and compute the cohomology lattices of the degenerate surfaces produced by this construction.
In the four subsequent sections, we carry out the proof strategy for the 4 moduli spaces mentioned in \Cref{thm:mainthm}.

\subsection*{Conventions}
We work over the complex numbers \(\C\).
We set \(\zeta_N = \exp(2\pi i/N)\) and let \(\mathcal E = \Z[\zeta_3]\) be the ring of Eisenstein integers.
For a Cohen--Macaulay variety \(X\), we use \(K_{X}\) to denote the canonical divisor or the dualizing sheaf, depending on the context.
\(\delpezzo_m\) 
refers to a degree \(m\) del Pezzo surface.
If \(J\) and \(P\) are sub-lattices of an ambient lattice \(L\), with \(J \subset P\), the notation \(J^{\perp}_P\) denotes the orthogonal complement of \(J\) in \(P\), that is \(J^{\perp}_P = J^{\perp} \cap P\).

\subsection*{Version on arXiv}
The version on arXiv gives additional details of computations, which are skipped in the journal version to save space.

\subsection*{Acknowledgements}
The first author was partially supported by the National Science Foundation under the award DMS-2201222.
The second author was partially supported by the Australian Research Council under the award DE180101360.
The third author was partially supported by Korea University Grants K2422881 and K2424631. This project had started during a visit of the first author to the Sydney Mathematical Research Institute and he is  grateful to the SMRI for the support and hospitality.

\section{Moduli of K3 surfaces with a non-symplectic automorphism} \label{sec:moduli_K3s}
\subsection{The period map} \label{subsec:period_map}
Let \(U={\rm II}_{1,1}\) be the lattice \(\Z^2\) with the bilinear form \(b(x,y) = xy\) and $E_8={\rm II}_{0,8}$ be the standard negative definite lattice.
Then the K3 lattice \(L\) is 
\[ L\simeq {\rm II}_{3,19} \simeq U^{\oplus 3} \oplus E_8^{\oplus 2}.\]
Let \(\D \subset \P (L \otimes \C)\) be the subset 
\[ \D = \left\{[x] \mid  (x \cdot  x ) = 0 \text{ and } ( x \cdot  \overline x ) > 0 \right\}.\]
Then \(\D\) is a complex analytic manifold of dimension \(20\).

Let \(X\) be a smooth K3 surface.
Choose a generator \(\omega_{X}\) of \(H^0(X, K_X) = H^{2,0}(X)\) and an isometry \(H^2(X, \Z) \to L\).
Then the image of \(\omega_{X}\) lies in \(\D\) and is called the period of \(X\).
It is well-defined up to the action of \(O(L)\).

Fix an isometry \(\rho \colon L \to L\) of order \(N > 1\).
Let \(\sigma\) be a purely non-symplectic aucotomorphism of \(X\) of order \(N\), i.e. $\sigma^*(\omega_X) = \zeta_N\omega_X$.
A \emph{\(\rho\)-marking} of \((X,\sigma)\) is an isometry \(\phi \colon H^2(X,\Z) \to L\) such that  \( \sigma^{*} = \phi^{-1}\circ \rho \circ \phi\).
The period  a \(\rho\)-marked \((X,\sigma)\) is the image of \(\omega_X\) in \(\D\) under the \(\rho\)-marking.

We say that \((X,\sigma)\) is \emph{\(\rho\)-markable} if it has a \(\rho\)-marking.
In this case, we say that \(\sigma\) has \emph{(Hodge) type \(\rho\)}.
The period point of a \(\rho\)-markable \((X,\sigma)\) is well-defined up to the action of the subgroup \(\Gamma_{\rho} \subset O(L)\) defined by
\[ \Gamma_{\rho} \colonequals \{\gamma \in O(L) \mid \gamma \circ \rho = \rho \circ \gamma\}.\]
Let \(L_\C^{\zeta_N} \subset L_\C \colonequals L \otimes \C\) be the eigenspace of \(\rho\) corresponding to the eigenvalue \(\zeta_N\).
Then, the period point of \(X\) lies in
\begin{equation}\label{eqn:drho}
  \D_{\rho} \colonequals \D \cap \P(L_\C^{\zeta_N}).
\end{equation}
Note that \(\Gamma_{\rho}\) preserves \(D_{\rho}\).

Let \(L_\C^{\rm prim} \subset L_\C\) be the subspace spanned by the \(\rho\)-eigenspaces whose eigenvalues are primitive \(n\)-th roots of unity.
We call
\( T_{\rho} \colonequals L \bigcap L_\C^{\rm prim}\)
the \emph{generic transcendental lattice} associated to \(\rho\), and \(S_{\rho} \colonequals T_\rho^{\perp}\) the \emph{generic Picard lattice} associated with \(\rho\).
The signature of \(T_\rho\) is \((2,\
rk (T_\rho)-2)\).
Let \(T_{\rho, \C}^{\zeta_N} = L_\C^{\zeta_N}\) be the \(\rho\)-eigenspace of \(T_{\rho, \C} \colonequals T_{\rho} \otimes \C\) with eigenvalue \(\zeta_N\).
We can then rewrite \(\D_{\rho}\) from \eqref{eqn:drho} as
\begin{equation}\label{eqn:drho2}
  \D_{\rho} = \{x \in \P(T^{\zeta_N}_{\rho, \C}) \mid ( x \cdot x ) = 0 \text{ and } ( x \cdot  \overline x ) > 0\}.
\end{equation}
Every root \(\delta \in L\), defines a hyperplane \(\delta^{\perp} \subset \P(L_{\C})\).
Define the \emph{discriminant} \(\Delta_{\rho} \subset \D_{\rho}\) as the union
\[ \Delta_{\rho} = \bigcup_{\delta} \delta^{\perp} \cap \D_{\rho},\]
where \(\delta\) ranges over all roots in \(\left(L^{\rho}\right)^{\perp}\).
Note that \(\Gamma_{\rho}\) preserves \(\Delta_{\rho}\).

We have the following \cite[Theorem~2.10]{ale.eng.han:24}.
%\va{it is Thm.2.9 on arXiv. Maybe we should update the version on arXiv}
%\ch{that's a great idea. Should we do that after submitting this paper?}
\begin{theorem} \label{thm:moduli_rho_markable_K3}
  Let \(F_{\rho}\) be the moduli space of \(\rho\)-markable K3 surfaces.
  The period map
  \[F_{\rho} \to \left(\D_{\rho} \setminus \Delta_{\rho}\right)/\Gamma_{\rho}.\]
  induces an isomorphism on the separated quotient \(F^{\rm sep}_{\rho}\) of \(F_{\rho}\).
\end{theorem}

\begin{remark} \label{rmk:dolgachev_kondo_difference}
    Note that \(F_\rho\) and \(F_\rho^{\rm sep}\) from \Cref{thm:moduli_rho_markable_K3} are constructed as moduli spaces of \(\rho\)-markable K3 surfaces as in \cite{ale.eng.han:24}.
    This approach differs from Dolgachev and Kond\=o's approach in \cite{dol.kon:07} via \cite{dol:96}, which is based on the moduli of $S$-polarized surfaces. Our construction, based on \cite[\S 2C]{ale.eng.han:24}, is a more direct approach.
\end{remark}

\subsection{The Baily--Borel compactification} \label{subsec:BB}
From now on, assume \(N \geq 3\); this simplifies the period domain.
For \(x \in T_{\rho, \C}^{\zeta_N}\), we have
\[ ( x \cdot  x ) = (  \rho x \cdot  \rho x ) = \zeta_N^2 (  x \cdot x ),\]
so \(( x \cdot  x ) = 0\).
Then, in \eqref{eqn:drho2}, the condition \(( x \cdot  x ) = 0\) is vacuous, and \(\D_{\rho}\) is a Hermitian symmetric domain of type I.

The compact dual \(\D_{\rho}^c\) of \(\D_{\rho}\) is simply the projective space \(\P\left(T_{\rho,\C}^{\zeta_N}\right)\).
A \emph{rational boundary component} of \(\D_{\rho}\) is a subset of \(\D_{\rho}^c\) of the form \(\P(J_{\C}) \cap \D_{\rho}^c\), where \(J \subset T_{\rho}\) is a rank 2 primitive isotropic sub-lattice of \(T_{\rho}\).
The \emph{rational closure} \(\D_{\rho}^{\rm BB}\) of \(\D_{\rho}\) is the union of \(\D_{\rho}\) with all of its rational boundary components, topologised by the horoball topology along the boundary.
The quotient \(\overline{\D_{\rho}/\Gamma_{\rho}}^{\rm BB} \colonequals \D_{\rho}^{\rm BB}/\Gamma_{\rho}\) is the \emph{Baily--Borel compactification} of \(\D_{\rho}/\Gamma_{\rho}\).
It has the structure of a projective variety \cite{bai.bor:66} and it contains \(\D_{\rho}/\Gamma_{\rho}\) as dense open subvariety.

Let \(J \subset T_{\rho}\) be a rank 2 primitive isotropic sub-lattice such that \(\P (J_{\C}) \cap \P (T_{\rho, \C}^{\zeta})\) is non-empty.
We observe that \(J\) is automatically \(\rho\)-invariant and \(\P (J_{\C}) \cap \P (T_{\rho, \C}^{\zeta})\) is a point.
Indeed, let \(v \in (J_{\C}) \cap T_{\rho, \C}^{\zeta_N}\) be non-zero.
Then the conjugate \(\overline v\) lies in \((J_{\C}) \cap T_{\rho, \C}^{\zeta_N^{-1}}\) and \(J_{\C}\) is spanned by \(v\) and \(\overline v\).
As a result, \(J_{\C}\), and hence \(J\), is \(\rho\)-invariant.
Furthermore, \(\P (J_{\C}) \cap P(T_{\rho}^{\zeta_N})\) is the point \([v]\).

The boundary \((\overline{\D_{\rho}/\Gamma_{\rho}}^{\rm BB}) \setminus (\D_{\rho}/\Gamma_{\rho})\) is a collection of (finitely many) points, called \emph{cusps}, which are in bijection with the \(\Gamma_{\rho}\)-orbits of \(\rho\)-invariant isotropic primitive rank 2 sub-lattices \(J \subset T_{\rho}\).

\subsection{Semitoroidal compactifications} \label{subsec:semitoroidal}
Recall that we specialise to \(N \geq 3\), so that we have a type I domain where the constructions are easier.
Here, we recall the construction of (semi)toroidal compactifications of \(\D_{\rho}/\Gamma_{\rho}\) from \cite[\S 3E]{ale.eng.han:24}.
See also \cite{kon:93,loo:03,ash.mum.rap.tai:10,mok:12} for more details.

\medskip

We first summarize preliminary definitions on lattices and subgroups of orthogonal groups.
Let \(O_\rho(T_\rho)\) be the subgroup of \(O(T_\rho)\) consisting of isometries commuting with~\(\rho\).
Let \(\overline{\Gamma}_\rho \subset O_\rho(T_\rho)\) be the image of \(\Gamma_\rho \subset {\rm Stab}_{T_\rho}(O(L))\) under the restriction map \({\rm Stab}_{T_\rho}(O(L)) \to O(T_\rho)\). 
If \(N=3\), then \(\overline{\Gamma}_\rho = O_\rho(T_\rho)\) by the proof of \cite[Theorem 3.6]{ma.oha.tak:15}.
Consider a primitive \(\rho\)-invariant isotropic rank \(2\) sublattice \(J \subset T_\rho\) such that \(J_\C^{\zeta_N} \neq \emptyset\); in this case,  \(J_{\C}^{\zeta_N} \simeq J \otimes_{\Z[\zeta_N]} \C\) where \(\zeta_N\) acts on \(J\) via \(\rho\). For every such \(J\), there is an exact sequence of groups
\[
    0 \to U_J \to {\rm Stab}_J(\overline{\Gamma}_\rho) \to \Gamma_J \to 0,
\]
where \(U_J\) is the unipotent radical of \({\rm Stab}_J(\overline{\Gamma}_\rho)\) (i.e., \(U_J\) acts trivially on \(J\), \(J^\perp_{T_\rho}/J\), and \(T_\rho/J^\perp_{T_\rho}\)), and \(\Gamma_J \subset \GL(J) \times O(J^\perp_{T_\rho}/J)\).
In fact, \(\Gamma_J\) is finite because \(\rho\) acts on \(J\) as a rotation and elements of \(O(J)\) commuting with a fixed rotation forms a finite group.  
%Note that the action of \(U_J\) on \(T_\rho\) induces the action of a \(\rho\)-invariant translation subgroup \(\overline{U}_J \subset J^\perp/J\) of finite index on \((J^\perp_{T_\rho})_\C^{\zeta_N}/J_\C^{\zeta_N} \simeq J^\perp_{T_\rho}/J \otimes_{\Z[\zeta_N]} \C\).

\medskip

The construction of semitoroidal compactifications \(\overline{\D_\rho/\Gamma_\rho}^{\mathfrak F}\) of \(\D_\rho/\Gamma_\rho\) from \cite[\S 3E]{ale.eng.han:24} (which is the \(\rho\)-equivariant analog of construction in \cite[\S 6]{loo:03}) depends on the data of a \emph{\(\Gamma_{\rho}\)-admissible semi-fan} \(\mathfrak F\), which is a \(\Gamma_\rho\)-invariant collection \(\{\mathfrak F_J\}\) of a primitive \(\rho\)-invariant sub-lattice
\[ \mathfrak F_J \subset J^\perp_{T_\rho} / J\]
for every \(\rho\)-invariant isotropic primitive rank 2 sub-lattice \(J \subset T_{\rho}\) such that \(J_{\C}^{\zeta_N} \neq 0\).
Then the {\it semitoroidal compactification} \(\overline{\D_\rho/\Gamma_\rho}^{\mathfrak F}\) is a normal projective variety such that for every such \(J\), the associated cusp of \(\overline{\D_\rho/\Gamma_\rho}^{\rm BB}\) is replaced by a finite quotient \(\widehat{\mathcal{A}}_J^{\mathfrak{F}_J}/\Gamma_J\) of an abelian variety \(\widehat{\mathcal{A}}_J^{\mathfrak{F}_J}\) isogeneous to a product of elliptic curves of the same \(j\)-invariant.
In fact, there is an isogeny \(\mathcal{A}_J^{\mathfrak{F}_J} \to \widehat{\mathcal{A}}_J^{\mathfrak{F}_J}\), where \(\mathcal{A}_J^{\mathfrak F_J} \colonequals ((J^\perp_{T_\rho}/J)/\mathfrak{F}_J) \otimes_{\Z[\zeta_N]} E\) is defined by \(\zeta_N\)-action on \(J^\perp_{T_\rho}/J\) via \(\rho\) and an elliptic curve \(E = \C/\Z[\zeta_N]\) with \({\rm Aut}(E,0) \supset \Z/N\Z\).

The semi-toroidal compactifications interpolate between the Baily--Borel and the toroidal compactification (unique in this case, which is a ball quotient).
The Baily--Borel compactification \(\overline{\D_\rho/\Gamma_\rho}^{\rm BB}\) corresponds to the maximal semifan, namely \(\mathfrak F_J = J^{\perp}_{T_{\rho}}/J\).
The toroidal compactification \(\overline{\D_\rho/\Gamma_\rho}^{\rm tor}\) corresponds to the minimal semifan, namely \(\mathfrak F_J = 0\).
In this case, the boundary components are of codimension 1 and finite quotients of \(\mathcal{A}_J \colonequals \mathcal{A}_J^0 = (J^\perp_{T_\rho}/J) \otimes_{\Z[\zeta_N]} E\).
For a general semifan \(\mathfrak F\), we have birational morphisms
\begin{equation} \label{eq:mor_tor->KSBA}
    \overline{\D_\rho/\Gamma_\rho}^{\rm tor} \to \overline{\D_\rho/\Gamma_\rho}^{\mathfrak F} \to \overline{\D_\rho/\Gamma_\rho}^{\rm BB}.
\end{equation}
These maps are the identity on \(\D_\rho/\Gamma_\rho\).
Above the Baily--Borel cusp corresponding to \(J \subset T_{\rho}\), the first map is induced by the quotient \(\mathcal{A}_J \to \mathcal{A}_J^{\mathfrak{F}_J}\) of the translates of \(\mathfrak{F}_J \otimes_{\Z[\zeta_N]} E\) in \(\mathcal{A}_J\).

\subsection{KSBA stable pair compactifications} \label{subsec:KSBA_compactification}
We refer the reader to \cite{kol:23} for the definition
of semi-log-canonical (slc for short) singularities and the existence of the KSBA compactifications of moduli
spaces via (KSBA-)stable pairs.

In our context, a \emph{(KSBA-)stable pair} \((X,(w+\epsilon)D)\) consists of a rational number \(w \in \Q \cap [0,1)\), a seminormal surface \(X\), and an effective \(\Q\)-divisor \(D\) on \(X\) such that the pair \((X,D)\) has slc singularities, \(K_X+wD \sim_\Q 0\), and \(K_X+(w + \epsilon)D \sim_\Q \epsilon D\) is \(\Q\)-Cartier and ample for every \(0 < \epsilon \ll 1\) bounded above in terms of \(D^2\). 
When \(D^2\) is fixed, then there is a projective coarse moduli space for such pairs. For full details, see \cite{ale.eng.tho:23} and \cite{kol.xu:20}.

Assuming that every pair \((X,\sigma) \in F_\rho\) satisfies \eqref{eq:g2}, there is an associated stable pair \((\overline{X},\epsilon\overline{C}_g)\).
Here \(\overline X\) is obtained from \(X\) by contracting the \(C_g\)-trivial curves and \(\overline C_g \subset \overline X\) is the image of \(C_g \subset X\).
As a result, \(F_\rho^{\rm sep}\) is a moduli space of ADE K3 surfaces \((\overline{X},\epsilon\overline{C}_g)\) equipped with an induced automorphism \(\overline{\sigma}\) of order \(N\), where the divisorial part of the fixed locus of \(\overline{\sigma}\) is \(\overline{C}_g\).

Define \(\overline{F}_\rho^{\rm KSBA}\) as the normalization of the closure of \(F_\rho^{\rm sep}\) in the space of KSBA stable pairs; it parameterizes stable pairs \((\overline{X},\epsilon\overline{R})\) with trivial dualizing sheaves. 
By \cite[Proposition 3.27]{ale.eng.han:24}, any \((\overline{X},\epsilon\overline{R}) \in \overline{F}_\rho^{\rm KSBA}\) comes equipped with an automorphism \(\overline{\sigma} \from \overline{X} \to \overline{X}\) with \(\overline{R} = {\rm Fix}(\overline{\sigma})\).
Moreover, by \cite[Theorem 3.26]{ale.eng.han:24}, \(\overline{F}_\rho^{\rm KSBA} = \overline{\D_\rho/\Gamma_\rho}^{\mathfrak{F}}\) for a particular semitoroidal compactification for a semifan \(\mathfrak{F} = \{\mathfrak{F}_J\}\); call \(\mathfrak{F}\) the \emph{KSBA semifan}. 
One of the main goals of this paper is to compute \(\mathfrak{F}\) for various examples of \(\rho\).

In some cases, an explicit classifications of members of \(\overline{F}_\rho^{\rm KSBA}\) can be obtained by classifying KSBA degenerations of the \(\overline{\sigma}\)-quotients \((Y,\frac{N-1+\epsilon}{N}B)\) of \((\overline{X},\epsilon\overline{C}_g) \in F_\rho^{\rm sep}\), where \(B\) is the branch divisor of the quotient \(\overline{X} \to Y \colonequals \overline{X}/\overline{\sigma}\).
Specifically, \(Y\) is a (possibly singular) rational surface with \((N-1)B \sim -NK_Y\), i.e., \((Y,(\frac{N-1+\epsilon}{N})B)\) is a stable pair.
Let \(\Y_\rho\) be the normalization of the KSBA compactification of the moduli space of such quotient pairs \((Y,\frac{N-1+\epsilon}{N}B)\). Then, \cite[Theorem 4.2, Corollary 4.3]{ale.eng.han:24} says the following.
\begin{theorem} \label{thm:stable_quotient_pair}
    There is an isomorphism of coarse moduli spaces
    \begin{align*}
        \overline{F}_\rho^{\rm KSBA} &\to \Y_\rho\\
        (\overline{X},\epsilon \overline{R}) &\mapsto \left(Y, \frac{N-1+\epsilon}{N}B\right),
    \end{align*}
    where \(Y = \overline{X}/\overline{\sigma}\) and \(B\) is the branch divisor of \(\overline{X} \to Y\).
\end{theorem}

\section{Degenerations and Kulikov surfaces}
\label{sec:Kulikov_surfaces}
\subsection{Kulikov models}
\label{subsec:Kulikov_models}
We summarize properties of one-parameter degenerations of K3 surfaces; see \cite{fri.sca:86} for more details.

\begin{definition} \label{def:Kulikov_model}
    Let \(X^* \to C^*\) be a family of smooth complex K3 surfaces over a punctured curve or disk \(C^* = C \setminus 0\). 
    A \emph{Kulikov}, or \emph{Kulikov--Persson--Pinkham}, \emph{model} is a proper extension \(X \to C\) such that \(K_X \sim_C 0\) and \(X \to C\) is semistable, i.e., \(X\) is nonsingular and the central fiber \(X_0\) is reduced with normal crossing singularities. 
    The central fiber \(X_0\) is called a \emph{Kulikov surface}.
\end{definition}

By \cite{kul:77, per.pin:81}, A Kulikov model of \(X^* \to C^*\) exists after a finite base-change \(C' \to C\).

There are three types of Kulikov models depending on the dual complex of the Kulikov surface \(X_0\).
Let \(V_i\) be the irreducible components of \(X_0\) and set \(D_{i,j} = V_i \cap V_{j}\).
\begin{enumerate}[label={(\Roman*)}]
    \item \(X_0\) is smooth, i.e., it is a K3 surface.
    \item The dual complex of \(X_0\) is an interval of length \(m\) ordered by vertices \(\{0,\dotsc,m\}\), \(D_{0,1} \simeq \dotsb \simeq D_{m-1,m} \simeq E\) are elliptic curves, \(V_0\) and \(V_m\) are rational, and \(V_i \to E\) is generically ruled for every \(0<i<m\). 
    The double locus \(D_i \colonequals \sum_j D_{i,j}\) is an anticanonical divisor on each component \(V_i\) of \(X_0\).
    \item The dual complex of \(X_0\) is a triangulation of the sphere \(S^2\). 
    The double locus \(D_i\) is an anticanonical divisor on each component \(V_i\) of \(X_0\) and is a wheel of rational curves.
\end{enumerate}

If \((X^* \to C^*, \sigma \in {\rm Aut}_{C^*}(X^*))\) is a family of \(\rho\)-markable K3 surfaces and \(N = \ord(\rho) \ge 3\), then the associated Kulikov surface \(X_0\) is of Type I or II by \cite[Corollary 3.17]{ale.eng.han:24}.
If \(N \neq 3,4,6\), then in fact \(X_0\) is always of Type I.

A snc surface of the form \(X_0 = V_0 \cup V_1\) with trivial \(K_{X_0}\) is a Kulikov surface if and only if \(N_{D_0/V_0} \otimes N_{D_1/V_1} \simeq \mathcal{O}_E\) by \cite[Proposition 3.7]{fri.mir:83}; this condition is called \emph{\(d\)-semistability}.
So a Type II Kulikov surface of the form \(X_0 = V_0 \cup V_1\) satisfies \(D_0^2+D_1^2=0\).

Given a Kulikov model \(X \to C\), we have an  \emph{extended period map}
\[
    f \colon C \to \overline{\D_{\rho}/\Gamma_{\rho}}^{\rm BB}.
\]
If \(f(0) \in \D_\rho/\Gamma_\rho\), then \(X \to C\) is of Type I; otherwise, \(X \to C\) is of Type II or III.

Given a Kulikov model \(X \to C\), the \emph{Picard--Lefschetz transform} of \(X^* \to C^*\) is the unipotent monodromy \(T \colon H^2(X_t,\Z) \to H^2(X_t,\Z)\). 
It is of the form \(T = \exp(\mathcal N)\) with
\[
    \mathcal N(x) = (x \cdot \lambda)\delta - (x \cdot \delta)\lambda,
\]
where \(\delta \in H^2(X_t,\Z)\) is primitive isotropic and \(\lambda \in \delta^\perp/\delta\) is a vector called the \emph{monodromy invariant}, which satisfies
\[
    \lambda^2 = \#\{\text{triple points of } X_0\}.
\]
In the \(\rho\)-markable case, there are no Type III Kulikov surfaces so \(\lambda^2=0\). 
When \(\lambda = 0\), then \(X_0\) is of Type I. 
Otherwise, \(X_0\) is of Type II and the length \(m\) of the dual complex of \(X_0\) coincides with the imprimitivity of \(\lambda \in \delta^\perp/\delta\).
Denote \(J \colonequals (\Z\delta \oplus \Z\lambda)^{\rm sat}\) to be the primitive isotropic plane corresponding to \(X_0\). 
In the \(\rho\)-markable case, \((\sigma_t^*)^{-1}\mathcal N\sigma_t^* = \mathcal N\) as \(T\) commutes with the automorphism \(\sigma_t \in {\rm Aut}(X_t)\). 
By expanding the previous equality, it is easy to see that \(\Z\delta \oplus \Z\lambda\) is \(\sigma_t^*\)-invariant, hence \(J\) is also \(\sigma_t^*\)-invariant.
If \(\phi_t \colon H^2(X_t,\Z) \to L\) is a \(\rho\)-marking, then the cusp \(f(0) \in \overline{\D_{\rho}/\Gamma_{\rho}}^{\rm BB}\)  corresponds to \(\phi_t(J) \subset T_\rho\).
%\chinline{For above, the computation is this: \(\sigma_t^*\) changes \(\delta,\lambda\) involved in \(\mathcal N\) into \(\sigma_t^*\delta, \sigma_t^*\lambda\)}

Two Kulikov models \(X \to C\) of the same family \(X^* \to C^*\) differ by a sequence of flops in curves 
\(F \subset X_0\). 
If the central fiber \(X_0\) is of Type I or II, then the flops are of the following types:
\begin{enumerate}[start=0, label={(\text{M}\arabic*)}]
    \item \(F \subset V_i\) and \(F \cap D_i = \emptyset\); so \(F\) is an interior \((-2)\)-curve.  This flop is a nontrivial birational map \(X \dashrightarrow X'\) over \(C\) but \(X_0'=X_0\) are canonically identified.
    \item \(F \subset V_i\) with \(F^2=-1\) and \(F \cap D_i = p \in D_{i,j}\).  This flop contracts \(F\) on \(V_i\) to \(p\) and blows up \(p \in V_j\) to create a \((-1)\)-curve \(F' \subset V_j\).
\end{enumerate}

\subsection{Nef, divisor, and stable models}
\label{subsec:nef_div_stable}

We define various types of  models of \(X^* \to C^*\) based on \cite[\S 3B]{ale.eng:23}:

\begin{definition} \label{def:nef_model}
    Let \(L^*\) be a line bundle on \(X^*\), relatively nef and big over \(C^*\). 
    A relatively nef extension \(L\) to a Kulikov model \(X \to C\) is called a \emph{nef model}.
\end{definition}

\begin{definition} \label{def:divisor_model}
    Let \(R^* \subset X^*\) be the vanishing locus of a section of \(L^*\) as above, containing no vertical components. 
    A \emph{divisor model} is an extension \(R \subset X\) to a relatively nef divisor \(R \subset |L|\) for which \(R_0\) contains no triple point, double curve, or component of \(X_0\). We call the central fiber \((X_0,R_0)\) a \emph{divisor limit}.
\end{definition}

A divisor model is not necessarily unique.
Two divisor models are related by a sequence of M0 and M1 modifications along curves disjoint from \(R\).

\begin{definition} \label{def:stable_model}
      The (KSBA-)stable model \((\overline{X},\epsilon\overline{R})\) is \(\operatorname{Proj}_C \bigoplus_{n \ge 0} \pi_*\mathcal{O} (nR)\) for some divisor model \(\pi \colon (X,R) \to C\). 
    It is unique, depending only on the family \((\overline{X}^*,\overline{R}^*) \to C^*\), and stable under base change. 
    We call \((\overline{X}_0,\epsilon \overline{R}_0)\) the \emph{stable limit}.
\end{definition}

Given a \(\rho\)-markable K3 family \((X^* \to C^*,\sigma)\) where every fiber satisfies \eqref{eq:g2}, define \(R^* \subset X^*\) to be the irreducible component of \((X^*)^\sigma\) which is a family of smooth curves of genus \(g \ge 2\) over \(C^*\).
Then, up to a finite base change on \(C\), there is a Kulikov model of \(X^* \to C^*\).
Then \(R^*\) extends to \(R \subset X\) by \cite[Theorem 3.26]{ale.eng.han:24}, but \(R_0\) (the flat limit of \(R_t\)) may not be nef.
In that case, a divisor model can be achieved by taking a sequence of \text{M}0 and \text{M}1 modifications on \(X \to C\). 

\subsection{Periods of Type II Kulikov surfaces}
\label{subsec:limiting_periods}
We define periods of type II Kulikov surfaces equipped with an automorphism following \cite{ale.eng:23, kon:85}.

Let us first recall the theory without the automorphism.
Consider a type II Kulikov surface \(X_0 = V_0 \cup \cdots \cup V_m\) with \(D_{ij} = V_i \cap V_{j}\) in \(V_{i}\); note that \(D_{ij} \neq \emptyset\) whenever \(|i-j|=1\), and \(D_{ij} \simeq D_{ji}\).
Define the \emph{lattice of numerically Cartier divisors}
\[ \widetilde \Lambda =  \ker \left(\bigoplus_i H^2(V_i,\Z) \to \bigoplus_i H^2(D_{i},\Z)\right).\]
In the map above, the summand \(H^2(V_i,\Z) = \operatorname{Pic}V_i\) maps to \(H^2(D_{i,i+1},\Z)\) and \(H^2(D_{i,i-1},\Z)\) by restriction.
Let \(\xi_i \colonequals \sum_j (D_{ij} - D_{ji}) \in \oplus_\ell \operatorname{Pic} V_\ell\) for every \(0 \le i \le n\).
Set \(\Xi = \sum \xi_i \in \widetilde \Lambda\).
Define the \emph{reduced lattice of numerically Cartier divisors}
\[ \Lambda = \widetilde \Lambda / \Xi.\]
We have \(\rk \widetilde \Lambda = 19\)  and \(\rk \Lambda = 18\).

Let \(E\) be the elliptic curve isomorphic to any of the double curves \(D_{i,i+1}\).
The \emph{period} of \(X_0\) is the homomorphism \(\psi \colon \Lambda \to E = \C/J\) is defined in \cite[Construction 4.3]{ale.eng:23} and \cite[\S 2]{kon:85}.
Specifically, when \(X_0 = V_0 \cup V_1\) consists of two irreducible components, then \(\psi(\alpha) = \alpha_{0}|_{D_{01}} - \alpha_{1}|_{D_{01}}\in {\rm Pic}(D_{01}) \simeq E\) for every \(\alpha \in \Lambda\).

\medskip

Suppose \(X_0\) is the central fiber of a Kulikov model \(X \to C\) corresponding to a primitive isotropic 2-plane \(J \subset H^2(X_t,\Z)\).
Then by \cite[Proposition 3.20]{ale.eng:23} we have an isometry of unimodular lattices
\begin{equation}\label{eqn:JperpJ}
  \Lambda \simeq J^{\perp}/J.
\end{equation}
By \cite[Theorem~4.16]{ale.eng:23} (see also \cite{kon:85}), the period \(\psi\) of \(X_0\) lies in the period domain \(J^\perp/J \otimes E = \Hom(J^\perp/J,E)\). 
Observe that \text{M}0 and \text{M}1 modifications do not change \(\psi\).

Fix an isometry \(L \simeq H^2(X_t,\Z)\), a primitive sublattice \(S \subset L\), and set \(T = S^{\perp}\).
Suppose that \({\rm Pic}(X^*/\Delta^*)\) contains a primitive subgroup whose restriction to every fiber \(X_t\) is \(S\).
Let \(X \to C\) be a Kulikov model with central fiber \(X_{0}\) corresponding to a primitive isotropic plane \(J \subset T\).
Let \(S^{\rm sat} \subset J^{\perp}/J\) be the saturation of the image of \(S\).
Then \(J^\perp_{T}/J\) and \(S^{\rm sat}\) are two saturated mutually orthogonal sublattices of \(J^\perp/J\).
Using \((J^\perp/J)/S^{\rm sat} = (J^\perp_{T}/J)^*\), we have the following (split) short exact sequence
\[
	0 \to \Hom((J^\perp_{T}/J)^*,E) \to \Hom((J^\perp/J)/S,E) \to \Hom(S^{\rm sat}/S,E) \to 0.
\]
The period \(\psi \in \Hom(J^{\perp}/J, E)\) of \(X_0\) vanishes on \(S\), and hence lies in \(\Hom((J^{\perp}/J)/S, E)\).
By \cite[\S 4C]{ale.eng:23}, its image \(\chi\) in \(\Hom(S^{\rm sat}/S, E)\) depends only on \(S\) and \(J \subset T\).
In other words, the set of all periods of \(S\)-quasipolarized Type II surfaces over the cusp \(J\) is the translate of \(\Hom((J^\perp_{T}/J)^*,E) = (J^\perp_{T}/J)\otimes E\) in \(\Hom((J^\perp/J)/S,E)\) over \(\chi\).

Fix an automorphism \(\rho\) of \(L\) of finite order.
Suppose that the family \(X^{*} \to C^{*}\) is equipped with an automorphism \(\sigma\) and is \(\rho\)-markable.
Let \(T = T_{\rho}\) and \(S = S_{\rho}\) as defined in \Cref{subsec:period_map}.
By \(\sigma^{*}\)-equivariance of the monodromy, \(J\) lies in \(T\) and is \(\rho\)-invariant.
From \eqref{eqn:JperpJ}, we get the isometry
\begin{equation}\label{eqn:JperpJrho}
  \Lambda^{\rm prim} \cong J^{\perp}_{T}/J.
\end{equation}
The double curve \(E = \C/J\) of \(X_0\) also admits a \(\rho\)-action, generating \(\Z/N\Z \subset {\rm Aut}(E,0)\).
The period \(\psi\) of \(X_0\) is \(\rho\)-invariant in \(\Hom((J^\perp/J)/S,E)\), and so is the associated character~\(\chi\).
So the period domain of Kulikov limits of \(\rho\)-markable K3 surfaces associated with fixed \(J \subset T\) is a \(\chi\)-translate of 
\[
    \Hom_{\Z[\zeta_N]}((J^\perp_{T_\rho}/J)^*,E) = J^\perp_{T_\rho}/J \otimes_{\Z[\zeta_N]} E \simeq \mathcal{A}_J.
\]
From now on, we identify this period domain with \(\mathcal{A}_J\).

Using the \(\rho\)-equivariant analog of \cite[Theorem 4.16]{ale.eng:23} and the construction of the toroidal compactification \(\overline{\D_\rho/\Gamma_\rho}^{\rm tor}\) (see \cite[\S 5B]{ale.eng:23} and \cite[\S 3E]{ale.eng.han:24}), the period map \(f^* \from C^* \to \D_\rho/\Gamma_\rho\) of a \(\rho\)-markable family \(X^* \to C^*\) as above extends to an extended period map
\begin{equation}\label{eqn:extendedperiod}
  f \from C \to \overline{\D_\rho/\Gamma_\rho}^{\rm tor}.
\end{equation}
Recall from \Cref{subsec:semitoroidal} that the boundary component of \(\overline{\D_\rho/\Gamma_\rho}^{\rm tor}\) above the cusp \(J\) is a finite quotient of \(\mathcal{A}_J\).
Then, \(\rho\)-equivariant version of \cite[Theorem 4.16]{ale.eng:23} implies that \(f(0)\) in the boundary component is the image of \(\psi \in \mathcal{A}_J\).

\section{Moduli space of quadruple covers of \texorpdfstring{\(\P^2\)}{P2}} \label{sec:moduli_quartic_P2}
In this section, we identify the KSBA compactification of a moduli space of K3 surfaces with a non-symplectic automorphism of order $4$.

Let \(X \to \P^2\) be a cyclic cover of degree \(4\) branched along a general quartic \(B \subset \P^2\).
Then \(X\) is a smooth K3 surface.
The cyclic cover construction equips it with a non-symplectic automorphism $\sigma \colon X \to X$ of order $4$.
We choose it so that $\sigma^{*} \omega_X = \zeta_4 \omega_{X}$.

Recall that $L$ is the K3 lattice.
Let \(\rho \in O(L)\) be the automorphism of order 4 arising from a marking of \((X,\sigma)\).
Let \(\tau = \sigma^2\), so that \(\tau^{*}\) is identified with \(\rho^2 \in O(L)\) under this marking.
Recall from \Cref{subsec:period_map} that \(T_{\rho} = L \cap L_{\C}^{\rm prim}\) and \(S_{\rho} = T_{\rho}^{\perp}\).
In this case, we have
\[T_{\rho} = \{x \in L \mid \tau^{*} x = -x\}. \]
Kond\=o describes \(S_\rho\) and \(T_{\rho}\) in \cite[\S~2]{kon:00} (the notation there is \(L_{+}\) and \(L_{-}\)) as
\[
    T_{\rho} \simeq U(2) \oplus U(2) \oplus D_8 \oplus A_1^{\oplus 2}, \quad S_\rho\simeq \left<2\right> \oplus A_1^{\oplus 7}.
\]

We now describe the automorphism $\rho \colon T_{\rho} \to T_{\rho}$ induced by $\sigma$.
To do so, it is helpful to write a different direct sum decomposition for \(T_{\rho}\).
\begin{lemma}
  We have \(T_{\rho} \simeq U \oplus U(2) \oplus D_4^{\oplus 2} \oplus A_1^{\oplus 2}\).
\end{lemma}
\begin{proof}
  The lattices \(U(2) \oplus D_8\) and \(U \oplus D_4^{\oplus 2}\) are even, 2-elementary, indefinite, and have the same basic invariants 
  $(t_+,t_-,a,\delta)=(1,9,4,0)$.
  So by \cite[Theorem 3.6.2]{nik:79}, %\ch{This theorem is more specific than citing a section of \cite{nik:79}.} %\cite[\S~1.13]{nik:79}, 
  they are isomorphic.
\end{proof}

Choose a basis \(\langle e, f\rangle\) of \(U\) so that \(e^2 = f^2 = 0\) and \(ef = 1\).
Likewise, choose a basis \(\langle e', f' \rangle\) of \(U(2)\) such that \({e'}^2 = {f'}^2=0\) and \(e' f' = 2\).
Consider \(D_4\) as an index 2 sublattice of \(\Z^4\) in the standard way, namely consisting of vectors whose coordinates add up to an even number.
Let \(e_1, \dots, e_4\) be the standard basis of \(\Z^4\).
\begin{proposition}
  Set \(T = T_{\rho} = U \oplus U(2) \oplus D_4^{\oplus 2} \oplus A_1^{\oplus 2}\).
  Up to conjugation by an element of $O(T)$, the order 4 automorphism $\rho = \sigma^*$ acts on $T$ as follows.
  \begin{enumerate}
  \item On $U\oplus U(2)=\langle e, f\rangle \oplus \langle e', f'\rangle$ by
    \begin{displaymath}
      e\mapsto -e+e',\ e'\mapsto -2e+e', \quad f\mapsto f+f',\ f'\mapsto -2f-f'.
    \end{displaymath}
  \item On each $D_4 \subset \Z^4=\langle e_1,\dotsc, e_4\rangle$, by $e_{2i-1}\mapsto e_{2i} \mapsto -e_{2i-1}$ for $i=1,2$.
   \item On $A_1\oplus A_1 = \langle h_1, h_2\rangle$ by $h_1\mapsto h_2\mapsto -h_1$.
  \end{enumerate}
\end{proposition}
\begin{proof}
  For a 2-elementary lattice \(M\), let \(A_M = M^{*}/M\) be the discriminant group, equipped with the discriminant form \(q_M(x) = x^2 \mod 2\Z \in \frac{1}{2}\Z/2\Z\).
  Set 
  \[N=\{x\in A_T\mid q(x)\in \Z/2\Z\}.\]
  Let $\phi \colon T \to T$ be the action described in the statement.
  We see that $\phi|_N$ is the identity.
  Let $S = S_{\rho}$ be the generic Neron-Severi lattice.
  By \cite[Lemma 2.2(ii)]{kon:00}, we see that $\phi = \sigma^{*}$ on $A_T = A_S$.
  Therefore, the pair $(\sigma^{*}|_S, \phi)$ lifts to an isometry of $L$; call it $g$.

  The isometry $g \colon L \to L$ is a purely non-symplectic automorphism with invariant subspace $L^g$ of rank 1, isomorphic to the lattice $\langle 4\rangle$.
  By \cite{bra.hof:23} and the accompanying database, up to conjugation there exists a unique isometry of $L$ with these properties.
  So $g$ and $\sigma^*$ are conjugate.
\end{proof}
\begin{corollary}\label{cor:4j}
  Let $J=\langle e, e'\rangle \subset U\oplus U(2) \subset T$.
  Then $J$ is $\rho$-invariant, isotropic, and satisfies $J_{T}^\perp/J \simeq D_4^2\oplus A_1^{\oplus 2}$.
\end{corollary}
\begin{proof}
  Indeed, $J_{T}^\perp = J \oplus D_4^{\oplus 2}\oplus A_1^{\oplus 2}$, so $J_{T}^\perp/J \simeq
  D_4^2\oplus A_1^{\oplus 2}$. 
\end{proof}

It turns out that this $J$ defines the unique cusp of the Baily--Borel compactification.
\begin{theorem}
  \label{thm:bbcusp4}
  In the setup above, the Baily--Borel compactification \(\overline{\D_{\rho}/\Gamma_{\rho}}^{\rm BB}\) has a unique cusp, corresponding to the isotropic plane \(J \subset U(2) \oplus U(2) \subset T_{\rho}\).
  For this \(J\), we have \(J^\perp_{T_\rho}/J = D_4^{\oplus 2} \oplus A_1^2.\)
\end{theorem}
\begin{proof}
  By \cite[Corollary~4.6]{art:09}, \(\overline{\D_{\rho}/\Gamma_{\rho}}^{\rm BB}\) has a unique cusp.
  Now use \Cref{cor:4j}.
\end{proof}

\subsection{Stable and Kulikov degenerations}
Following the notation in \Cref{subsec:KSBA_compactification}, let \(\Y_{\rho}\) be the normalization of the KSBA compactification of weighted pairs \((Y,(\frac{3}{4}+\epsilon)B)\), where \(Y \simeq \P^2\) and \(B \subset Y\) is a quartic curve.
This is a normal projective variety of dimension \(6\).
Recall, again from \Cref{subsec:KSBA_compactification}, that \(\overline F_{\rho}^{\rm KSBA}\) is the normalization of the KSBA compactification of weighted triples \((X, \epsilon R, \sigma)\), where \(X \to Y\) is a cyclic degree 4 cover branched along \(B\), \(R \subset X\) is the ramification divisor, and \(\sigma\) is the automorphism of type \(\rho\) induced by the covering construction.
\Cref{thm:stable_quotient_pair} gives an isomorphism \(\overline F_{\rho}^{\rm KSBA} \xrightarrow{\simeq} \Y_{\rho}\).

The geometry of \(\Y_{\rho}\) is described in \cite[\S~11.1]{hac:04}.
The pairs that give a type II surface correspond to those with reducible \(Y\).
\begin{proposition}\label{prop:z4}
  The locus \(Z \subset \Y_{\rho}\) that parametrises \((Y,B)\) with reducible \(Y\) is irreducible of dimension \(4\).
  Its points parametrise \((Y, B)\) where \(Y = Y_0 \cup Y_1\) is the union of two quadric cones \(Y_{i} = \P(1,1,2)\) glued along a line (with the singularities glued together) and \(B\) is the union of \(B_i \subset Y_i\) of class \(-K_{Y_i}\).
\end{proposition}
\begin{proof}
  See \cite[\S~11.1]{hac:04}.
\end{proof}

Let \((Y,B)\) correspond to a generic point of \(Z\).
Let $X \to Y$ be a cyclic degree 4 cover branched along $B$.
We describe the two components of $X$.
Consider \(\P(1,1,2) \subset \P^3\) as a singular quadric.
The class of a line through the singular point generates the Weil divisor class group; we call it $\mathcal{O}_{\P(1,1,2)}(1)$.
The canonical class \(K_{\P(1,1,2)}\) is \(\mathcal{O}_{\P(1,1,2)}(-4)\) and the double locus \(D_i \subset Y_i \simeq \P(1,1,2)\) is of class \(\mathcal{O}_{\P(1,1,2)}(1)\).
Since \(K_{\P(1,1,2)}+(3/4)B_i+D_i \sim 0\), the class of \(B_i \subset Y_i\) is \(\mathcal{O}_{\P(1,1,2)}(4)\).
\begin{proposition}
  Let \(C \subset \P(1,1,2)\) be a smooth curve of class \(-K = \mathcal{O}_{\P(1,1,2)}(4)\).
  Let \(V \to \P(1,1,2)\) be a cyclic degree \(4\) cover branched along \(C\) and \(\sigma\) the order 4 involution on \(V\) induced by the cyclic covering.
  Then \(V\) is a smooth \(\delpezzo_2\).
  The action of \(\Z/4\Z = \langle \sigma \rangle\) on \(V\) is free except along the pre-image of \(C\), where the stabilizer group is \(\Z/4\Z\), and at two points, where the stabilizer group is \(\Z/2\Z\).
  Let \(H \subset \P(1,1,2)\) be generic of class \(\mathcal{O}_{\P(1,1,2)}(1)\).
  Then the pre-image of \(H\) in \(V\) is the elliptic curve of \(j\)-invariant \(1728\).
\end{proposition}
\begin{proof}
  Setting \(W = V / \sigma^2\), we have
  \[ V \to W \to \P(1,1,2)\]
  where each map is of degree 2.
  The double cover \(\phi \colon W \to \P(1,1,2)\) is branched along \(C\).
  Since \(\mathcal{O}_{\P(1,1,2)}(C) = \mathcal{O}_{\P(1,1,2)}(2)^{\otimes 2}\) and \(\mathcal{O}_{\P(1,1,2)}(2)\) is a Cartier divisor, \(\phi\) is \'etale except over \(C\).
  In particular, the pre-image of the \(A_1\) singularity of \(\P(1,1,2)\) is the union of two points, both of which are \(A_1\)-singularities of \(W\).

  The double cover \(\psi \colon V \to W\) is branched along \(D = \frac{1}{2}\phi^{*}(C)\).
  We have \(\mathcal{O}_W(D) = \phi^{*}\mathcal{O}_{\P(1,1,2)}(2) = \phi^{*}\mathcal{O}_{\P(1,1,2)}(1)^{\otimes 2}\) and \(\mathcal{O}_{\P(1,1,2)}(1)\) is not Cartier at the two \(A_1\)-singularities of \(W\).
  Therefore, \(\phi\) is a non-trivial double cover in a neighborhood of the \(A_1\)-singularities but is \'etale on the punctured neighborhood.
  The only such cover, in local coordinates, is \(\C^2 \to \C^2/\pm 1\).
  Thus, we see that \(V\) is smooth.

  Let \(\pi = \psi \circ \phi\).
  Then
  \[ K_V = \phi^{*}(K_{\P(1,1,2)} + (3/4) C) = \phi^{*}\mathcal{O}_{\P(1,1,2)}(-1).\]
  So we see that \(K_V\) is anti-ample and \(K_V^2 = 4 \cdot (1/2) = 2\).
  That is, \(V\) is a smooth \(\delpezzo_2\).

  The cardinality of the fibers of \(V \to \P(1,1,2)\) is \(1\) over \(C\), is \(2\) over the \(A_1\)-singularity, and is \(4\) elsewhere.
  Therefore, \(\Z/4\Z = \langle \sigma \rangle\) has stabilizer \(\Z/4\Z\) at the points of the pre-image of \(C\), stabilizer \(\Z/2\Z\) at the points of the pre-image of the \(A_1\)-singularity, and trivial stabilizer elsewhere.

  Let \(E \subset V\) be the pre-image of \(H\).
  Then \(E\) is an elliptic curve with faithful \(\Z/4\Z\)-action.
  There is a unique such curve, namely the curve of \(j\)-invariant \(1728\).
\end{proof}
\begin{corollary}\label{cor:X4}
  Let \((Y,B)\) be the pair parametrised by a generic point of \(Z \subset \Y_{\rho}\) and let \(X \to Y\) be the degree 4 cover branched along \(B\).
  Then \(X = X_0 \cup X_1\) where each \(X_i\) is a smooth \(\delpezzo_2\) with a \(\Z/4\Z\)-action, and the two are glued along an elliptic curve which is anti-canonical in each component.
\end{corollary}

The surface \(X\) described in \Cref{cor:X4} is not a Type II Kulikov surface.
Indeed, letting \(E_i \subset X_i\) be the double curve, we have \(E_0^2 + E_1^2 = 2+2 \neq 0\).

We now  modify \(X\) to obtain a Kulikov surface (of Type II).
Let \(p \in D_0 \subset Y_0\) be a generic point.
Let \(\widetilde Y_0 \to Y_0\) be the blow-up at \(p\) and let
\begin{equation}\label{eqn:yt}
  \widetilde Y = \widetilde Y_0 \cup Y_1,
\end{equation}
glued along the proper transform of \(\widetilde D_0\) and \(D_1\).
Let
\begin{equation}\label{eqn:xt}
  \widetilde X = \widetilde X_{0} \cup X_1 \to \widetilde Y = \widetilde Y_0 \cup Y_1
\end{equation}
be the cyclic cover of degree 4 branched along \(B\).
Note that \(\widetilde X_1\) is simply the blow-up of \(X\) along the 4 points of the pre-image of \(p\).
\begin{proposition} \label{prop:Kulikov_quadruple_cover}
  There is a Kulikov model \(\mathcal X \to \Delta\) with a \(\Z/4\Z\)-action along the fibers such that the generic fiber is a cyclic degree 4 cover of \(\P^2\) branched along a quartic and the special fiber is the surface \(\widetilde X\).
\end{proposition}
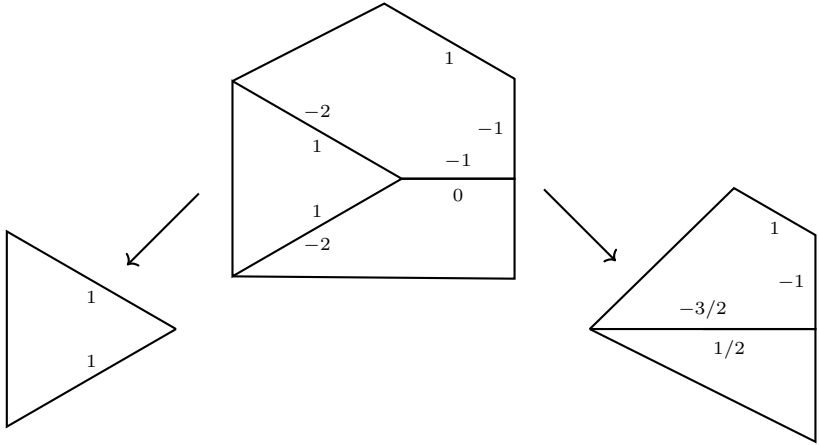
\begin{figure}[b]
  \centering
  \begin{tikzpicture}[thick, selfint/.style={font=\tiny}]
    \draw (0:1.5) -- node[below, selfint]{\(1\)} (120:1.5) -- (240:1.5) -- node[above, selfint]{\(1\)} (0:1.5);
    \draw(45:1) node (A){};
    \begin{scope}[xshift=3cm, yshift=2cm]
      \draw (0:1.5) -- node[below, selfint]{\(1\)} node[above, selfint]{\(-2\)} (120:1.5) -- (240:1.5) -- node[above, selfint]{\(1\)} node[below, selfint]{\(-2\)} (0:1.5);
      \draw (0:1.5) -- node[above, selfint]{\(-1\)} node[below, selfint]{\(0\)} (0:3) -- node[left, selfint] {\(-1\)} ++(90:1.33) -- node[below, selfint]{\(1\)} +(150:2) -- (120:1.5);
      \draw (0:1.5) -- (0:3) -- ++(270:1.33) -- (240:1.5);
      \draw (A) ++ (45:1.75) node (B) {};
      \draw (0:3.25) node (B2) {};
      \draw (B) edge [->] (A);     
    \end{scope}
        \begin{scope}[xshift=7cm]
      \draw (0:0) -- node[above, selfint]{\(-3/2\)} node[below right, selfint]{\(1/2\)} (0:3) -- node[left, selfint] {\(-1\)} ++(90:1.25) -- node[below, selfint]{\(1\)} +(150:1.25) -- (0:0);
      \draw (0:1.5) -- (0:3) -- ++(270:1.5) -- (0:0);
      \draw (B2) ++(-45:1.75) node (C) {};
      \draw (B2) edge [->] (C);
    \end{scope}
  \end{tikzpicture}
  \caption{We obtain a degeneration of \(\P^2\) to \(\Bl_p\P(1,1,2) \cup \P(1,1,2)\) by blowing up succesively in two lines and blowing down a \(\P^2\).
    The figure shows the central fibers in this process.
    The numbers next to the edges represent self-intersections.
    By taking a cyclic degree 4 cover, we obtain a type II Kulikov degeneration. }
  \label{fig:T}
\end{figure}
\begin{proof}
  Start with \(\P^2 \times \Delta \to \Delta\) and let \(L_0, L_1 \subset \P^2\) be two lines in the central fiber.
  Let \(\beta \colon \widehat{\mathcal Y} \to \P^2 \times \Delta\) be the blow-up along \(L_0\) followed by another blow-up along the proper transform of \(L_1\).
  Then the central fiber of \(\widehat{\mathcal Y} \to \Delta\) has three components (see \Cref{fig:T}).
  Let \(P \cong \P^2\) be the proper transform of the original central fiber.
  Blow it down to get \(\widetilde{\mathcal Y} \to \Delta\).
  Note that \(P\) has normal bundle \(\O_{\P^2}(-2)\), so the blow-down is singular.
  The central fiber of \(\widetilde{\mathcal Y} \to \Delta\) is \(\widetilde Y = \widetilde Y_0 \cup Y_1\) as in \eqref{eqn:yt}.
  It is easy to see that \(B \subset \widetilde Y\) can be smoothed to a divisor \(\mathcal B \subset \widetilde{\mathcal Y}\).
  We let \(\mathcal X \to \widetilde{\mathcal Y}\) be a cyclic cover of order 4 branched along \(\mathcal B\).
  This is the required Kulikov model.
  We leave it to the reader to verify that \(\mathcal X\) is indeed smooth and \(K\)-trivial.
  % We explain why \(\mathcal X\) is smooth and \(K\)-trivial.
  % To do so, consider the smooth threefold \(\widehat{\mathcal Y}\) and the map \(\beta \colon \widehat{\mathcal Y} \to \P^2 \times \Delta\).
  % Consider the divisor \(\mathcal B + 2 P \subset \widehat{\mathcal Y}\).
  % It is a divisor of class
  % \[\mathcal{O}_{\widehat{\mathcal Y}}(\mathcal B + 2P) = \beta^{*}\mathcal{O}_{\P^2 \times \Delta}(4) \otimes \mathcal{O}_{\widehat{\mathcal Y}}(2P) \otimes \mathcal{O}_{\widehat{\mathcal Y}}(-2\widehat Y_0-2\widehat Y_1) = \beta^{*}\mathcal{O}_{\P^2 \times \Delta}(4) \otimes \mathcal{O}_{\widehat{\mathcal Y}}(4P),\]
  % which is divisible by \(4\) in the Picard group.
  % Let \(\widehat{\mathcal X}\) be the normalization of a cyclic degree 4 cover of \(\widehat{\mathcal Y}\) branched along \(\mathcal B\).
  % The preimage of \(P\) in the central fiber of \(\widehat{\mathcal X}\) is the union of two copies of \(P \simeq \P^2\).
  % The threefold \(\mathcal X\) is obtained from \(\widehat{\mathcal X}\) by blowing down these copies of \(\P^2\).
  % Their normal bundle is \(\mathcal{O}_{\P^2}(-1)\), so \(\mathcal X\) is smooth.
  % We see that \(K_{\widehat{\mathcal X}}\) is the pull-back of \(\mathcal{O}_{\widehat{\mathcal Y}}(P)\), so \(K_{\mathcal X}\) is trivial.
\end{proof}

\subsection{The KSBA semifan}
Let \(\widetilde X = \widetilde X_0 \cup X_1\) be a very general choice of a Kulikov surface constructed in \eqref{eqn:xt} with double curve \(E\).
In this case, the lattice of numerically Cartier divisors \(\widetilde \Lambda\) is
\(\widetilde \Lambda = \ker (H^2(\widetilde X_0, \Z) \oplus H^2(X_1, \Z) \to H^2(E,\Z))\),
and the reduced lattice of numerically Cartier divisors \(\Lambda\)  is \( \Lambda = \widetilde \Lambda / (E, -E)\) (see \Cref{subsec:limiting_periods}).
Both are equipped with an order \(4\) automorphism \(\rho = \sigma^*\).
Let \(J \subset T_{\rho}\) be the isotropic plane corresponding to the unique cusp of \(\overline{\D_{\rho}/\Gamma_{\rho}}^{\rm BB}\).
By \eqref{eqn:JperpJrho} and \Cref{thm:bbcusp4}, we have isomorphisms
\[\Lambda^{\rm prim} = J^\perp_{T_\rho}/J = D_4^{\oplus 2} \oplus A_1^{\oplus 2}.\]
The next proposition identifies the $A_1^{\oplus 2}$ summand.
\begin{proposition} \label{prop:A1^2_summand}
  Let \(E_1,\dots, E_4 \subset \widetilde X_0\) be the exceptional divisors of \(\widetilde X_0 \to X_0\), labelled so that \(\rho\) acts on them by the 4-cycle \((1234)\).
  Then \(E_1-E_3\) and \(E_2-E_4\) span the \(A_1^{\oplus 2}\) summand of \(\Lambda^{\rm prim}\).
\end{proposition}
\begin{proof}
  It is clear that \(M = \langle E_1-E_3, E_2-E_4 \rangle\) lies in \(\Lambda^{\rm prim}\) and is isomorphic as a lattice to \(A_1^{\oplus 2}\).
  It remains to show that there is a projection \(\Lambda^{\rm prim} \to M\).
  Consider \(\alpha \in \Lambda^{\rm prim}\) represented by \((\alpha_0,\alpha_1) \in H^2(\widetilde X_0,\Z) \oplus H^2(X_1,\Z)\).
  Then \(\alpha_0 \in H^2(\widetilde X_0, \Z)\) is well-defined up to adding multiples of \(E\) and \(\tau(\alpha_0) = -\alpha_0 \pmod E\).
  From this, it is easy to check that \(\alpha_0 \cdot E_1 \equiv \alpha_0 \cdot E_3 \pmod 2\) and likewise for \(E_2\) and \(E_4\).
  The projection \(\Lambda^{\rm prim} \to M\) is 
  \[ (\alpha_0,\alpha_1) \mapsto \frac{1}{2}\left(\alpha_0 \cdot (E_3-E_1), \alpha_0 \cdot (E_4 - E_2)\right).\]
\end{proof}

\begin{theorem}\label{thm:ksba4}
  The space \(\overline{F}_{\rho}^{\rm KSBA}\) is isomorphic to the semi-toroidal compactification for the semifan \(\mathfrak F_J = A_1^{\oplus 2} \subset J_{T_\rho}^{\perp}/J = D_4^{\oplus 2} \oplus A_1^{\oplus 2}\).
\end{theorem}
\begin{proof}
  %The proof is similar to the proof of \Cref{thm:ksba01}.
  Let \(E\) be the elliptic curve with an order \(4\) automorphism \(\rho\).
  Observe that \(\rho^2\) acts as \(-1\), so \(\rho\) makes \(E\) a \(\Z[i]\)-module.

  Recall from \Cref{subsec:limiting_periods} that \(\mathcal{A}_J = J^\perp_{T_\rho}/J \otimes_{\Z[i]} E\) is the period domain for Kulikov limits of $\rho$-markable K3 surfaces.
  By \Cref{subsec:semitoroidal} and \Cref{subsec:KSBA_compactification}, the morphism
  \begin{equation}\label{eq:torksba}
  	\overline{\D_{\rho}/\Gamma_{\rho}}^{\rm tor} \to \overline{F}_\rho^{\rm KSBA}
  \end{equation}
  restricted to the boundary components over the unique cusp \(J\) of \(\overline{\D_{\rho}/\Gamma_\rho}^{\rm BB}\) is identified (up to a quotient by finite groups) with
  \[ \mathcal{A}_J \to \mathcal{A}_J /(\mathfrak F_J \otimes_{\Z[i]} E).\]
  That is, the fibers of \eqref{eq:torksba} are the translates of \(\mathfrak F_J \otimes_{\Z[i]} E\).

  We now show that \(\mathfrak{F}_J = A_1^{\oplus 2}\).
  Consider a Kulikov surface \(\widetilde X = \widetilde X_0 \cup X_1\) constructed in \eqref{eqn:xt}.
  Note that the log canonical contraction \(\widetilde{X} \to X\) contracts the exceptional divisors \(E_1, \dots, E_4\) of \(\widetilde{X_0} \to X_0\).
  Recall from \eqref{prop:A1^2_summand} that \(\langle E_{1}-E_3, E_2-E_4\rangle\) is the \(A_1^{\oplus 2}\) summand of \(J_{T_\rho}^{\perp}/J \subset J^{\perp}/J\).

  Let \(\psi \in \mathcal{A}_J \simeq \Hom_{\Z[i]}((J^{\perp}_{T_\rho}/J)^*, E)\) be the period of \(\widetilde X\).
  Consider the construction of \(\widetilde X\) as described in \eqref{eqn:yt} and \eqref{eqn:xt}.
  Let \(\widetilde X'\) be obtained from \(X\) by blowing up a point \(p' \in D_0 \subset Y_0\) different from \(p\).
  Let \(\psi'\) be the period of \(\widetilde X'\).
  Then \(\psi'\) and \(\psi\) differ only on the \(A_1^{\oplus 2}\)-summand.
  Conversely, any \(\psi'\) that differs from \(\psi\) only on the \(A_1^{\oplus 2}\)-summand arises from such a \(\widetilde X'\).
  Since \(\widetilde X\) and \(\widetilde X'\) have the same stable model, it follows that the translates of \(A_1^{\oplus 2} \otimes_{\Z[i]]} E\) are contracted by \eqref{eq:torksba}.
  That is, we have \(A_1^{\oplus 2} \subset \mathfrak F_{J}\).
  By a dimension count, we see that \(\mathfrak F_J\) must have rank 2.
  We conclude that \(A_1^{\oplus 2} = \mathfrak F_{J}\).
\end{proof}

\section{Eisenstein and 3-elementary lattices} \label{sec:K3_aut3_Eisenstein}

For this and the subsequent sections, assume that \(N=3\). Let \(\omega \colonequals \zeta_3 = e^{2\pi i/3}\).

\subsection{Generalities on Eisenstein lattices}
\label{subsec:eis}

An Eisenstein lattice \(M\) is a finite torsion-free module over the ring \(\E = \Z[\omega]\), together with an \(\E\)-valued Hermitian form.
Since \(E\) is a PID, \(M\) is a free \(\E\)-module of some rank \(n\).
It is thus a \(\Z\)-module of rank \(r=2n\) together with an automorphism \(\rho\colon x \to \omega\cdot x\) of order \(3\) that does not have a non-zero fixed vector.

If \(M\) is an ordinary lattice with a \(\Z\)-valued bilinear form \(\langle x, y \rangle\) and an automorphism \(\rho\) of order \(3\) without a non-zero fixed vector, then \(\rho\) defines a complex structure \(J\) on \(M \otimes \R\) by \(J = \frac1{\sqrt3}(\rho- \rho^2)\), and the associated Hermitian structure is \(H(x,y) = \langle x, y\rangle + i \langle x, Jy\rangle\). 
This Hermitian structure is not necessarily \(\E\)-valued but the rescaled form
\begin{equation*}
    h(x,y) = \frac32 H(x,y) = \frac12\big( 3\langle x,y\rangle + \theta \langle x, \rho y-\rho^2y\rangle \big)
\end{equation*}
is, where \(\theta = \omega-\omega^2 = \sqrt{-3}\).  
We call this associated Eisenstein lattice \(M^\E\). The effect of this rescaling is that \(h(x,y)\in \theta \E\) for all \(x,y\in~M^\E\). 
Equivalently,
\begin{equation} \label{eq:theta}
    M^\E \subset \theta (M^\E)', \quad\text{where}\quad
    (M^\E)' = \{x\in M^\E \mid h(x,y)\in \E\}.
\end{equation}

Vice versa, for any \(c\in \theta\E\) one has \(\frac23 \operatorname{Re} c\in\Z\). 
Thus for any Eisenstein lattice \(M^\E\) satisfying condition~\eqref{eq:theta}, the bilinear form \(\frac23 \operatorname{Re} h(x,y)\) on the underlying \(\Z\)-module \(M\) is \(\Z\)-valued. 
This gives a bijection between the \(\Z\)-lattices \((M,\rho)\) with an order \(3\) automorphism without a non-zero fixed vector and Eisenstein lattices \(M^\E\) satisfying condition~\eqref{eq:theta}.

Following \cite{art.sar:08}, we call the pair \((M,\rho)\) an \(\E\)-lattice; and if in addition \(\rho\) acts trivially on the discriminant group \(A_M = M^*/M\), we call it an \(\E^*\)-lattice. 
A lattice \(M\) is \(3\)-elementary if \(3M^*\subset M\), i.e. \(A_M= \Z_3^a\) for some integer \(a\), the \(\Z_3\)-rank of \(M\). 
One has \(M^* = \theta (M^\E)'\). 
This implies that \(M\) is \(3\)-elementary if and only if \(3\theta (M^\E)'\subset M^\E\). 

For \(x\in M^*\) one has \(\rho x - x = (1-\omega)x = \theta x/\omega\).
Thus, \(M\) is an \(\E^*\)-lattice if and only if \(\theta M^*\subset M\), i.e. \(M\) is ``\(\theta\)-elementary". 
In terms of the Hermitian form this is equivalent to the condition \(3 (M^\E)'\subset M^\E\), or that the matrix of \(3h^{-1}\) has entries in \(\E\). 
In particular, any \(\E^*\)-lattice is \(3\)-elementary (see \cite[Lemma 1.3]{art.sar:08} for another proof).

\subsection{Eisenstein 1-cusps} \label{subsec:eis_cusp}

The goal of this subsection is to prove \Cref{thm:Estar-1cusps}.

Recall the hyperbolic plane \(U = {\rm II}_{1,1} = \langle e, f\rangle\) with \(e^2 = f^2 = 0\) and \(ef = 1\).
It is well known that \(U\oplus U = \langle e_1,f_1\rangle \oplus \langle e_2, f_2 \rangle\) and \(U\oplus U(3) = \langle e_1,f_1 \rangle \oplus \langle e_1', f_1' \rangle\) are \(\E^*\)-lattices, with \(\rho\) acting as follows:

\begin{enumerate}
    \item \(e_1 \mapsto e_2\),\quad \(e_2\mapsto -e_1-e_2\),\quad\quad \(f_1\mapsto f_1+f_2\),\quad \(f_2\mapsto -f_1\)
    \item \(e_1\mapsto e_1-e'_2\),\quad \(e'_2\mapsto 3e_1-2e'_2\),\quad\quad \(f_1\mapsto -2f_1-f'_2\),\quad \(f'_2\mapsto 3f_1+f'_2\)
\end{enumerate}
and that this action is unique up to conjugation. 
The corresponding Eisenstein lattices have Hermitian forms
\begin{equation*}
    \begin{pmatrix}
    0&\theta\\ \bar\theta&0
  \end{pmatrix}
  \quad\text{and}\quad
  \begin{pmatrix}
    0&3\\ 3&0
  \end{pmatrix}.
\end{equation*}

\begin{lemma}\label{lem:3elem}
    The only even \(3\)-elementary lattices of signature \((1,1)\) are \(U\) and \(U(3)\), and of signature \((2,2)\) are \(U\oplus U\), \(U\oplus U(3)\) and \(U(3)\oplus U(3)\).
\end{lemma}
\begin{proof}
    The \(\Z_3\)-rank \(a\) is even in all cases. 
    For a hyperbolic lattice (i.e., \((1,1)\)-lattice) it is true by Rudakov--Shafarevich \cite{rud.sha:81}, cf. \cite[Thm. 1.5]{art.sar:08}. 
    And for a \((2,2)\)-lattice it is true by the same theorem applied to the orthogonal complement in the K3 lattice \(U^{\oplus 3}\oplus E_8^2\). 
    In the boundary cases \(a=0\), resp. \(a=\rk H\) the lattice \(H\) is unimodular, resp. unimodular dilated by \(3\), so it it unique. 
    In the case \(r=4, a=2\), the lattice is unique by \cite[1.13.2]{nik:79}.
\end{proof}

\begin{lemma}\label{lem:Estar}
    The only \(\E^*\)-lattices of signature \((2,2)\) are \(U\oplus U\) and \(U\oplus U(3)\) with the above \(\rho\)-action.
\end{lemma}
\begin{proof}
    \(U(3)\oplus U(3)\) is not an \(\E^*\)-lattice. 
    Indeed, the \(\rho\)-action on \(H\) also defines a \(\rho\)-action on \(H(\frac13) = U\oplus U\), which is unique by the above. 
    It is easy to see that the induced action on \(A_{U(3)\oplus U(3)}\) is nontrivial. 
    Another way to see it is: in this case \(3 (M^\E)' = \frac1{\theta}(M^\E)'\not\subset M^\E\).
\end{proof}

For a lattice \(T\), we call an isotropic line \(I=\Z e\subset T\) a \(0\)-cusp and an isotropic plane, i.e. a saturated isotropic sublattice \(J\simeq \Z^2 \subset T\), a \(1\)-cusp. 
The justification for this is the fact that when \(T\) has signature \((2,n)\), \(I\) (resp., \(J\)) define a \(0\)-cusp (resp., \(1\)-cusp) of the corresponding Type IV Siegel domain \(\mathcal D(T)\). 

\begin{lemma}\label{lem:0cusp}
    Let \(T\) be a \(3\)-elementary lattice and \(e\in T\) be an isotropic vector. 
    Then there exists an orthogonal decomposition \(T = H \oplus \overline{T}\) with \(H\simeq U\) or \(U(3)\), \(e\in H\) and \(e^\perp/e \simeq \overline{T}\).
\end{lemma}
\begin{proof}
    An isotropic vector \(e\in T\) defines an isotropic vector \(e^* = e/{\rm div} (e)\in A_T\) in the discriminant group.
    Here the divisibility \({\rm div}(e)\) is defined by \(e\cdot T = {\rm div}(e)\Z\). 
    Let \(\overline{T} \colonequals e^\perp/e\). 
    One has \({\rm div}(e) = 1\) or \(3\). If \({\rm div}(e)=1\) then it is contained in a unimodular sublattice \(H\subset T\), which is isomorphic to \(U\) by \Cref{lem:3elem}. 
    It splits off as a direct summand, and the statement follows.

    If ${\rm div}(e)=3$ then $e^*\ne 0$ and the lattice \(\overline{T}\) has the discriminant group \((e^*)^\perp/e^*\). 
    Thus, \(a(\overline{T}) = a(T)-2\).
    Consider the projection \(e^\perp \to \overline{T}\). 
    Choose a lift \(i\from \overline{T} \to e^\perp \subset T\), which is automatically an isometry. 
    Let \(H = i(\overline{T})^\perp\) in \(T\). By \cite[1.15.1]{nik:79} it is \(3\)-elementary, so it is isomorphic to \(U\) or \(U(3)\) by \Cref{lem:3elem}. 
    The first case is impossible, so \(H\simeq U(3)\). 
    We have an inclusion of lattices \(H \oplus \overline{T} \subset T\), and the lattices on the left and on the right are \(3\)-elementary with the same \(\Z_3\)-rank \(a\). 
    Thus this inclusion is an identity and the statement follows. 
\end{proof}

Now let \(T\) be an \(\E\)-lattice.  
If \(e\in T\) is an isotropic vector then so are \(\rho e\) and \(\rho^2 e\), and the identity
\(e+\rho e+ \rho^2 e=0\) implies that \(\langle e, \rho e\rangle = 0\). 
Thus, the saturation of the span of \(e\) and \(\rho e\) is an isotropic plane \(J\simeq\Z^2\subset T\). 

\begin{theorem}\label{thm:Estar-1cusps}
    Let \(T\) be an \(\E^*\)-lattice and \(J\subset T\) be an isotropic plane. 
    Then there exists an orthogonal decomposition 
    of \(\E^*\)-lattices
    \(T = H \oplus \overline{\overline{T}}\) with \(H\simeq U\oplus U\) or \(H=U\oplus U(3)\) an \(\E^*\)-lattice as above, \(J\) contained in \(H\), and \(J^\perp/J = \overline{\overline{T}}\) as an \(\E^*\)-lattice. 
    The isotropic planes in \(T\) up to the group \(O_\rho(T)\) of isometries commuting with \(\rho\) are in a bijection with the isomorphism classes of the \(\E^*\)-lattices \(J^\perp/J\) appearing this way.
\end{theorem}
\begin{proof}
    Let us work with the corresponding Eisenstein lattice \(J^\E\subset T^\E\). 
    Then \(J^\E = \E e\) is a free rank-\(1\) \(\E\)-module. 
    Repeating the argument of \Cref{lem:0cusp}, we get a sublattice \(H^\E = (\E e)^\perp\) and an inclusion \(H^\E\oplus \overline{\overline{T}}^\E \subset T^\E\).  
    (Note here that the orthogonals in \(L\) and \(L^\E\) are the same.)  
    On the other hand, we can obtain the same lattice \(\overline{\overline{T}}\) in two steps by working with \(\Z\)-lattices, by adding two ordinary isotropic vectors \(e_1\in T\) and then \(e_2\in\overline{T}\). 
    Thus, by \Cref{{lem:0cusp}} applied twice there is an \emph{equality} \(T = H\oplus \overline{\overline{T}}\) and \(H\) is one of the lattices of \Cref{lem:3elem}.  \
    Since \(T\) is an \(\E^*\)-lattice then so is its direct summand \(H\). 
    Thus by \Cref{lem:Estar} \(H=U\oplus U\) or \(U\oplus U(3)\).

    One has \(a(H) + a(\overline{\overline{T}}) = a(T)\). 
    Thus the isomorphism class of \(\overline{\overline{T}} = J^\perp/J\) defines \(H\) uniquely.  
    Two decompositions with the isomorphic \(H\) and \(\overline{\overline{T}}\) differ by an isometry, mapping the first (resp. second) summand to the first (resp. second) summand.
\end{proof}

Note that the action of \(\Gamma_\rho \subset O(L)\) on \(T_\rho\) is induced by the restriction \(\Gamma_\rho \to O_\rho(T_\rho)\), which is surjective by the proof of \cite[Theorem 3.6]{ma.oha.tak:15}. 
Thus, $\D_\rho/\Gamma_\rho = \D_\rho/O_\rho(T_\rho)$ and
cusps of \(\overline{\D_\rho/\Gamma_\rho}^{\rm BB}\), which are \(\Gamma_\rho\)-orbits of \(\rho\)-invariant isotropic planes \(J \subset T_\rho\), are in a bijection with the isomorphism classes of \(J^{\perp}_T/J\) by \Cref{thm:Estar-1cusps}.

\subsection{$1$-cusps of K3 $\E^*$-lattices}
\label{subsec:1-cusps-Estar}

Artebani and Sarti \cite{art.sar:08} classified the period lattices \(T\) (corresponding to \(T_\rho\) from \Cref{subsec:period_map}) of K3 surfaces with a nonsymplectic automorphism of order \(3\). 
The main result is given by \cite[Table~2]{art.sar:08} which lists the possible lattices \(T=T(n,k)\) and their orthogonals \(S=T^\perp_L\) in the K3 lattice $L$ (they are denoted \(N\) in that paper). 
All of them turn out to be $\E^*$-lattices. 
Our main interest are the maximal lattices with \(g\ge2\) which are as follows:

\begin{table}[h]
  \centering
  \begin{tabular}[h]{c c l l l}
    \toprule
    \(g\)& \((n, k)\) & \(S = T^\perp_L\) & \(T\) & \(P = T^\perp_\Lambda \)\\
    \midrule
  5&(0,2) &\(U\) &\(U^2\oplus E_8^2\) &\(E_8\)\\
  4&(0,1) &\(U(3)\) &\(U\oplus U(3)\oplus E_8^2\) &\(E_6\oplus A_2\)\\
  3&(1,1) &\(U(3)\oplus A_2\) &\(U\oplus U(3)\oplus E_6\oplus E_8\)
                             &\(E_6\oplus A_2^2\)\\      
  2&(2,1) &\(U(3)\oplus A_2^2\) &\(U\oplus U(3)\oplus E_6^2\)
                             &\(E_6\oplus A_2^3\)\\
    \bottomrule\\
\end{tabular}
  \caption[Maximal \(\E^*\)-lattices]{Maximal K3 \(\E^*\)-lattices with \(g\ge2\)}
  \label{tab:ASlattices}
\end{table}

The negative definite lattice \(P\) of the last column is the orthogonal of \(T\) in the even unimodular lattice \(\Lambda = {\rm II}_{2,26}\). 
We will need it in the proof of \Cref{thm:t11-t21cusps}.

\begin{theorem}\label{thm:t11-t21cusps}
    The cusps $J$ of the lattices $T$ in \Cref{tab:ASlattices} modulo \(O_\rho(T)\) are uniquely determined by the root sublattices of $J^\perp/J$, which are given in \Cref{tab:AScusps}. Here, $R^*$ denotes an index-$3$ overlattice of a root lattice $R$.
\end{theorem}

\begin{table}[h]
  \centering
  \begin{tabular}[h]{c c l}
    \toprule
    \(g\)& \((n,k)\) & \( J^\perp/J \) \\
    \midrule
  5&(0,2) & $ E_8^2$ \\
  4&(0,1) & $ E_8^2$, $ E_8\oplus E_6\oplus A_2$, $ (E_6^2\oplus A_2^2)*$ \\
  3&(1,1) & \( E_8\oplus E_6\), \( E_6^2\oplus A_2\), \( E_8\oplus A_2^3\), \( (E_6\oplus A_2^4)*\)  \\    
  2&(2,1) & \( E_8\oplus A_2^2\), \( E_6^2\), \( E_6\oplus A_2^3\), \( A_2^6* \)   \\
    \bottomrule\\
\end{tabular}
  \caption[The $1$-cusps of \(\E^*\)-lattices]{The $1$-cusps of \(\E^*\)-lattices with \(g\ge2\)}
  \label{tab:AScusps}
\end{table}

%\va{I rewrote this theorem to include all cases, including the known ones, in one place, for convenience and readability.}

\begin{proof}
The Eisenstein $1$-cusps of $T(0,1)$ were found in \cite[Theorem 5.9]{cas.jen.laz:12}. The case \(T(0,2)\) is very easy. The proof in \cite{cas.jen.laz:12} proceeds by using an extension of Scattone's method \cite{sca:87} to the Eisenstein case. 
For the lattices \(T(2,1)\) and \(T(1,1)\) in addition to this method we also need \Cref{thm:Estar-1cusps} to show uniqueness of $J$ for a given $J^\perp/J$.

    The lattice \(\Lambda = {\rm II}_{2,26} \simeq U^2\oplus E_8^3\) has a unique \(\rho\)-action making it into an Eisenstein lattice \(\Lambda^\E\). 
    For each \(T\)-lattice of \Cref{tab:ASlattices}, \(T^\E\) can be embedded into \(\Lambda^\E\). 
    Let \(P^\E\) be its orthogonal. 
    Note that \(P\) is a negative definite \(\E^*\)-lattice given in the last column of \Cref{tab:ASlattices}. 
    Let \(J\subset T\) be an isotropic plane. We seek $J^\perp/J = J^\perp_T/J$, where $J^\perp_T$ denotes the orthogonal in~$T$.

    On the other hand, $J^\perp_\Lambda/J=N$ is a \emph{unimodular} negative definite lattice of rank $24$, that is one of the $24$ Niemeier lattice (see e.g. \cite[Ch. 18]{con.slo:99}) or the Leech lattice. Let $R$ denote the root sublattice of $N$; it is well known that $N$ is uniquely determined by $R$. One has $P\subset R$ (in particular, $R\ne 0$ so $N$ is not the Leech lattice) and $J^\perp_T/J = P^\perp_N$. 

    It follows that $N$ and $R$ must be $\E$-lattices. The only such lattices $N$ with $P\subset R$ in our four cases are $N(E_8^3)$ and $N(E_6^4)$. It remains to find the embeddings $P\subset R$ and to compute $P^\perp_N$, which is the saturation of $P^\perp_R$ in $N$. Moreover, computing $P^\perp_R$ is very easy since \((A_2)^\perp_{E_8} = E_6\), \((E_6)^\perp_{E_8} = A_2\), \((A_2)^\perp_{E_6} = A_2^2\), \((A_2^2)^\perp_{E_6} = A_2\) and $(A_2^2)^\perp_{E_8}=A_2^2$.

    The lattice $P(0,2)=E_8$ has a unique embedding into $E_8^3$.
    $P(0,1)=E_6\oplus A_2$ has two embeddings into $E_8^3$ and one into $E_6^4$.
    $P(1,1)=E_6\oplus A_2^2$ has three embeddings into $E_8^3$ and two into $E_6^4$. $P(2,1)=E_6\oplus A_2^3$ has four embeddings into $E_8^3$ and four into $E_6^4$.
    For each of these embeddings we compute $P^\perp_R$ and obtain the root systems in the statement.
    
    For $R=E_8^3$ one has $N=R$, so $P^\perp_N=P^\perp_R$. 
    For $R=E_6^4$ one has $A_R=R^*/R=\Z_3^4$ and $N$ is an index $9$ sublattice of $R^*$ corresponding to a ``glue'', an isotropic space $V\simeq\Z_3^2\subset A_R$ listed in \cite[Ch. 18]{con.slo:99}). 
    Each nonzero glue vector $v\in V$ has exactly one zero coordinate in $\Z_3^4$. On the other hand, in the non-starred cases the image of $P^\perp_{R^*}$ in $\Z_3^4$ has at least two zero coordinates, so $P^\perp_R$ is saturated in $N$. In the three starred cases one has $V\cap (P^\perp_{R^*}/R)=\Z_3\subset A_R$, so the saturation is a $\Z_3$-extension of $P^\perp_R$. (Here, we use $(A_2)^\perp_{E_6^*}/E_6=\Z_3$ and $(A_2^2)^\perp_{E_6^*}/E_6=0$.)   
    \iflongversion
    
    Here is how one computes $P^\perp_{R^*}$. Using Bourbaki's notation for the roots of $E_6$: Let $A_2 = \langle \alpha_1,\alpha_3\rangle \subset E_6$. Then $A_2^\perp = \langle \alpha_2^*, \alpha_4^*, \alpha_5^*, \alpha_6^*\rangle$. Here $\alpha_i^*$ are the fundamental coweights, giving the dual basis to the root basis. Similarly, for $A_2^2 = \langle \alpha_1,\alpha_3, \alpha_5,\alpha_6\rangle \subset E_6$ one has $A_2^\perp = \langle \alpha_2^*, \alpha_4^*\rangle$. In $A(E_6)=\Z_3$ one has $\alpha_2^*=\alpha_4^*=0$ and $\alpha_i^*\ne0$ for $i=1,3,5,6$. Now, $A(E_6^4)=\Z_3^4$, and the above computation shows which coordinates in $P^\perp_{R^*}$ are zero.
    Tables \ref{tab:expl01}, \ref{tab:expl11}, \ref{tab:expl21-1}, \ref{tab:expl21-2} give the details of this computation.
    \begin{table}[h]
    \centering
    \begin{tabular}[h]{|lll|lll|llll|}
    $P$ &$R$ &$P_R^\perp$ &$P$ &$R$ &$P_R^\perp$ &$P$ &$R$ &$P_R^\perp$ &$P^\perp_{R^*}/R$ \\
    \midrule
    $E_6A_2$ & $E_8$ & $0$ 
    & $E_6$ & $E_8$ & $A_2$ 
    & $E_6$ & $E_6$ & $0$ & $0$ \\
     & $E_8^2$ & $E_8^2$  
    & $A_2$ & $E_8$ & $E_6$
    & $A_2$ & $E_6$ & $A_2^2$ & $\Z_3$\\
    & & & & $E_8$ & $E_8$ 
    & & $E_6^2$ 
    & $E_6^2$ & $\Z_3^2$\\
    \end{tabular}
    \medskip
    \caption{Case $(n,k)=(0,1)$}
    \label{tab:expl01}
    \end{table}

    \begin{table}[h]
    \centering
    \begin{tabular}[h]{|lll|lll|lll|llll|llll|}
    $P$ & $R$  &$P_R^\perp$ 
    & $P$ &$R$ & $P_R^\perp$ 
    & $P$ &$R$ & $P_R^\perp$ 
    & $P$ &$R$ & $P_R^\perp$ & $P^\perp_{R^*}/R$ 
    & $P$ &$R$ & $P_R^\perp$ & $P^\perp_{R^*}/R$ 
    \\
    \midrule
$E_6A_2$ & $E_8$ & $0$
& $E_6$ & $E_8$ & $A_2$
& $E_6$ & $E_8$ & $A_2$ 
& $E_6$ & $E_6$ & $0$ & $0$
& $E_6$ & $E_6$ & $0$ & $0$
\\
$A_2$ & $E_8$ & $E_6$
& $A_2^2$ & $E_8$ & $A_2^2$
& $A_2$ & $E_8$ & $E_6$
& $A_2^2$ & $E_6$ & $A_2$ & $0$
& $A_2$ & $E_6$ & $A_2^2$ & $\Z_3$
\\
& $E_8$ & $E_8$
& & $E_8$ & $E_8$
& $A_2$ & $E_8$ & $E_6$
& & $E_6^2$ & $E_6^2$ & $\Z_3^2$
& $A_2$ & $E_6$ & $A_2^2$ & $\Z_3$
\\
&& &&& &&& &&&& & $E_6$ 
& & $E_6$ & $\Z_3$
    \end{tabular}
    \medskip
    \caption{Case $(n,k)=(1,1)$}
    \label{tab:expl11}
    \end{table}  

    \begin{table}[h]
    \centering
    \begin{tabular}[h]{|lll|lll|lll|lll|}
    $P$ & $R$ & $P_R^\perp$ 
    & $P$ & $R$ & $P_R^\perp$
    & $P$ & $R$ & $P_R^\perp$
    & $P$ & $R$ & $P_R^\perp$
    \\
    \midrule
$E_6A_2$ & $E_8$ & $0$
& $E_6A_2$ & $E_8$ & $0$
& $E_6$ & $E_8$ & $A_2$
& $E_6$ & $E_8$ & $A_2$
\\
$A_2^2$ & $E_8$ & $A_2^2$
& $A_2$ & $E_8$ & $E_6$
& $A_2^3$ & $E_8$ & $A_2$
& $A_2^2$ & $E_8$ & $A_2^2$
\\
& $E_8$ & $E_8$
& $A_2$ & $E_8$ & $E_6$
& & $E_8$ & $E_8$
& $A_2$ & $E_8$ & $E_6$
\\
    \end{tabular}
    \medskip
    \caption{Case $(n,k)=(2,1)$, embeddings into $N(E_8^3)$}
    \label{tab:expl21-1}
    \end{table}

    \begin{table}[h]
    \centering
    \begin{tabular}[h]{|llll|llll|llll|llll|}
$P$ &$R$ & $P_R^\perp$ & $P^\perp_{R^*}/R$
& $P$ &$R$ & $P_R^\perp$ & $P^\perp_{R^*}/R$
& $P$ &$R$ & $P_R^\perp$ & $P^\perp_{R^*}/R$
& $P$ &$R$ & $P_R^\perp$ & $P^\perp_{R^*}/R$
    \\
    \midrule
$E_6$ & $E_6$ & $0$ & $0$
& $E_6$ & $E_6$ & $0$ & $0$
& $E_6$ & $E_6$ & $0$ & $0$
& $E_6$ & $E_6$ & $0$ & $0$
\\
$A_2^3$ & $E_6$ & $0$ & $0$
& $A_2^2$ & $E_6$ & $A_2$ & $0$
& $A_2^2$ & $E_6$ & $A_2$ & $0$
& $A_2$ & $E_6$ & $A_2^2$ & $\Z_3$
\\
& $E_6^2$ & $E_6^2$ & $\Z_3^2$
& $A_2$ & $E_6$ & $A_2^2$ & $\Z_3$
& $A_2$ & $E_6$ & $A_2^2$ & $\Z_3$
& $A_2$ & $E_6$ & $A_2^2$ & $\Z_3$
\\
&&&
& & $E_6$ & $E_6$ & $\Z_3$
& & $E_6$ & $E_6$ & $\Z_3$
& $A_2$ & $E_6$ & $A_2^2$ & $\Z_3$
    \end{tabular}
    \medskip
    \caption{Case $(n,k)=(2,1)$, embeddings into $N(E_6^4)$}
    \label{tab:expl21-2}
    \end{table}    
\else
 The version on arXiv has details of this computation.
\fi
\end{proof}

\iflongversion
\begin{remark}
    In several cases, it happens that the same cusp of $T$ mod $O(T)$ corresponds to two different cusps of $\Lambda$ mod $O(\Lambda)$, to both $N(E_8^3)$ and $N(E_6^4)$. For $(n,k)=(1,1)$ this happens for $E_6^2\oplus A_2$, and for $(n,k)=(2,1)$ for $E_6^2$ and $E_6\oplus A_2^3$. This is possible because not every isometry of $T$ extends to an isometry of $N$. Indeed, in these cases the homomorphism $\phi_T\colon O(T)\to O(q_T)$ is surjective but $\phi_P\colon O(P)\to O(q_P)=O(q_T)$ is not. An isometry $g\in O(T)$ extends to an isometry of $\Lambda$ precisely when $\phi_T(g) \in \im \phi_P$.
    
    An isometry $g\in O(E_6\oplus A_2^k)$ sends the $E_6$ block to itself and permutes the $A_2$ blocks, because $g$ sends a root system to a root system and the direct decomposition into irreducible root systems is unique. But the group $O(q_{E_6\oplus A_2^2})\simeq S_3\ltimes \mathbb Z_2^3$ has order $48\ne 2^3\cdot 2$, so it contains an element that does not preserve the block decomposition.
\end{remark}
\fi

\begin{remark}
    By Theorem~\ref{thm:Estar-1cusps} one has $T\simeq H\oplus J^\perp/J$. Here, $H=U^2$ iff the $\Z_3$-ranks of $T$ and $J^\perp/J$ coincide. In this case $S$ is saturated in $J^\perp/J$. The opposite case is $a(T)=a(J^\perp/J)+2$, then $H=U\oplus U(3)$ and the torsion subgoup of $(J^\perp/J)/S$ is $\Z_3$.
\end{remark}
%\vainline{I added this remark. The part about saturation of $S$  applies to the case when $J^\perp$ is taken in any unimodular lattice, e.g. in the K3 lattice $II_{3,19}$ instead of the lattice $II_{2,16}$ here. }

\section{K3 surfaces with an automorphism of order 3} \label{sec:Kulikov_trigonal}
\subsection{The triple Tschirnhausen construction}\label{sec:triple-tschirnhausen}
Let \(P\) be a reduced connected projective curve, possibly singular, and let \(p_1, \dots, p_n \in P\) be distinct points in the smooth locus of \(P\).
Let \(\phi \colon C \to P\) be a Gorenstein triple cover, that is, a finite flat morphism of degree 3 with Gorenstein fibers.
Assume that \(\phi\) is \'etale over the generic points of the components of \(P\) and also over \(p_1,\dots, p_n\).
Let \(E = \left(\phi_{*}\O_C/\O_P\right)^{\vee}\) be the Tschirnhausen bundle of \(\phi\) and \(\br \phi \subset P\) the branch divisor of \(\phi\).
Then \(\O_P(\br \phi) = \det E^{\otimes 2}\).
Assume that the line bundle \(\det E^{\vee} \otimes \O_P(2p_1 + \dots + 2p_n)\) is divisible by 3 in \(\Pic P\).
In fact, let \(\eta \in \Pic P\) be a cube root of this line bundle.
We associate to \((\phi \colon C \to P, p_1,\dots, p_n, \eta)\) a surface \(X\) together with an automorphism \(\sigma\) of order 3.
In our applications, \(\Pic P\) will be torsion-free, so there will be a unique choice of \(\eta\).
In that case, we drop \(\eta\) from the notation.
\Cref{fig:triple-tschirnhausen} shows a sketch of the construction.

By \cite[Theorem 1.3]{cas.eke:96}, we have a canonical embedding \(C \subset \P E\) whose image is a Cartier divisor of class \(\O_{\P E}(3) \otimes \pi^{*}\det E^{\vee}\).
Let \(\pi \colon \P E \to P\) be the projection and let \(F_1, \dots, F_n\) be the fibers of \(\pi\) over \(p_1,\dots,p_n\).
Let \(\widehat \P E \to \P E\) be the blow-up at the \(3n\) points \(\bigsqcup_i C \cap F_{i}\).
Let \(\widehat C\) and \(\widehat F_i\) be the proper transforms in \(\widehat \P E\) of \(C\) and \(F_i\).
Note that these curves are pairwise disjoint.
Set \(\widehat F = \bigsqcup_i \widehat F_i\).
Let \(\epsilon \subset \widehat \P E\) be the exceptional divisor of the blow-up.
Then \(\epsilon\) is the disjoint union of \(3n\) smooth rational curves of self-intersection \(-1\).

Consider the cyclic triple cover \(\widehat \P E\) branched over \(\widehat C + 2\widehat F\).
This is constructed as follows.
Let \(H \in \Pic \widehat \P E\) be the class of the pull-back of \(\O_{\P E}(1)\).
Observe that in \(\Pic \widehat \P E\), we have
\[ \widehat C + 2\widehat F = 3(H + \eta - \epsilon).\]
Let \(s\) be a global section of \(\O_{\P E}(3(H + \eta - \epsilon))\) whose zero locus is \(\widehat C + 2\widehat F\).
We have the \(\O_{\P E}\)-algebra
\[\O_{\widehat \P E} \oplus \O_{\widehat \P E}(-H-\eta+\epsilon) \oplus \O_{\widehat \P E}(-2H-2\eta+2\epsilon),\]
where the multiplication map sends 
\[ \O_{\widehat \P E}(-H-\eta+\epsilon) \otimes \O_{\widehat \P E}(-2H-2\eta+2\epsilon) \to \O_{\widehat \P E}\]
by the map induced by \(s\).
The cyclic triple cover is the relative spectrum of this algebra.
Let \(q \in \widehat F\) be a point.
Choose local coordinates \(x,y\) on \(\widehat \P E\) at \(q\) so that \(x = 0\) cuts out \(\widehat F\).
Then, in coordinates, the triple cover is given by
\[ \spec \C[x,y,t]/(t^3-x^2)  \to \spec \C[x,y].\]
Let \(\widehat X\) be the normalisation of the cyclic triple cover constructed above along the pre-image of \(\widehat F\).
Then \(\widehat X\) is non-singular along the pre-image of \(\widehat F\), and for each \(i\), the pre-image of \(\widehat F_i\) is a \((-1)\)-curve on \(\widehat X\).
Let \(\widehat X \to X\) be the blow-down of these \((-1)\)-curves.
Let \(\widehat \sigma\) be the automorphism of \(\widehat X\) arising from the cyclic triple cover and let \(\sigma\) be the induced automorphism of \(X\).
The construction is now complete.

We say that \((X, \sigma)\) is obtained from \((\phi \colon C \to P, p_1, \dots, p_n, \eta)\) by the \emph{triple Tschirnhausen construction pinched over \(p_1, \dots, p_n\)}.

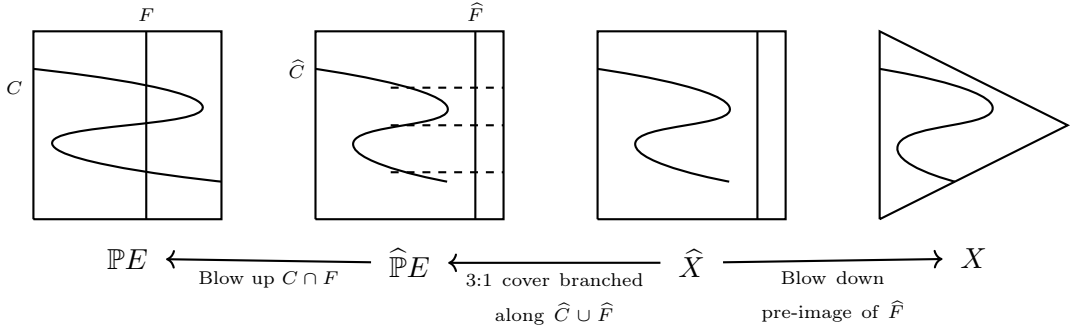
\begin{figure}
  \centering
      \begin{tikzpicture}[thick, scale=2.5]
      \draw (0,0) -- (0,1) -- (1,1) -- (1,0) -- (0,0);
      \draw plot [smooth, tension=1] coordinates { (0,.8)  (.9,.6)  (.1,.4) (1,.2)};
      \draw (0,.7) node[left] {\tiny \(C\)};
      \draw (.6,0)  -- (.6,1) node[above] {\tiny \(F\)};
      \draw (0.5,-0.1) node[below] (PE) {\(\P E\)} ;

      \begin{scope}[xshift = 1.5cm]
        \draw (0,0) -- (0,1) -- (1,1) -- (1,0) -- (0,0);
        \draw plot [smooth, tension=1] coordinates { (0,.8)  (.7,.6)  (.2,.4) (.7,.2)};
        \draw (0,.8) node[left] {\tiny \(\widehat C\)};
        \draw[dashed] (0.4,0.7) -- (1,0.7);
        \draw[dashed] (0.4,0.5) -- (1,0.5);
        \draw[dashed] (0.4,0.25) -- (1,0.25);
        \draw (.85,0) -- (.85,1) node[above] {\tiny \(\widehat F\)};;
        \draw (0.5,-0.1) node[below] (PEh) {\(\widehat \P E\)};
      \end{scope}

      \begin{scope}[xshift = 3.0cm]
                \draw (0,0) -- (0,1) -- (1,1) -- (1,0) -- (0,0);
        \draw plot [smooth, tension=1] coordinates { (0,.8)  (.7,.6)  (.2,.4) (.7,.2)};
        \draw (.85,0) -- (.85,1);
        \draw (0.5,-0.1) node[below] (Xh) {\(\widehat X\)};
      \end{scope}

       \begin{scope}[xshift = 4.5cm]
        \draw (0,0) -- (1,0.5) -- (0,1) -- (0,0);
        \draw plot [smooth, tension=1] coordinates {(0,.8)  (.6,.6)  (.1,.4) (.4,.2)};
        \draw (0.5,-0.1) node[below] (X) {\(X\)};
      \end{scope}
      \draw[->,shorten <=.1cm, shorten >=.1cm]
      (PEh) edge node[below]{\tiny Blow up \(C \cap F\)} (PE)
      (Xh) edge node[below, text width=2.5cm, align=center]{\tiny 3:1 cover branched along \(\widehat C \cup \widehat F\)} (PEh)
      (Xh) edge node[below, text width=2.5cm, align=center]{\tiny Blow down pre-image of \(\widehat F\)} (X);
    \end{tikzpicture}
    \caption{The pinched triple Tschirnhausen construction}
    \label{fig:triple-tschirnhausen}
\end{figure}

\begin{remark}
  We can describe the cover \(\widehat X \to \widehat \P E\) as a standard cover in the sense of \cite[Theorem 2.1]{par:91}.
  Let \(\chi \colon \mu_3 \to \C^{*}\) be the tautological character.
  In the notation of \cite{par:91}, the cover \(\widehat X \to \widehat \P E\) corresponds to the building data
  \begin{align*}
    L_{\chi} = \O_{\widehat \P E}(H+\eta-\epsilon), &\qquad
    L_{\chi^2}= L_{\chi}^{\otimes 2} \otimes \O_{\widehat \P E}\left(-\widehat F\right), \\
    D_{\mu_3, \zeta_3} = \widehat C, &\text{ and }
    D_{\mu_3, \zeta_3^2} = \widehat F.
  \end{align*}
  As an \(\O_{\widehat \P E}\)-module, we have
  \[ \O_{\widehat X} = \O_{\widehat \P E} \oplus L_{\chi}^{\vee} \oplus L_{\chi^2}^{\vee}.\]
  We have an action of \(\mu_3\) on \(\O_{\widehat X}\).
  It preserves the direct sum decomposition and acts by the characters \(\chi^0, \chi^1\), \(\chi^2\) on the summands.
\end{remark}

Let \(\widetilde C\) be the image in \(X\) of the pre-image of \(\widehat C \subset \widehat \P E\) in \(\widehat X\).
Note that \(\widetilde C\) is an isomorphic copy of \(C\).
Let \(x_i\) be the image in \(X\) of the pre-image of \(\widehat F_i \subset \widehat \P E\) in \(\widehat X\).
\begin{proposition}\label{prop:fixedlocus}
  The fixed locus of \(\sigma\) on \(X\) is the disjoint union of \(\widetilde C\) and the points \(x_1, \dots, x_n\).
\end{proposition}
\begin{proof}
  Since \(\widehat X \to \widehat \P E\) is totally ramified over \(\widehat C \sqcup \widehat F\), the fixes locus of \(\sigma\) on \(\widehat X\) is the disjoint union of the pre-images of \(\widehat C\) and \(\widehat F_i\).
  The fixed locus on \(X\) is simply the image.
\end{proof}

\begin{proposition}\label{prop:canonical}
  Assume that \(P\) and \(C\) are Gorenstein.
  Set \(p = \sum p_i\) and let \(\pi \colon X \to P\) be the natural map.
  In \(\Pic X \otimes \Q\), we have
  \[ K_X = \frac{1}{3}\left(\det E + \pi^{*}(p) \right) + K_P\]
\end{proposition}
\begin{proof}
  This is a straight-forward computation using the Riemann--Hurwitz formula.
  %\vainline{Where is the condition that $C$ is Gorenstein used here? Do you want instead $P$ and thus $\mathbb P E$ to be Gorenstein instead?}
  %\chinline{Good point! I think the correct condition is \(\phi \from C \to P\) being a Gorenstein morphism. Anand, can you double check this for us?}
  % We suppress pull-back symbols if they are clear from context.
  % All following equalities hold in the rational Picard group.
  % We have
  % \begin{align*}
  %   K_{\P E} &= -2H + \det E + K_{P},\\
  %   K_{\widehat \P E} &= -2H + \det E + \epsilon + K_P,\\
  %   \widehat C &= 3H-\det E - \epsilon, \text{ and }\\
  %   \widehat F &= \pi^{*}(p) - \epsilon.
  % \end{align*}
  % Since \(\widehat X \to \widehat \P E\) is a cyclic cover branched along \(\widehat C \sqcup \widehat F\), we have
  % \begin{align*}
  %   K_{\widehat X} &= K_{\widehat \P E} + \frac{2}{3} \left(\widehat C + \widehat F\right) \\
  %   &= \frac{1}{3} \left(\pi^{*}(p) - \epsilon\right) + \frac{1}{3}\left(\det E + \pi^{*}(p) \right) + K_P.
  % \end{align*}
  % Observe that \(\pi^{*}(p) - \epsilon = \widehat F\), so the first term above is the exceptional divisor of \(\widehat X \to X\).
  % The statement follows.
\end{proof}

\begin{proposition}\label{prop:k3}
  Suppose \(P = \P^1\), \(C\) is smooth, and \(\deg \br \phi + 2n = 12\).
  Then \(X\) is a smooth K3 surface and \(\sigma\) is a non-symplectic automorphism of \(X\) of order 3.
  The fixed curve of \(\sigma\) is isomorphic to \(C\), and in addition, \(\sigma\) has \(n\) isolated fixed points.
\end{proposition}

\begin{proof}
  Since \(C\) is smooth, so is \(X\).
  Recall that \(2\deg E = \deg \br\phi\), so we have \(\deg E + n = 6\).
  \Cref{prop:canonical} gives \(K_X = 0\) in \(\Pic X \otimes \Q\).
  To see that \(X\) is a K3 surface, it suffices to show that its topological Euler characteristic is \(24\).
  Let \(m = \deg \br\phi\).
  Then
  \[\chi_{\rm top}(C) = 6-m \text{ and } \chi_{\rm top} (\widehat \P E) = 4 + 3n.\]
  Since \(\widehat X \to \widehat \P E\) is a triple cover totally ramified over \(\widehat C \sqcup \widehat F\), we have
  \begin{align*}
    \chi_{\rm top}(\widehat X) &= 3(4+3n - \chi_{\rm top}(C) - \chi_{\rm top}(F)) + \chi_{\rm top}(C) + \chi_{\rm top}(F) \\
                               &= 2m+5n.
  \end{align*}
  Since \(\widehat X \to X\) blows down \(n\) \((-1)\)-curves and \(m + 2n = 2(\deg E + n) = 12\), we have
  \[\chi_{\rm top}(X) = \chi_{\rm top}(\widehat X) - n = 2m+4n = 24,\]
  as desired.
  Since \(\sigma\) contains a curve in its fixed locus, it is non-symplectic.
  Since it arises from a cyclic triple cover, it is of order 3.
  The fixed locus \(\widehat \sigma\) on \(\widehat X\) is the disjoint union of \(\widehat C \simeq C\) and \(\widehat F\).
  The curve \(\widehat C \subset \widehat X\) maps isomorphically to \(\widetilde C \subset X\).
  The curve \(\widehat F\) contracts to \(n\) isolated fixed points in \(X\).
\end{proof}
\begin{remark}\label{rmk:rho_inv_sublattice}
  Let \(k\) be the number of connected components of \(C\).
  The K3 surface \(X\) in \Cref{prop:k3} corresponds to the row \((n,k)\) of \cite[Table~2]{art.sar:08}.
  In particular, the lattice \(H^2(X, \Z)^{\sigma}\) is the lattice \(N=T^\perp\) described there.
  The four maximal families are in \Cref{tab:ASlattices}.
\end{remark}

The triple Tschirnhausen construction works in families.
In particular, applying it to a family of marked triple covers of \(\P^1\) as in \Cref{prop:k3}, we get a family of K3 surfaces with a non-symplectic automorphism of order 3.

\subsubsection{Kulikov models from degenerations of triple covers}\label{sec:kulikov}
Let \(\Delta\) be a smooth curve, possibly non-proper.
Let \(0 \in \Delta\) be a point.
Let \(P \to \Delta\) be a proper map whose fiber over any point of \(\Delta^* = \Delta \setminus 0\) is a smooth curve of genus \(0\) and special fiber \(P_0\) isomorphic to a nodal union \(\P^1 \cup \P^1\) of genus \(0\).
Let \(\phi \colon C \to P\) be a triple cover and let \(p_1, \dots, p_n \colon \Delta \to P\) be disjoint sections lying in the smooth locus of \(P \to \Delta\).
Assume that
\begin{enumerate}
\item\label{etale} \(C \to P\) is \'etale over the node of \(P_0\) and over the sections \(p_1,\dots,p_n\),
\item\label{smooth} \(C \to \Delta\) is smooth, except at points in the pre-image of the node of \(P_0\),
\item\label{balanced} the divisor \(\br \phi + 2 \sum p_i\) has total degree 12 and has degree 6 on each component of \(P_0\).
\end{enumerate}

\begin{proposition}\label{prop:kulikov}
  Suppose \((\phi \colon C \to P, p_1, \dots, p_n)\) satisfies the conditions above.
  Let \(X \to \Delta\) be the family of surfaces obtained by applying the triple Tschirnhausen construction pinched at \(p_1, \dots, p_n\) to the family of marked triple covers \((\phi \colon C \to P,p_1,\dots,p_n)\) over \(\Delta\).
  Then \(X \to \Delta\) is a type II Kulikov degeneration of K3 surfaces.
\end{proposition}
\begin{proof}
  Using conditions \eqref{etale} and \eqref{smooth}, it is easy to see that \(X\) is a non-singular threefold and the central fiber of \(X \to \Delta\) is a simple normal crossings divisor.
  From \Cref{prop:k3}, we know that the generic fiber of \(X \to \Delta\) is a K3 surface.
  Using \eqref{balanced} and \Cref{prop:canonical}, we see that \(K_{X}\) is numerically trivial on the central fiber.
  Since \(K_X\) is trivial on the generic fiber, it follows that \(K_X\) must in fact be trivial.
  Therefore, \(X \to \Delta\) is a Kulikov degeneration.
  The central fiber is the union of two surfaces, so its dual graph is a chain with 2 vertices.
  This is a type II degeneration.
\end{proof}

\begin{remark}\label{rem:ramified-admissible-cover}
  In \Cref{prop:kulikov}, the central fiber of \(C \to P\) is an admissible cover \'etale over the node of the base.
  A version of \Cref{prop:kulikov} also holds for an admissible cover ramified over the node.
  We do not require such degenerations, so we omit the details.
\end{remark}

\begin{remark}\label{rem:automorphism}
  The total space \(X\) of the Kulikov degeneration comes equipped with an automorphism \(\sigma\) of order 3, acting on the fibers of \(X \to \Delta\).
\end{remark}

\subsection{Kulikov surfaces and their lattices of numerically Cartier divisors}\label{sec:lattices}
Recall that \Cref{prop:kulikov} gives Kulikov degenerations of K3 surfaces associated with a degenerating family of marked triple covers.
Let \((\phi \colon C \to P, p_1, \dots, p_n)\) be the central fiber of such a family of marked triple covers.
% More explicitly, take \(P = \P^1 \cup \P^1\) and let \(p_1,\dots, p_n \in P\) be distinct points away from the node of \(P\).
% Let \(\phi \colon C \to P\) be a finite flat morphism of degree 3 \'etale over the node of \(P\) and over the points \(p_1, \dots, p_n\) such that the divisor \(\br \phi + 2 \sum_{i}p_i\) has degree 6 on each component of \(P\).
Let \(X\) be the Kulikov surface obtained from \((\phi \colon C \to P, p_1, \dots, p_n)\) the triple Tschirnhausen construction pinched over \(p_1,\dots,p_n\).
Let \(X = X_0 \cup X_1\) and let \(D = X_0 \cap X_1\) be the double curve.
Recall the lattice of numerically Cartier divisors
\[\widetilde \Lambda = \ker (H^2(X_0, \Z) \oplus H^2(X_1, \Z) \to H^2(D,\Z)),\]
and its quotient
\[ \Lambda = \widetilde \Lambda / \langle(D,-D)\rangle.\]
Since \(X\) admits an automorphism \(\sigma\) of order 3, so does \(\Lambda\).
We have the saturated sublattice \(\Lambda^{\rho} \subset \Lambda\) of \(\rho\)-fixed vectors and its orthogonal complement \(\Lambda^{\rm prim} \subset \Lambda\).
It is easy to check that \(\Lambda^{\rm prim}\) is a negative definite even lattice.

We use the notation introduced for the triple Tschirnhausen construction in \Cref{sec:triple-tschirnhausen}.
Set \(Y = \widehat \P E\) and let \(Y_{0}\) and \(Y_1\) be the components of \(Y\).
For \(i = 0,1\), let \(S_i \subset H^2(Y_i,\Z)\) be the sub-lattice of classes orthogonal to the curves \(\widehat F_1, \dots, \widehat F_n\).
Then the pull-back map \(H^2(Y_i,\Z) \to H^2(\widehat X_i, \Z)\) sends \(S_i\) to \(H^2(X_i,\Z)\).
The lattice \(H^2(X_i,\Z)^{\rho}\) is the saturation of \(S_i\).
Let \(\Lambda_i = H^2(X_i,\Z)^{\rm prim}\) be the orthogonal complement.
Then \(\Lambda_i\) is a negative definite even lattice.
Observe that the class \(D \in H^2(X_i,\Z)\) is the pull-back of the class of a fiber of \(Y_i \to P_i\), and in particular, it lies in the image of \(S_i\).
Therefore, for all \(\alpha \in \Lambda_i\), we have \(\alpha \cdot D = 0\).
As a result, we have a map \(\Lambda_i \to \Lambda^{\rm prim}\).
It is easy to see that the map is injective, and the images of \(\Lambda_0\) and \(\Lambda_1\) are mutually orthogonal.
As a result, we get a map
\[ \Lambda_0 \oplus \Lambda_1 \to \Lambda^{\rm prim}.\]

The following proposition says that the root sublattices of \(\Lambda^{\rm prim}\) and \(\Lambda_0 \oplus \Lambda_1\) agree.
\begin{proposition}\label{prop:root-sublattice-splits}
  The image of \(\Lambda_0 \oplus \Lambda_1\) is a finite index sublattice of \(\Lambda^{\rm prim}\).
  Furthermore, if \(\alpha \in \Lambda^{\rm prim}\) is a root, then \(\alpha\) lies in the image of \(\Lambda_0\) or \(\Lambda_1\).
\end{proposition}
\begin{proof}
  Take \(\alpha \in \Lambda^{\rm prim}\) and write \(\alpha\) as the image of \((a_0,a_1)\) where \(a_i \in H^2(X_i,\Z)\).
  Then \((a_0,a_1)\) is well-defined up to adding multiples of \((D,-D)\).
 
  Let \(y \in H^2(X, \Z)^{\rho}\) be such that \(y \cdot D \neq 0\) (for example, an ample class), and let \(y_i \in H^2(X_i,\Z)\) be its restriction.
  The image of \(y\) in \(\Lambda\) lies in \(\Lambda^{\rho}\), so \((a_0,a_1) y = a_0 \cdot y_0 + a_1 \cdot y_1 = 0\).
  Let \((m,n) = (-y \cdot D, a_0 \cdot y_0)\).
  Then \((ma_0+nD) \cdot y_0 = 0\).
  For any \(x \in H^2(X_0, \Z)^{\rho}\) with \(x \cdot D = 0\), the element \((x,0) \in \Lambda\) is \(\rho\)-fixed, and hence \((a_0,a_1) \cdot (x,0) = a_0 \cdot x = 0\).
  As a result, \((ma_0+nD) \cdot x = 0\).
  It follows that \(ma_0 + nD\) is orthogonal to \(H^2(X_0,\Z)^{\rho}\).
  Likewise, \((ma_1-nD)\) is orthogonal to \(H^2(X_1,\Z)^{\rho}\).
  So \((ma_0,ma_1) + n(D,-D) = m \alpha \in \Lambda\) lies in \(\Lambda_0 \oplus \Lambda_1\). Therefore, \(\Lambda_0 \oplus \Lambda_1 \subset \Lambda^{\rm prim}\) is of finite index.

    Since \(D^2 = 0\) and \(K_{X_i} = -D\), the lattice \(V_i = D^{\perp}_{H^2(X_i,\Z)} / \langle D\rangle\) is even and negative definite.
  Suppose \(\alpha \in \Lambda^{\rm prim}\) is a root, that is, \(\alpha^2 = -2\).
  Write \(\alpha = (a_0,a_1)\) with \(a_i \in H^2(X_i,\Z)\).
  Then, we have seen that \(a_i \cdot D = 0\), so \(a_i\) represents a class in \(V_i\).
  Since \(V_i\) is even and negative definite, and \(a_0^2 + a_1^2 = -2\), we must have \(a_i = 0\) in \(V_i\) for some \(i\); that is, \(a_i \in H^2(X_i, \Z)\) is a multiple of \(D\).
  By changing \((a_0,a_1)\) by adding a multiple of \((D,-D)\), we may assume that \(a_i = 0\).
  Then it follows that \(\alpha\) lies in the image of \(\Lambda_{j}\) for \(j \neq i\).
\end{proof}

We now compute the root sublattice of \(\Lambda\).
By \Cref{prop:root-sublattice-splits}, the root sublattice of \(\Lambda\) is the direct sum of the root sublattices of \(\Lambda_0\) and \(\Lambda_1\).
The lattice \(\Lambda_i\) depends only on the component \(X_i\), which is obtained by the triple Tschirnhausen construction from \(\phi \colon C_i \to P_i\) along with the marked points \(p_j\) that lie on \(P_i\).
So we may focus on the two components \(\phi \colon C_i \to P_i\) individually.

Let \((\phi \colon C \to \P^1,p_1, \dots, p_m)\) be a marked triple cover such that \(\deg \br \phi + 2m = 6\).
In addition to the number \(m\), we have some additional numbers, namely the genera and degrees over \(\P^1\) of the connected comonents of \(C\).
The space of marked triple covers with fixed numerical invariants is irreducible.

\begin{proposition}\label{prop:lattices}
  Let \((\phi \colon C \to \P^1, p_1, \dots, p_m)\) be a general marked triple cover with the numerical invariants specified in the first two columns of the table below.
  Let \(X\) be the surface obtained from it by the triple Tschirnhausen construction pinched at \(p_1,\dots, p_m\) and let \(\rho\) be the order 3 automorphism on \(X\) induced by the triple cover.
  Then \(X\) is a rational surface admitting an elliptic fibration, and the lattice \(H^2(X, \Z)^{\rm prim}\) is given by the last column.

  \begin{center}
    \begin{tabular}{cll}\toprule
      \(m\) & \((\text{Genus}, \text{degree})\) of components of \(C\) & \(H^2(X,\Z)^{\rm prim}\) \\
      \midrule
      \(0\) & \((1,3)\) & \(E_6 \oplus A_2\)\\
      \(0\) & \((0,1)\) and \((2,2)\) & \(E_8\) \\
      \(1\) & \((0,3)\)  & \(A_2^{\oplus 3}\)\\
      \(1\) & \((0,1)\) and \((1,2)\) & \(E_6\) \\
      \(2\) & \((0,1)\) and \((0,2)\) & \(A_2^{\oplus 2}\)\\
      \(3\) & \((0,1)\) and \((0,1)\) and \((0,1)\) & \(A_2\) \\
      \bottomrule
    \end{tabular}
  \end{center}
\end{proposition}
We devote the rest of \Cref{sec:lattices} to the proof of \Cref{prop:lattices}.
The proof is by a case-by-case analysis, which also illuminates the geometry of each \(X\).
Within the proof, we also record whether each component of \(\widetilde C \subset X\) is nef or not; this will be useful for computing stable pairs.
We first make some general remarks and then treat the cases one-by-one.

Let \(V\) be any surface with an order 3 automorphism \(\rho\).
Suppose \(\epsilon \subset V\) is a \(\rho\)-invariant \((-1)\)-curve and let \(V \to \overline V\) be the blow-down of \(\epsilon\).
Then
\[ H^2(V, \Z) = \Z \langle \epsilon \rangle \oplus H^2(\overline V, \Z).\]
Therefore,
\begin{equation}\label{eqn:1bd}
  H^2(V, \Z)^{\rho} = \Z \oplus H^2(\overline V, \Z)^{\rho} \text{ and } H^2(V, \Z)^{\rm prim} = H^2(\overline V, \Z)^{\rm prim}.
\end{equation}
Suppose \(\epsilon_1,\epsilon_2,\epsilon_3 \subset V\) is a set of three mutually disjoint \((-1)\)-curves that form an orbit under \(\langle \rho \rangle\) and let \(V \to \overline V\) be the blow-down of \(\epsilon_1 \cup \epsilon_2 \cup \epsilon_3\).
Then
\[ H^2(V, \Z) = \Z\langle \epsilon_1,\epsilon_2,\epsilon_3\rangle \oplus H^2(\overline V, \Z).\]
Therefore, 
\begin{equation}\label{eqn:3bd}
  H^2(V,\Z)^{\rho} = \Z \oplus H^2(\overline V, \Z)^{\rho} \text{ and } H^2(V,\Z)^{\rm prim} = A_2 \oplus H^2(\overline V, \Z)^{\rm prim},
\end{equation}
where the \(A_2\) is spanned by \(\epsilon_1-\epsilon_2\) and \(\epsilon_2-\epsilon_3\).
Finally, if \(\rk H^2(V, \Z)^{\rho} = 1\) and \(K_V \neq 0\), then
\begin{equation}\label{eqn:kperp}
  H^2(V,\Z)^{\rm prim} = \langle K_V \rangle ^{\perp}.
\end{equation}
We use \eqref{eqn:1bd}, \eqref{eqn:3bd}, and \eqref{eqn:kperp} repeatedly.

Before we begin, we recall our notation.
We denote by \(\F_n\) the \(n\)th Hirzebruch surface with \(\sigma\) the section of self-intersection \(-n\) and \(f\) a fiber.
We denote by \(C \subset \P E\) the Tschirnhausen embedding and by \(F_i \subset \P E\) the fiber of \(\P E\) over the marked point \(p_i\).
We let \(\widehat \P E\) be the blow-up of \(\P E\) at \(\bigsqcup_i C \cap F_i\).
For a curve \(\alpha \in \P E\), we let \(\widehat \alpha \subset \widehat \P E\) be the proper transform.
We have the triple cover \(\widehat X \to \widehat \P E\) branched over \(\widehat C\) and \(\bigsqcup_i \widehat F_i\) and the blow-down \(\widehat X \to X\) along the pre-images of \(\widehat F_i\).
We denote by \(\widetilde \alpha\) the image in \(X\) of the pre-image of \(\widehat \alpha\) in \(\widehat X\).
Note that \(\widetilde \alpha\) may be disconnected.

Let \(D \subset X\) be the pre-image of a general fiber of \(\widehat \P E \to \P^1\).
In the central fiber described in \Cref{prop:kulikov}, the double curve is of this form.
From the triviality of the canonical bundle of the central fiber in \Cref{prop:kulikov}, it follows that \(D \subset X\) is an anti-canonical divisor.
Since \(D\) is the pre-image of a fiber, we also have \(D^2 = 0\).

Recall that \(H^2(X, \Z)^{\rm prim}\) is isomorphic to the orthogonal complement of \(\langle \widehat F_1, \dots, \widehat F_k \rangle\) in \(H^2(\widehat \P E, \Z)\).
It follows that \(\rk H^2(X, \Z)^{\rm prim} = 2 + 2k\).

\subsubsection{Case \(m = 0\) and \(\Gamma = \{(1,3)\}\)}\label{sec:k0-13}
In this case, \(\rk H^2(X, \Z)^{\rho} = 2\).
We have \(\P E = \F_1\) and \(C \subset \F_1\) is of class \(3\sigma+3f\).
The curve \(\widetilde \sigma \subset X\) is the disjoint union of three \((-1)\)-curves, which are in an orbit under \(\langle \rho \rangle\).
Let \(X \to \overline X\) be the blow-down of \(\widetilde \sigma\).
Let \(\overline D \subset \overline X\) be the image of \(D\).
Then \(\overline D\) is an anti-canonical divisor with \(\overline D^2 = 3\).
So \(\overline X\) is a \(\delpezzo_3\).
Using \(\rk H^2(X, \Z)^{\rho} = 2\) and \eqref{eqn:3bd}, we get
\[ \rk H^2(\overline X, \Z)^{\rho} = 1 \text{ and } H^2(X, \Z)^{\rm prim} = A_2 \oplus H^2(\overline X, \Z)^{\rm prim}.\]
Finally, using \eqref{eqn:kperp}, we get
\( H^2(\overline X, \Z)^{\rm prim} = \langle K_{\overline X} \rangle ^{\perp}\),
which is \(E_6\) since \(\overline X\) is a \(\delpezzo_3\).

In this case, \(\widetilde C \subset X\) is nef.
It is zero only on the three \(-1\) curves contracted by \(X \to \overline X\).
See \Cref{fig:k0-13} for a sketch of \(X\) (left), \(\overline X\) (right), \(D\) (blue), and \(C\) (red).
\begin{figure}
  \begin{tikzpicture}[xscale=2,yscale=2]
    \draw[thick] (0,0) -- (0,1) -- (1,1) -- (1,0) -- (0,0);
    \draw[thick,blue]  (0.5,0) -- (0.5,1);
    \draw[dashed]  (0,0.7) -- (1,0.7) (0,0.8) -- (1,0.8) (0,0.9) -- (1,0.9);
    \draw[thick, red] plot[smooth] coordinates {(1,0.6)  (0.35,0.5) (0.65,0.3) (0,0.2)};
    \draw (1.2, 0.5) node (X) {};
    \begin{scope}[xshift=2cm]
      \draw[thick] (0,0) -- (1,0) -- (0.5,1) -- (0,0);
      \draw[thick,blue]  (0.5,0) -- (0.5,1);
      \draw[thick, red] plot[smooth] coordinates {(0.7,0.6)  (0.35,0.5) (0.65,0.3) (0,0)};
      \draw (-0.2, 0.5) node (Xb) {};
      \draw (0.5, 0.1) node[right, blue] {\tiny 3};
      \draw (0.65, 0.3) node[right, red] {\tiny 3};
    \end{scope}
    \draw[->] (X) edge (Xb);
  \end{tikzpicture}
  \caption{In the case \(m = 0\) and \(C=(1,3)\), the surface \(X\) (left) is the blow-up of a \(\delpezzo_3\) \(\overline X\) (right).
  The double curve is blue, the ramification curve is red, and the numbers are self-intersections.
  }\label{fig:k0-13}
\end{figure}
\subsubsection{Case \(m = 0\) and \(\Gamma = \{(0,1), (2,2)\}\)}\label{sec:k0-01-22}
In this case, \(\rk H^2(X, \Z)^{\rho} = 2\).
We have \(\P E = \F_3\) and \(C \subset \F_3\) is the disjoint union of \(\sigma\) and a curve \(Q\) of class \(2 \sigma + 6 f\).
The curve \(\widetilde \sigma \subset X\) is a \(\rho\)-fixed \((-1)\)-curve.
Let \(X \to \overline X\) be its blow-down.
Let \(\overline D \subset \overline X\) be the image of \(D\).
Then \(\overline D\) is an anti-canonical divisor with \(\overline D^2 = 1\).
So \(\overline{X}\) is a \(\delpezzo_1\).
Using \(\rk H^2(X, \Z)^{\rho} = 2\) and \eqref{eqn:1bd}, we get
\[ \rk H^2(\overline{X},\Z)^{\rho} = 1 \text{ and } H^2(X, \Z)^{\rm prim} = H^2(\overline{X}, \Z)^{\rm prim}.\]
Finally, using \eqref{eqn:kperp}, we get
\( H^2(\overline{X},\Z)^{\rm prim} = \left\langle K_{\overline{X}} \right\rangle^{\perp}\),
which is \(E_8\) since \(\overline{X}\) is a \(\delpezzo_1\).

In this case, \(\widetilde \sigma\) is not nef.
But \(\widetilde Q\) is nef and it is zero only on \(\widetilde \sigma\).
See \Cref{fig:k0-01-22} for a sketch of \(X\) (left), \(\overline X\) (right), \(D\) (blue), and \(C\) (red).
\begin{figure}
  \begin{tikzpicture}[xscale=2,yscale=2]
    \draw[thick] (0,0) -- (0,1) -- (1,1) -- (1,0) -- (0,0);
    \draw[thick,blue]  (0.5,0) -- (0.5,1);
   \draw[thick, red, dashed]  (0,0.7) -- (1,0.7);
    \draw[thick, red] plot[smooth, tension=2] coordinates {(0,0.5)  (0.65,0.4) (0,0.3)};
    \draw (1.2, 0.5) node (X) {};
    \begin{scope}[xshift=2cm]
      \draw[thick] (0,0) -- (1,0) -- (0.5,1) -- (0,0);
      \draw[thick,blue]  (0.5,0) -- (0.5,1);
      \draw[thick, red] plot[smooth] coordinates {(0.3,0.6)  (0.65,0.3) (0,0)};
      \draw (-0.2, 0.5) node (Xb) {};
      \draw (0.5, 0.1) node[right, blue] {\tiny 1};
      \draw (0.65, 0.3) node[right, red] {\tiny 4};
    \end{scope}
    \draw[->] (X) edge (Xb);
  \end{tikzpicture}
  \caption{In the case \(m = 0\) and \(C = (0,1) \sqcup (2,2)\), the surface \(X\) (left) is the blow-up of a \(\delpezzo_1\) \(\overline X\) (right).
    The double curve is blue, the ramification curve is red, and the numbers indicate self-intersections.
  }\label{fig:k0-01-22}
\end{figure}

\subsubsection{Case \(m = 1\) and \(\Gamma = \{(0,3)\}\)}\label{sec:k1-03}
In this case, \(\rk H^2(X,\Z)^{\rho} = 4\).
We have \(\P E = \P^1 \times \P^1\) and \(C \subset \P E\) is a curve of class \(3\sigma+f\).
Let \(c_1,c_2,c_3\) be the three points of \(C \cap F_1\).
For \(i = 1,2,3\), let \(L_i \subset \P E\) be the curve of class \(\sigma\) through \(p_i\).
Then \(\widetilde L_i \subset X\) is a disjoint union of three \((-1)\)-curves that form an orbit under \(\langle \rho \rangle\).
Let \(X \to \overline X\) be the blow-down of all \(\widetilde L_i\) for \(i=1,2,3\).
Let \(\overline D \subset \overline X\) be the image of \(D\).
Then \(\overline D\) is an anti-canonical divisor with \(\overline D^2 = 9\).
It follows that \(\overline X \simeq \P^2\).
Using \eqref{eqn:3bd}, we get
\[ H^2(X, \Z)^{\rm prim} = A_2^{\oplus 3} \oplus H_2(\overline X, \Z)^{\rm prim}.\]
But since \(X \simeq \P^2\), we have \(H^2(X, \Z)^{\rm prim} = \langle K_{\overline X} \rangle^{\perp} = 0\).
So \(H^2(X, \Z)^{\rm prim} = A_2^{\oplus 3}\).

Note that \(\widetilde C\) is disjoint from \(\widetilde L_i\) for all \(i = 1,2,3\) and \(\widetilde C^2 = 1\); therefore, the image of \(\widetilde C\) in \(\overline X \simeq \P^2\) is a line.
This in turn implies that \(\widetilde C\) is nef.
See \Cref{fig:k1-03} for a sketch of \(X\) (left), \(\overline X\) (right), \(D\) (blue), and \(C\) (red).
\begin{figure}
    \begin{tikzpicture}[xscale=2,yscale=2]
    \draw[thick] (0,0) -- (0,1) -- (1,1) -- (1,0) -- (0,0);
    \draw[thick,blue]  (0.5,0) -- (0.5,1);
    \draw[dashed]
    (0,0.7) -- (1,0.7) (0,0.72) -- (1,0.72) (0,0.74) -- (1,0.74)
    (0,0.8) -- (1,0.8) (0,0.82) -- (1,0.82)     (0,0.84) -- (1,0.84)
    (0,0.9) -- (1,0.9) (0,0.92) -- (1,0.92)     (0,0.94) -- (1,0.94);
    \draw[thick, red] plot[smooth] coordinates {(1,0.5)  (0.35,0.4) (0.65,0.2) (0,0.1)};
    \draw (1.2, 0.5) node (X) {};
    \begin{scope}[xshift=2cm]
      \draw[thick] (0,0) -- (1,0) -- (0.5,1) -- (0,0);
      \draw[thick,blue]  (0.5,0) -- (0.5,1);
      \draw[thick, red] plot[smooth] coordinates {(0.7,0.6)  (0.35,0.5) (0.65,0.3) (0,0)};
      \draw (-0.2, 0.5) node (Xb) {};
      \draw (0.5, 0.1) node[right, blue] {\tiny 9};
      \draw (0.65, 0.3) node[right, red] {\tiny 1};
    \end{scope}
    \draw[->] (X) edge (Xb);
  \end{tikzpicture}
  \caption{In the case \(m = 1\) and \(C = (0,3)\), the surface \(X\) (left) is the blow-up of \(\P^2\) (right).
   The double curve is blue, the ramification curve is red, and  the numbers indicate self-intersections.
  }\label{fig:k1-03}
\end{figure}
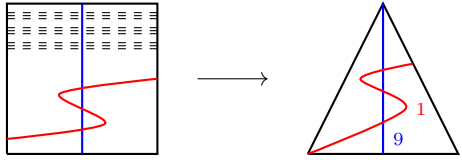

\subsubsection{Case \(m = 1\) and \(\Gamma = \{(0,1), (1,2)\}\)}\label{sec:k1-01-12}
In this case, \(\rk H^2(X,\Z)^{\rho} = 4\).
We have \(\P E = \F_2\) and \(C \subset \P E\) is the disjoint union of \(\sigma\) and a curve \(Q\) of class \(2 \sigma + 4f\).
Let \(c_1 = F \cap \sigma\) and let \(c_2\) be one of the two points of \(Q \cap F\).
For \(i = 1,2\), let \(\epsilon_i \subset \widehat \P E\) be the exceptional divisor over \(c_i\).

On \(X\), the curves \(\widetilde \sigma\), \(\widetilde{\epsilon}_1\), and \(\widetilde{\epsilon}_2\) form a chain of \(\rho\)-invariant smooth rational curves of self-intersections \(-1, -2, -2\).
Let \(X \to \overline X\) be the blow-down of this chain.
Let \(\overline D \subset \overline X\) be the image of \(D\).
Then \(\overline D\) is an anti-canonical divisor with \(\overline D^2 = 3\).
So \(\overline X\) is a \(\delpezzo_3\).
Using that \(\rk H^2(X, \Z)^{\rho} = 4\) and \eqref{eqn:1bd} three times, we get
\[ \rk H^2(\overline X, \Z)^{\rho} = 1 \text{ and } H^2(X, \Z)^{\rm prim} = H^2(\overline{X}, \Z)^{\rm prim}.\]
Using \eqref{eqn:kperp}, we get \(H^2(\overline{X},\Z)^{\rm prim} = \langle K_{\overline{X}}\rangle ^{\perp}\), which is \(E_6\) since \(\overline{X}\) is a \(\delpezzo_3\).

Note that \(\widetilde \sigma\) is not nef.
But \(\widetilde Q\) is nef, and it is zero only on \(\widetilde \sigma\) and the middle curve in the chain.
See \Cref{fig:k1-01-12} for a sketch of \(X\) (left), \(\overline X\) (right), \(D\) (blue), and \(C\) (red).
\begin{figure}
  \begin{tikzpicture}[xscale=2,yscale=2]
        \draw[thick] (0,0) -- (0,1) -- (1,1) -- (1,0) -- (0,0);
    \draw[thick,blue]  (0.5,0) -- (0.5,1);
   \draw[thick, red, dashed]  (0,0.7) -- (1,0.7);
   \draw[thick, red] plot[smooth, tension=2] coordinates {(0,0.3)  (0.65,0.2) (0,0.1)};
   \draw[thick, dashed] (0.3,0.8) -- (0.15, 0.4) (0.15, 0.6) -- (0.3,0.2);
    \draw (1.2, 0.5) node (X) {};
    \begin{scope}[xshift=2cm]
      \draw[thick] (0,0) -- (1,0) -- (0.5,1) -- (0,0);
      \draw[thick,blue]  (0.5,0) -- (0.5,1);
      \draw[thick, red] plot[smooth] coordinates {(0.7,0.6)  (0.35,0.5) (0.65,0.3) (0,0)};
      \draw (-0.2, 0.5) node (Xb) {};
      \draw (0.5, 0.1) node[right, blue] {\tiny 3};
      \draw (0.65, 0.3) node[right, red] {\tiny 3};
    \end{scope}
    \draw[->] (X) edge (Xb);
  \end{tikzpicture}
  \caption{In the case \(m = 1\) and \(C = (0,1) \sqcup (1,2)\), the surface \(X\) (left) is the blow-up of a \(\delpezzo_3\) \(\overline X\) (right).
   The double curve is blue, the ramification curve is red, and  the numbers indicate self-intersections.
  }\label{fig:k1-01-12}
\end{figure}

\subsubsection{Case \(m = 2\) and \(\Gamma = \{(0,1), (0,2)\}\)}\label{sec:k2-01-02}
In this case, \(\rk H^2(X, \Z)^{\rho} = 6\).
We have \(\P E = \F_1\) and \(C \subset \P E\) is the disjoint union of \(\sigma\) and a curve \(Q\) of class \(2 \sigma + 2f\).
For \(i = 1,2\), let \(F_i \cap Q = \{q_{i,1}, q_{i,2}\}\).
For \(j = 1,2\), let \(L_j\) the unique curve of class \(\sigma+f\) through \(q_{1,2}\) and \(q_{2,j}\).
Let \(\epsilon_1 \subset \widehat \P E\) be the exceptional divisor over \(\sigma \cap F_1\) and \(\epsilon_2 \subset \widehat \P E\) the exceptional divisor over \(q_{1,1}\).
The curves introduced so far have the following dual graph in \(\widehat \P E\)
\[
  \begin{tikzpicture}[curve/.style={draw,circle,inner sep=0.2em}]
    \draw
    (0,0) node[curve] (s) {} node[above=0.5em] {\(\widehat \sigma\)}
    (1,0) node[curve] (e1) {} node[above=0.5em] {\(\epsilon_1\)}
    (2,0) node[curve] (f1) {} node[above=0.5em] {\(\widehat F_1\)}
    (3,0) node[curve] (e2) {} node[above=0.5em] {\(\epsilon_2\)}
    (4,0) node[curve] (q) {} node[above=0.5em] {\(\widehat Q\)}
    (0,-1) node[curve] (f2) {} node[below=0.5em] {\(\widehat F_2\)}
    (1,-1) node[curve] (l1) {} node[below=0.5em] {\(\widehat L_1\)}
    (2,-1) node[curve] (l2) {} node[below=0.5em] {\(\widehat L_2\)};
    \draw
    (s) edge (e1)
    (q) edge (e2)
    (f1) edge (e1)
    (f1) edge (e2);
 \end{tikzpicture}
\]
On \(X\), the curves \(\widetilde \sigma\), \(\widetilde \epsilon_1\), \(\widetilde \epsilon_2\) form a chain of \(\rho\)-invariant smooth rational curves of self-intersections \(-1,-2,-2\).
For \(j= 1,2\), the curve \(\widetilde L_j\) is the disjoint union of three \((-1)\)-curves that form an orbit under \(\langle \rho \rangle\).
The curves \(\widetilde L_1\) and \(\widetilde L_2\) are disjoint from each other and also from the chain \(\widetilde \sigma\), \(\widetilde \epsilon_{1}\), \(\widetilde \epsilon_2\).
Let \(X \to \overline X\) be the blow-down of \(\widetilde L_1\), \(\widetilde L_2\), \(\widetilde \sigma\), \(\widetilde \epsilon_{1}\), and \(\widetilde \epsilon_2\).
Let \(\overline D \subset \overline X\) be the image of \(D \subset X\).
Then \(\overline D\) is an anti-canonical divisor with \(\overline D^2 = 9\).
So \(\overline X \simeq \P^2\).
Using \eqref{eqn:1bd} and \eqref{eqn:3bd}, we get
\[ H^2(X, \Z)^{\rm prim} = A_2^{\oplus 2} \oplus H^2(\overline X, \Z)^{\rm prim}.\]
Since \(\overline X \simeq \P^2\), we have \(H^2(\overline X, \Z)^{\rm prim} = \langle K_{\overline X} \rangle^{\perp} = 0\).

Note that \(\widetilde \sigma\) is not nef.
But \(\widetilde Q\) is nef and it is zero on \(\widetilde Q\) itself, \(\widetilde \sigma\), the middle curve in the chain contracted in \(\overline X\), and the \(-1\) curves contracted in \(\overline X\).
As \(\widetilde Q^2 = 0\), \(\widetilde Q\) is not a big divisor.
See \Cref{fig:k1-01-02} for a sketch of \(X\) (left), \(\overline X\) (right), \(D\) (blue), and \(C\) (red).
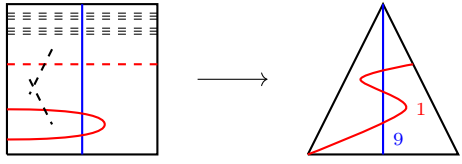
\begin{figure}
  \begin{tikzpicture}[xscale=2,yscale=2]
        \draw[thick] (0,0) -- (0,1) -- (1,1) -- (1,0) -- (0,0);
    \draw[thick,blue]  (0.5,0) -- (0.5,1);
   \draw[thick, red, dashed]  (0,0.6) -- (1,0.6);
   \draw[thick, red] plot[smooth, tension=2] coordinates {(0,0.3)  (0.65,0.2) (0,0.1)};
   \draw[thick, dashed] (0.3,0.7) -- (0.15, 0.4) (0.15, 0.5) -- (0.3,0.2);
   \draw[dashed]
    (0,0.8) -- (1,0.8) (0,0.82) -- (1,0.82)     (0,0.84) -- (1,0.84)
    (0,0.9) -- (1,0.9) (0,0.92) -- (1,0.92)     (0,0.94) -- (1,0.94);
    \draw (1.2, 0.5) node (X) {};
    \begin{scope}[xshift=2cm]
      \draw[thick] (0,0) -- (1,0) -- (0.5,1) -- (0,0);
      \draw[thick,blue]  (0.5,0) -- (0.5,1);
      \draw[thick, red] plot[smooth] coordinates {(0.7,0.6)  (0.35,0.5) (0.65,0.3) (0,0)};
      \draw (-0.2, 0.5) node (Xb) {};
      \draw (0.5, 0.1) node[right, blue] {\tiny 9};
      \draw (0.65, 0.3) node[right, red] {\tiny 1};
    \end{scope}
    \draw[->] (X) edge (Xb);
  \end{tikzpicture}
  \caption{The surface \(X\) (left) as the blow-up of \(\P^{2}\) (right), together with the double curve (blue) and the ramification curve (red).
  The numbers indicate self-intersection.
  }\label{fig:k1-01-02}
\end{figure}

\subsubsection{Case \(m = 3\) and \(\Gamma = \{(0,1), (0,1), (0,1)\}\)}
In this case, \(\rk H^2(X, \Z)^{\rho} = 8\).
We have \(\P E = \P^1 \times \P^1\) and \(C \subset \P E\) is the disjoint union of three curves of class \(\sigma\), say \(C = C_1 \sqcup C_2 \sqcup C_3\).
For \(i = 1,2,3\) and \(j = 1,2,3\), set \(q_{i,j} = C_i \cap F_j\) and let \(\epsilon_{i,j} \subset \widehat \P E\) be the exceptional divisor over \(q_{i,j}\).
Let \(Q \subset \P E\) be the unique curve of class \(\sigma + f\) through \(q_{1,2}, q_{2,3}\), and \(q_{3,1}\).
For \(i = 1,2,3\), the curves \(\widehat C_i\) and \(\widehat{\epsilon}_{i,i}\) form a chain of \(\rho\)-invariant smooth rational curves of self-intersections \(-1\) and \(-2\).
These three chains are mutually disjoint.
Furthermore, \(\widehat Q\) is the disjoint union of three \((-1)\)-curves that form an orbit of \(\langle \rho \rangle\), disjoint from \(\widehat C_i\) and \(\widehat{\epsilon}_{i,i}\).
Let \(X \to \overline X\) be the blow-down of the chains \(\widehat C_i\) and \(\widehat{\epsilon}_{i,i}\) for \(i = 1,2,3\) and of the three \((-1)\)-curves \(\widehat Q\).
Let \(\overline D \subset \overline X\) be the image of \(D \subset X\).
Then \(\overline D\) is an anti-canonical divisor with \(\overline D^2 = 0\).
So \(\overline X \simeq \P^2\).
Using \eqref{eqn:1bd} and \eqref{eqn:3bd}, we get
\[ H^2(X, \Z)^{\rm prim} = A_2 \oplus H^2(\overline X, \Z)^{\rm prim}.\]
Since \(\overline X \simeq \P^2\), we have \(H^2(\overline X, \Z)^{\rm prim} = \langle K_{\overline X} \rangle^{\perp} = 0\).

In this case, none of the components of \(C\) is nef in \(X\) as their self-intersections are all equal to \(-1\).

\subsection{Periods of surfaces with an anti-canonical cycle}\label{sec:dominance}
Fix a row in the table in \Cref{prop:lattices}.
Let \(m\), \(\Gamma\) and \(\Lambda\) be the entries in the first, second, and third columns of this row.

Let \(\mathcal H\) be the coarse moduli space of
\[ (\phi \colon C \to P, p_1, \dots, p_m, q)\]
where
\begin{itemize}
\item \(P \simeq \P^1\);
\item \(C\) is a smooth curve, not necessarily connected, and \(\phi \colon C \to P\) is a finite map of degree 3; the genera of the connected components of \(C\) and their degrees over \(P\) are given by \(\Gamma\);
\item \(p_1, \dots, p_m\) and \(q\) are pairwise distinct marked points on \(P\) over which \(\phi\) is \'etale.
\end{itemize}
Let \(b\) be the degree of the branch divisor of \(\phi\), and observe from \Cref{prop:lattices} that \(b = 6-2m\).
Note that \(\mathcal H\) is an example of a Hurwitz space.
By standard results on Hurwitz spaces \cite[\S~4]{rom.wew:06}, it follows that \(\mathcal H\) is an irreducible quasi-projective variety of dimension \(b + m + 1 - 3 = 4-m\).

Let \(E\) be the unique elliptic curve of \(j\)-invariant \(0\), so that \(\Aut(E,0) \simeq \Z/6\Z\).
The automorphism of \(E\) that acts by multiplication by \(\zeta_3\) on the universal cover makes \(E\) a \(\Z[\zeta_3]\)-module.

Let \(X\) be the surface obtained from \((\phi \colon C \to P, p_1, \dots, p_m)\) by the triple Tschirnhausen construction pinched over \(p_1, \dots, p_m\).
Let \(D \subset X\) be the fiber of \(X \to P\) over \(q\).
Then we have an isomorphism \(\Jac(D) \simeq E\).
Let \(\sigma\) be the automorphism on \(X\) induced by the triple cover construction normalised so that \(\sigma\) corresponds to the multiplication by \(\zeta_3\) on the universal cover of \(\Jac(D)\).
Via the map induced by \(\sigma\), the lattice \(H^2(X, \Z)^{\rm prim}\) becomes a \(\Z[\zeta_3]\)-module.

In our case, the surface \(X\) is a rational surface.
So we identify \(H^2(X, \Z) = \Pic(X)\) and think of \(H^2(X, \Z)^{\rm prim}\) as a subset of \(\Pic(X)\).
We have a map of \(\Z[\zeta_3]\)-modules
\[ \psi \colon H^2(X, \Z)^{\rm prim} \to \Jac(D)\]
given by
\[ \delta \mapsto \delta|_D.\]
After choosing identifications of \(\Z[\zeta_3]\)-modules \(H^2(X, \Z)^{\rm prim} \simeq \Lambda\) and \(\Jac(D) \simeq E\), which are unique up to the action of the finite group \(\Aut(\Lambda) \times \Aut(E,0)\), we can think of \(\psi\) as an element of \(\Hom_{\Z[\zeta_3]}(\Lambda, E)\).
The association \((\phi, p_1, \dots, p_m,q) \mapsto \psi\) yields a regular map
\[ \Psi \colon \mathcal H \to \Hom_{\Z[\zeta_3]}(\Lambda, E) / \Aut(\Lambda) \times \Aut(E,0).\]

\begin{proposition} \label{prop:boundary-period-dominant}
  The map \(\Psi\) defined above is dominant.
\end{proposition}
The proof of \Cref{prop:boundary-period-dominant} follows essentially from the Torelli theorem for rational surfaces with an anti-canonical cycle \cite[\S~5]{loo:81}.
We include an exposition for the lack of a good reference.

Fix \(d \in \{1, \dots, 6\}\).
Consider the lattice \({\rm I}_{1,9-d}\) with orthogonal basis vectors \(\ell, e_1, \dots, e_{9-d}\) of self-intersections \(1, -1, \dots, -1\).
Let \(\Lambda_{d} \subset {\rm I}_{1,9-d}\) be the orthogonal complement of \(3\ell - \sum e_i\).
For \(9-d = 3,4,5,6,7,8\), the lattice \(\Lambda_d\) is the root lattice \(A_2\oplus A_1,A_4, D_5, E_6, E_7, E_8\), respectively.

Let \(X\) be a smooth del Pezzo surface of degree \(d\).
Suppose \(X\) is obtained as the blow-up of \(\P^2\) at \(p_1, \dots, p_{9-d}\).
Let \(E_i\) be the exceptional divisor over \(p_i\) and \(H\) the pull-back of the hyperplane class from \(\P^2\).
Then we get the isomorphism
\[ {\rm I}_{1,9-d} \xrightarrow{\sim} H^2(X, \Z) \quad \ell \mapsto H \quad e_i \mapsto [E_i]\]
and the induced isomorphism
\[ \Lambda_d \xrightarrow{\sim} K_X^{\perp}.\]
Via this isomorphism, we identify \(\Lambda_d\) and \(K_X^{\perp}\).

Let \(D \subset X\) be a smooth anti-canonical divisor.
We call \((X,D)\) an \emph{anti-canonical pair}.
We have a map
\[ \psi \colon K_X^{\perp} \to \Jac(D) \qquad \delta \mapsto \delta|_D,\]
which we call the \emph{period} of the pair.
\begin{theorem}\label{prop:dptor}
  Let \((X_1, D_1)\) and \((X_2,D_2)\) be anti-canonical pairs with periods \(\psi_1\) and \(\psi_2\).
  Suppose we are given isomorphisms \(D_1 \to D_2\) and \(K_{X_2}^{\perp} \to K_{X_1}^{\perp}\) such that the following diagram commutes
  \[
    \begin{tikzcd}
      K_{X_1}^{\perp} \ar[<-]{r}\ar{d}{\psi_1} &K_{X_2}^{\perp}\ar{d}{\psi_2}\\
      \Jac(D_1) \ar{r} &\Jac(D_2).
    \end{tikzcd}
  \]
  Then there is a unique isomorphism \((X_1,D_1) \to (X_2,D_2)\) inducing the given isomorphisms.
\end{theorem}
\begin{proof}
  Let \((X,D) = (X_1,D_1)\) and fix an isomorphism \(K_X^{\perp} \cong \Lambda_d\).
  We recall how to reconstruct \((X,D)\) from \(\psi \colon \Lambda_d \to \Jac(D)\) (see \cite[\S~6]{mcm:07}).
  We choose an arbitrary \(p_1 \in D\) and for \(i = 1, \cdots, 9-d\), define \(p_i \in D\) by
  \[p_i = p_1 + \psi(e_i-e_1).\]
  We take the divisor class \(h\) on \(D\) of degree \(3\) given by
  \[ h = \psi(\ell-e_1-e_2-e_3) + p_1+p_2+p_3,\]
  and embed \(D \subset \P^2\) by the complete linear series \(|h|\).
  Then we have a unique isomorphism of \(X\) with the blow-up of \(\P^2\) at the points \(p_1, \dots, p_{9-d}\), compatible with \(K_X^{\perp} \cong \Lambda_d\).

  Considering \((X,D) = (X_2,D_2)\), we likewise obtain an isomorphism of \(X\) with \(\P^2\) blown up at the same \(9-d\) points, and hence an isomorphism \((X_1,D_1) \to (X_{2},D_2)\) as required.
\end{proof}

We are now ready to prove \Cref{prop:boundary-period-dominant}.
\begin{proof}[Proof of \Cref{prop:boundary-period-dominant}]
  Note that \(\dim \mathcal H = 4-m\).
  Observe that \(\Lambda\) is a free \(\Z\)-module of rank \(8-2m\), and hence a free \(\Z[\zeta_3]\)-module of rank \(4-m\).
  Therefore, 
  \[ \Hom_{\Z[\zeta_3]}(\Lambda, E) \cong E^{4-m}.\]
  So the source and target of \(\Psi\) have the same dimension \(4-m\).
  Therefore, to prove that \(\Psi\) is dominant, it suffices to prove that generic fibers of \(\Psi\) are finite.

  Let \((\phi \colon C \to P, p_1, \dots, p_m, q) \in \mathcal H\) be generic.
  Let \((X,D)\) be the anti-canonical pair associated to it by the triple Tschirnhausen construction.
  Fix isomorphisms \(H^2(X, \Z)^{\rm prim} \cong \Lambda\) and \(\Jac(D) \cong E\) and let \(\psi \colon \Lambda \to E\) be the period of \((X,D)\).

  Let \((\phi' \colon C' \to P', p'_1, \dots, p'_m, q') \in \mathcal H\) be another point, and let \((X',D')\) be the anti-canonical pair associated to it.
  Suppose \((X',D')\) also has the same period \(\psi\).

  We first show that a \((\Z/3\Z)\)-equivariant isomorphism \((X,D) \cong (X',D')\) yields an isomorphism  \((\phi, p_1, \dots, p_m, q) \cong (\phi', p'_1,\dots, p'_m, q')\), up to reordering \(p_1, \dots, p_m\).
  Let \(\iota \colon (X,D) \to (X',D')\) be a \((\Z/3\Z)\)-equivariant isomorphism.
  Since \(\iota\) takes \(D\) to \(D'\), which define elliptic fibrations, \(\iota\) induces an isomorphism \(P \cong P'\) such that the diagram
  \[
    \begin{tikzcd}
      X \ar{r}{\iota}\ar{d} & X' \ar{d}\\
      P \ar{r} & P'
    \end{tikzcd}
  \]
  commutes, and such that \(P \to P'\) maps \(q\) to \(q'\).
  Restricting \(\iota\) to the divisorial components of the fixed loci yields an isomorphism \(C \to C'\) compatible with \(P \to P'\).
  By restricting \(\iota\) to the isolated fixed points shows that the isomorphism \(P \to P'\) takes the marked points \(p_1,\dots,p_m\) to a reordering of \(p'_{1}, \dots, p'_m\).

  To  finish the proof of \Cref{prop:boundary-period-dominant}, it remains to show that there are finitely many \((X',D')\) up to a \((\Z/3\Z)\)-equivariant isomorphism that have the same period \(\psi\).
  
  The proof of \Cref{prop:lattices} yields a \((\Z/3\Z)\)-equivariant birational morphism
  \[ (X', D) \to (\overline X', \overline D')\]
  where
  \begin{enumerate}
  \item \(\overline X'\) is a smooth del-Pezzo surface and \(\overline D' \subset \overline X'\) is a smooth anti-canonical divisor,
  \item \(X' \to \overline X'\) is a sequence of blow-ups of points in \(\overline D'\) or its proper transforms; the center of the blow-up is either a \(\Z/3\Z\) fixed point or a triple of points forming an orbit under \(\Z/3\Z\),
  \item the pull-back map gives an embedding \(K_{\overline X'}^{\perp} \to H^2(X', \Z)^{\rm prim}\) as a direct summand.
  \end{enumerate}

  Let \(d\) be the degree of \(\overline X'\).
  Note that there are only finitely many embeddings of \(K_{\overline X'}^{\perp} \to \Lambda\).
  So, having fixed \(\psi \colon \Lambda \to E\), there are only finitely many possible periods \(K_{\overline X'}^{\perp} \to E\).
  For \(d \leq 6\), \Cref{prop:dptor} implies that there are finitely many \((\overline X', \overline D')\) up to a \((Z/3\Z)\)-equivariant isomorphism whose period is compatible with \(\psi\).
  The only remaining possibility is \(d = 9\), for which the finiteness is automatic (indeed, there are only finitely many \((\Z/3\Z)\)-equivariant embeddings of \(E \subset \P^2\)).

  Since there are finitely many \(\Z/3\Z\) fixed points on \(\overline D' \cong E\), there are only finitely many blow-ups of \(\overline X'\) in \(\Z/3\Z\)-fixed points of \(\overline D'\).
  The same argument applies to any further blow-ups.

  Let \(\{p_1,p_2,p_3\} \subset \overline D'\) be an orbit under \(\Z/3\Z\).
  Let \(E_i\) be the exceptional divisor over \(p_i\) in of the blow-up of this orbit.
  Then we obtain an \(A_2\)-summand in the primitive Picard lattice of the blow-up spanned by \([E_i]-[E_j]\).
  The period map sends \([E_i]-[E_j]\) to \(p_i-p_j\).
  Now, there are only finitely many embeddings \(A_2 \to \Lambda\).
  So, having fixed \(\psi \colon \Lambda \to E\), there are only finitely many possibilities for \(\psi([E_i]-[E_j])\).
  As a result, there are only finitely many orbits \(\{p_1,p_2,p_3\}\) whose blow-up has a period map compatible with \(\psi\).
  Henceforth, there are finitely many \((X', D')\) up to a \((\Z/3\Z)\)-equivariant isomorphism whose period is equal to \(\psi\).
  %The proof is now complete.
\end{proof}

\subsection{Dominance of the triple Tschirnhausen map} \label{subsec:ell_fib_trigonal}
Fix a pair \((n,k)\) of non-negative integers from \Cref{tab:ASlattices}.
By \cite[Theorem 3.3]{art.sar:08}, there is a (unique upto conjugation) \(\rho \in O(L)\) with \(S_{\rho} = S(n,k)\) and \(T_{\rho} = T(n,k)\).
The moduli space \(F_{\rho}\) of \(\rho\)-markable K3 surfaces \((X,\sigma)\) is irreducible of dimension \(9-n\).

Set \(g = g(n,k)\) from \Cref{tab:ASlattices}.
Let \(\mathcal H\) be the moduli space of marked simply branched triple covers \((C \to P, p_1, \dots, p_n)\) where
\begin{itemize}
\item \(P \cong \P^1\) and \(C\) is a smooth curve of genus \(g\) (except in the case \((n,k) = (0,2)\), where we take \(C\) to be the disjoint union of a genus 2 curve and a copy of \(P\)),
\item \(p_1, \dots, p_{n} \in P\) are marked points over which \(\phi\) is \'etale.
\end{itemize}
The space \(\mathcal H\) is again an example of a Hurwitz space, and it is easy to see that it is an irreducible quasi-projective variety of of dimension \(9-n\).
The triple Tschirnhausen construction gives a rational map
\( \Psi \colon \mathcal H \dashrightarrow F^{\rm sep}_{\rho}\).
\begin{proposition}\label{prop:triple-tschirnhausen-dominant}
  In the setup above, the map \(\Psi\) is dominant.
  That is, for \((n,k)\) in \Cref{tab:ASlattices}, a generic K3 with an automorphism of degree \(3\) with \(n\) fixed points and \(k\) fixed curves arises from the triple Tschirnhausen construction.
\end{proposition}
\begin{proof}
  We give two proofs, one direct (sketched here) and one by degeneration (\Cref{rem:dominance}).
  Let \((X,\sigma) \in F^{\rm sep}_{\rho}\) be very general.
  Then \cite[Lemma~3.3]{tak:11} yields an elliptic fibration \(\pi \colon X \to \P^1\) with \(n\) reducible fibers whose dual graph is the affine Dynkin diagram \(\widetilde A_2\).
  One checks that \(\sigma\) acts along the fibers of \(\pi\).
  The quotient by \(\sigma\) together with the fixed locus recovers the triple cover in its Tschirnhausen embedding.
  The marked points correspond to the location of the singular fibers.
\end{proof}

Let \(\overline{\mathcal H}\) be the compactification of \(\mathcal H\) by (marked) admissible covers, following \cite{har.mum:82}.
The boundary points of \(\overline{\mathcal H}\) are marked triple covers \((\phi \colon C \to P, p_1, \dots, p_{n})\) where \((P, \br \phi + p_1 + \dots + p_n)\) is a stable pointed rational curve and \(\phi\) is a triple cover with admissible ramification over the nodes of \(P\).
General points of the boundary divisors of \(\overline{\mathcal H}\) correspond to \(P \cong \P^1 \cup \P^1\).
The map \(\Psi\) yields rational maps from \(\overline {\mathcal H}\) to \(\overline F_{\rho}^{\rm KSBA}\) and \(\overline{\D_{\rho}/\Gamma_{\rho}}^{\rm tor}\).

Let \(\Delta \subset \overline{\mathcal H}\) be a boundary divisor whose generic point \(t = (\phi \colon C \to P, p_1, \dots, p_n)\) satisfies
\begin{itemize}
\item \(\phi\) is \'etale over the node of \(P = \P^{1} \cup \P^1\),
\item \(\br \phi + 2 \sum p_i\) has degree 6 on each component of \(P\).
\end{itemize}
Let \(X\) be the type II Kulikov surface associated to \(t\) by the triple Tschirnhausen construction (see \Cref{prop:kulikov}).
Let \(J \subset T_{\rho}\) be the corresponding cusp.
The properties of the extended period map \eqref{eqn:extendedperiod} allows us to explicitly describe
\begin{equation}\label{eqn:admtorboundary}
  \Psi \colon \Delta \dashrightarrow \overline {\D_{\rho}/\Gamma_{\rho}}^{\rm tor}.
\end{equation}
Let \(E\) be the elliptic curve with an order 3 automorphism, \(\Lambda\) the reduced lattice of numerically Cartier divisors on \(X\) (isometric to \(J^{\perp}_{T_{\rho}}/J\) by \eqref{eqn:JperpJrho}), and \(\psi \in \Lambda^{\rm prim} \otimes_{\Z[\zeta_3]} E \simeq \mathcal{A}_J\) the period of \(X\).
Let \(\delta\) be the boundary component of \(\overline{\D_\rho/\Gamma_\rho}^{\rm tor}\) corresponding to the cusp \(J\); recall from \Cref{subsec:semitoroidal} that \(\delta\) is \(\mathcal{A}_J/\widehat{\Gamma}_J\) for some finite group \(\widehat{\Gamma}_J\).
Then the map \(\Psi\) sends \(t\) to the \(\widehat{\Gamma}_J\)-orbit of \(\psi\).
\begin{proposition}\label{prop:dominanceonboundary}
  In the setup above, \(\Psi\) maps \(\Delta\) dominantly onto \(\delta \subset \overline{\D_\rho/\Gamma_\rho}^{\rm tor}\).
\end{proposition}
\begin{proof}
  Let \(P_0,P_1\) be irreducible componets of \(P\), and let \(q\) be the node of \(P\).
  Call \(X_i\) the result of triple Tschirnhausen construction on \(t_i \colonequals (\phi_i \from \phi^{-1}(P_i) \to P_i, \{p_1,\dotsc,p_n\} \cap P_i,q)\).
  Then \(X\) has two irreducible components \(X_0\) and \(X_1\).
  Let \(\mathcal{H}_i\) be the coarse moduli space from \Cref{sec:dominance} that contains \(t_i\).
  We have a generically finite rational map \(\Delta \dashrightarrow \mathcal{H}_0 \times \mathcal{H}_1\) that sends \(t\) to \((t_0,t_1)\).
  
  By \Cref{prop:root-sublattice-splits}, we have the finite index sublattice \[H^2(X_0, \Z)^{\rm prim} \oplus H^2(X_1, \Z)^{\rm prim} \subset \Lambda^{\rm prim},\]
  which induces the following isogeny
  \[ \mathcal{A}_J \simeq \Hom_{\Z[\zeta_3]}((\Lambda^{\rm prim})^*,E) \twoheadrightarrow \Hom_{\Z[\zeta_3]}(H^2(X_0, \Z)^{\rm prim} \oplus H^2(X_1, \Z)^{\rm prim},E) \]
  sending the period of \(X\) to the pair of periods of anticanonical pairs \((X_0,D_{01})\) and \((X_1,D_{10})\).
  So it suffices (by the isogeny above) to prove the dominance of
  \[\mathcal{H}_0 \times \mathcal{H}_1 \to \prod_{i=0}^1 \Hom_{\Z[\zeta_3]}(H^2(X_i, \Z)^{\rm prim}, E) / \Aut(H^2(X_i, \Z)^{\rm prim}) \times \Aut(E,0),\]
  which follows from \Cref{prop:boundary-period-dominant}.
\end{proof}
\begin{remark}\label{rem:dominance}
  The dominance of \(\Psi\) on the boundary (\Cref{prop:dominanceonboundary}) implies dominance on the interior (\Cref{prop:triple-tschirnhausen-dominant}).
\end{remark}

\section{Moduli space 1: \texorpdfstring{\(g=5\)}{g=5} and \texorpdfstring{\( (n,k) = (0,2) \)}{(n,k) = (0,2)}} \label{sec:n0k2}

In this section, we identify the KSBA compactification of the moduli space of K3 surfaces with a non-symplectic automorphism of order 3 with \(0\) isolated fixed points and \(2\) fixed curve (so \(n = 0\) and \(k = 2\)).

Let \((X, \sigma)\) be a generic K3 surface with an automorphism of order \(3\) with \(n = 0\) and \(k = 2\).
From \Cref{prop:triple-tschirnhausen-dominant}, we know that \((X,\sigma)\) arises from the triple Tschirnhausen construction applied to \((\phi \colon C \to \P^1)\), where \(C = \P^1 \sqcup Q\) with \(Q\) a smooth curve of genus \(5\) with \(\deg \phi|_Q = 2\).
Observe that, for a generic \(\phi\), the Tschirnhausen embedding identifies the \(\P^1\) component of \(C\) as the \((-6)\)-section \(\alpha\) of \(\F_6\) and \(Q\) as a divisor on \(\F_6\) of class \(2\alpha + 6f\) disjoint from \(\alpha\), where \(f\) is the class of a fiber of \(\F_6\).  The K3 surface \(X\) is simply the cyclic triple cover of \(\F_6\) branched over curves of classes \(\alpha\) and \(2\alpha + 6f\).
Artebani--Sarti give the same description using explicit equations \cite[Proposition~4.7]{art.sar:08}.

% \begin{remark}
%   In \cite[Proposition~4.2]{art.sar:08}, Artebani--Sarti gave explicit projective equations for \(X\) as a Jacobian elliptic fibration
%   \begin{equation}\label{eqn:artebani-sarti-02}
%     y^2 = x^3 + p_{12}(t),
%   \end{equation}
%   where \(p_{12}\) is a polynomial of degree \(12\) with no multiple roots.
%   Then the \(\Z/3\Z\)-fixed locus of such Jacobian elliptic fibration \(Y\) is a disjoint union of a section \(\widetilde \alpha\) (having self-intersection \(-2\)) and a smooth bisection \(\widetilde Q\) of genus \(5\).
%   Thus, the \(\Z/3\Z\)-quotient of \(Y\) is \(\F_6\), which agrees with the triple Tschirnhausen construction described above.
% \end{remark}

Fix an isometry of \(H^2(X, \Z)\) with the K3 lattice \(L\), and let \(\rho\) be the automorphism of \(L\) induced by \(\sigma\).
We recall the Hodge type of \(\rho\) from \Cref{tab:ASlattices}.
Recall that \(S_{\rho} \subset L\) is the sub-lattice of vectors fixed by \(\rho\) and \(T_{\rho} = S_{\rho}^{\perp}\).
Choose a basis \(\alpha+3f,f\) of \(H^2(\F_6, \Z) \simeq U\).
Denoting \(\widetilde \alpha\) to be the reduced preimage in \(X\) of \(\alpha \in H^2(\F_6, \Z)\), notice that \(3 \widetilde \alpha = \alpha\).
As \(\rho\) is identified with \(\sigma^*\), the pull-back of the cyclic triple cover \(X \to \F_6\) gives an identification
\[ S_{\rho} = H^2(\F_6, \Z) (3)^{\rm sat} = U \]
via the generators \(\widetilde \alpha, f \in H^2(X,\Z)\).
The lattice \(T_{\rho}\) turns out to be 
\[ T_{\rho} \simeq U \oplus U \oplus E_8 \oplus E_8.\]
The automorphism \(\rho\) acts on the summands \(U \oplus U\) and \(E_8\) and \(E_8\).
For \(U \oplus U\), the action is described in \Cref{subsec:eis_cusp} and for \(E_8\) it is described in \cite[Ch~II~\S~2.6]{con.slo:99}.

\subsection{Baily--Borel cusps} \label{subsec:BB_02}
Let \(F_\rho^{\rm sep} = (\D_{\rho} \setminus \Delta_\rho)/\Gamma_{\rho}\) be the period domain for \(\rho\)-markable K3 surfaces as described in \Cref{subsec:period_map}.
Recall from \Cref{thm:t11-t21cusps} that there is exactly one cusp of \(\overline{\D_\rho/\Gamma_\rho}^{\rm BB}\); choose a representative \(J \subset T_\rho\), which is a \(\rho\)-invariant isotropic plane. 
As a result, the boundaries of \(\overline{\D_\rho/\Gamma_\rho}^{\rm tor}\) and \(\overline{\D_\rho/\Gamma_\rho}^{\mathfrak F} = \overline{F}_\rho^{\rm KSBA}\) are irreducible.
Note also that \(J^\perp_{T_\rho}/J = E_8^2\) by \Cref{thm:t11-t21cusps}, and the \(\rho\)-action on \(J^\perp_{T_\rho}/J\) respects the direct sum decomposition.
On each summand, up to the action of the Weyl group, it is the unique automoprhism of order \(3\) that has no non-zero fixed vectors.

\subsection{Kulikov models} \label{subsec:Kulikov_02}
We use the notation in \Cref{subsec:ell_fib_trigonal}.
In particular, we let \(\mathcal H\) be the moduli space of triple covers \(\phi \colon Q \sqcup \P^1 \to \P^1\) and \(\overline {\mathcal H}\) its compactification by admissible covers.
We consider the boundary divisor \(\Delta \subset \overline {\mathcal H}\) whose generic point parametrises a degenerate triple cover with the dual graph shown in \Cref{fig:dualgraph02}.
In this graph, the vertices at the top represent irreducible components of \(C\), labelled by their genus, which we omit if it is zero.
The vertices at the bottom represent irreducible components of \(P = \P^1 \cup \P^1\).
Edges represent nodes, and the map \(C \to P\) corresponds to vertical downward projection.
Recall that \(C \to P\) is unramified over the node, so there are exactly 3 nodes of \(C\) and they map to the unique node of \(P\).
The degree of each component of \(C\) over the corresponding component of \(P\) is simply the number of edges incident to the corresponding vertex.

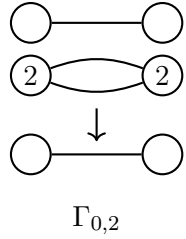
\begin{figure}
  \centering
  \begin{tikzpicture}[thick, scale=1.75, every node/.style={inner sep=0.2em}]
      \draw (0,0) node[circle, draw] (P1) {\phantom{0}} (1,0) node[circle, draw](P2) {\phantom{0}} (P1) edge (P2);
      \draw
      (0,0.6) node[circle, draw] (C1) {2}
      (0,1.0) node[circle, draw] (C2) {\phantom{0}}
      (1,0.6) node[circle, draw](E1) {2}
      (1,1.0) node[circle, draw](E2) {\phantom{0}}
      (C2) edge (E2) (C1) edge[bend left=20] (E1) (C1) edge[bend right=20] (E1);
      \draw (.5,-.5) node {\(\Gamma_{0,2}\)};
      \draw[->]        (0.5,0.35)   -- (0.5,0.1);
  \end{tikzpicture}
  \caption{For the case \(n = 0\) and \(k = 2\), we consider degenerations of triple covers whose central fibers have the dual graph shown above.}
  \label{fig:dualgraph02}
\end{figure}
The extended period map from \(\overline{\mathcal H}\) to the Baily--Borel compactification maps \(\Delta\) to the unique cusp.
The map to the toroidal compactification maps \(\Delta\) dominantly onto the boundary divisor over the cusp (\Cref{prop:dominanceonboundary}).
\subsection{Stable models and the KSBA semifan} \label{subsec:stable_model_02}
We now identify the KSBA semifan \(\mathfrak F\) that gives the KSBA compactification \(\overline F_{\rho}^{\rm KSBA}\).
For the cusp of \(\overline{\D_{\rho}/\Gamma_\rho}^{\rm BB}\) corresponding to an isotropic rank 2 sublattice \(J \subset T_{\rho}\), we must specify a sublattice \(\mathfrak F_J \subset J^{\perp}_{T_{\rho}}/J\).
\begin{theorem}\label{thm:ksba02}
  The space \(\overline{F}_{\rho}^{\rm KSBA}\) is isomorphic to the semi-toroidal compactification for the following semifan \(\mathfrak F = \{0\}\).
  In other words, \(\overline{F}_{\rho}^{\rm KSBA} \simeq \overline{\D_{\rho}/\Gamma_\rho}^{\rm tor}\).
\end{theorem}

We first explain the strategy to prove \Cref{thm:ksba02} and analogous theorems that follow.
Let \(\delta \subset \overline{\D_{\rho}/\Gamma_\rho}^{\rm tor}\) be the divisor lying over the unique cusp corresponding to \(J \subset T_\rho\).
Denoting by \(E\) the elliptic curve with an order 3 automorphism, we know that \(\delta\) is a quotient of \(\mathcal{A}_J = J^{\perp}_{T_{\rho}}/J \otimes_{\Z[\zeta_3]}E\) (see \Cref{subsec:semitoroidal} and \Cref{subsec:limiting_periods}).
The sublattice \(\mathfrak F_J\) is characterised by the property that the translates of \(\mathfrak F_J \otimes_{\Z[\zeta_3]} E \subset \mathcal{A}_J\) are contracted by \eqref{eq:torksba}.
The divisor \(\Delta \subset \overline {\mathcal H}\) maps dominantly onto \(\delta\).

\medskip

\paragraph{Strategy for finding the KSBA semifan}
\begin{enumerate}[start=1, label={(\text{S}\arabic*)}]
\item Take a generic point of \(\Delta\).
  Let \(\psi \in \mathcal{A}_J\) be the period point of \(X\).
\item\label{step:stabilisation}
    Find the KSBA stable limit of a one-parameter family of pairs degenerating to \((X, R)\). The stable limit gives the image of \(\psi\) in \(\overline F_{\rho}^{\rm KSBA}\).
\item Identify when two period points \(\psi\) lead to the same stable limit, and thus identify the sub-lattice \(\mathfrak F_J \subset J^\perp_{T_{\rho}}/J\). 
\end{enumerate}

The key step in the strategy is \ref{step:stabilisation}.
We observe that \((X, \epsilon R)\) is not KSBA stable when \(R\) is not ample on \(X\).
In this case, we modify \((X,R)\) by a sequence of
\begin{enumerate}
\item M1 modifications so that \(R\) becomes nef, and then
\item curve contractions so that \(R\) becomes ample.
\end{enumerate}

\begin{proof}[Proof of \Cref{thm:ksba02}]
  We execute the strategy outlined above.
Let \(\phi \colon C = P \sqcup Q \to P\) be a generic point of \(\Delta\).
Write \(P = P_0 \cup P_1\) and \(Q = Q_0 \cup Q_1\) such that \(P_j \sqcup Q_j\) maps to \(P_j\).
Let \(X\) be the associated type II surface and \(R \subset X\) the genus 5 component of the fixed divisor of \(\sigma\).
Note that \(R\) is a copy of \(Q\).
Write \(X = X_0 \cup X_1\) and \(R = R_0 \cup R_1\) with \(R_j \subset X_j\).

Observe that for every \(i = 0,1\), the curve \(R_i = Q_i\) is smooth of genus \(2\) and \(P_i \simeq \P^1\).
From \Cref{sec:k0-01-22}, we know that \(X_i\) is the blow up of \(\overline X_i\), which is a \(\delpezzo_1\), in a \((\Z/3\Z)\)-fixed point; the exceptional divisor of this blowup is \(P_i \subset X_i\). 
The curve \(R_i = Q_i\) is big and nef on \(X_i\), and is disjoint from the exceptional divisor \(P_i\) of the blowup.
On \(\overline X_i\), the image of \(R_i\) is anti-canonical, and hence ample.
Therefore, the KSBA stable surface associated to \(X\) is \(\overline X = \overline X_0 \cup \overline X_1\) (see \Cref{fig:kulstab02}).
Observe that (up to a finite choice), the \((\Z/3\Z)\)-isomorphism class of \(X\) and \(\overline X\) is determined by the \((\Z/3\Z)\)-isomorphism classes of the anti-canonical pairs \((\overline X_0,D)\) and \((\overline X_1,D)\).
It follows that stabilization map has finite fibers, so \(\mathfrak F_J = 0\).
\begin{figure}[ht]
  \centering
  \begin{tikzpicture}
    \draw[thick]  (0,0) -- (0,2) -- (1.5,1.5) -- (1.5,-0.5) -- (0,0) -- (0,2) -- (-1.5,1.5) -- (-1.5,-0.5) -- (0,0);
    \draw[color=blue, thick, dashed] (0,1.5) -- (0.5,1.34);
    \draw[color=blue, thick, dashed] (0,1.5) -- (-0.5,1.34);
    \draw[color=red, thick] (0,0.8) -- (0.5,0.64) (0,0.9) -- (0.5,0.74);
    \draw[color=red, thick] (0,0.8) -- (-0.5,0.64) (0,0.9) -- (-0.5,0.74);
    \draw (1.7,0.75) node (K) {};
    \begin{scope}[xshift=5cm, yshift=-0.25cm]
      \draw[thick]  (0,0) -- (1.5,1) -- (0,2) -- (0,0) -- (-1.5,1) -- (0,2) -- (0,0);
      \draw[color=red, thick] (0,1.1) -- (0.5,0.94) (0,1.0) -- (0.5,0.84);
      \draw[color=red, thick] (0,1.1) -- (-0.5,0.94) (0,1.0) -- (-0.5,0.84);
      \draw (-1.7, 1) node (S) {};      
      \filldraw [blue] (0,1.5) circle (1.5pt);
    \end{scope}
    \draw[thick, ->] (K) edge (S);
  \end{tikzpicture}
  \caption{The stable model is obtained by contracting \((\Z/3\Z)\)-fixed \((-1)\)-curves on each component (dashed blue) disjoint from the divisor \(R\) (red). The image of contracted curves is a single \(\Z/3\Z\)-fixed point (blue dot) of the double curve of \(\overline X\), which is the unique isolated \(\Z/3\Z\)-fixed point of \(\overline X\).}\label{fig:kulstab02}
\end{figure}
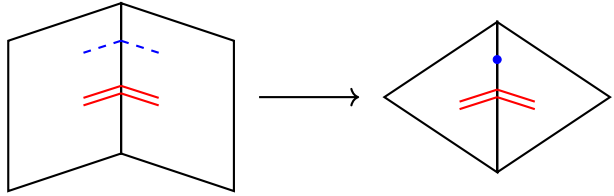
\end{proof}

\section{Moduli space 2: \texorpdfstring{\(g=4 \)}{g=4} and \texorpdfstring{\((n,k) = (0,1)\)}{(n,k) = (0,1)}}\label{sec:n0k1}
In this section, we identify the KSBA compactification of the moduli space of K3 surfaces with a non-symplectic automorphism of order 3 with \(0\) isolated fixed points and \(1\) fixed curve (so \(n = 0\) and \(k = 1\)).

Let \((X, \sigma)\) be a generic K3 surface with an automorphism of order \(3\) with \(n = 0\) and \(k = 1\).
From \Cref{prop:triple-tschirnhausen-dominant}, we know that \((X,\sigma)\) arises from the triple Tschirnhausen construction applied to \((\phi \colon C \to \P^1)\), where \(C\) is a smooth curve of genus \(4\).
Observe that, for a generic \(\phi\), the Tschirnhausen embedding identifies \(C\) as a divisor on \(\P^1 \times \P^1\) of class \((3,3)\).
Therefore, \(X\) is simply the triple cover of \(\P^1 \times \P^1\) branched over a curve of class \((3,3)\).
Artebani--Sarti give the same description using explicit equations \cite[Proposition~4.7]{art.sar:08}.

Fix an isometry of \(H^2(X, \Z)\) with the K3 lattice \(L\), and let \(\rho\) be the automorphism of \(L\) induced by \(\sigma\).
We recall the Hodge type of \(\rho\) from \Cref{tab:ASlattices} (also see \cite[\S~2]{kon:02}).
Recall that \(S_{\rho} \subset L\) is the sub-lattice of vectors fixed by \(\rho\) and \(T_{\rho} = S_{\rho}^{\perp}\).
As \(\rho\) is identified with \(\sigma^*\), the pull-back of the cyclic triple cover \(X \to \P^1 \times \P^1\) gives an identification
\[ S_{\rho} = H^2(\P^1 \times \P^1, \Z) (3) = U(3).\]
The lattice \(T_{\rho}\) turns out to be 
\[ T_{\rho} \simeq U \oplus U(3) \oplus E_8 \oplus E_8.\]
The automorphism \(\rho\) acts on the summands \(U \oplus U(3)\) and \(E_8\) and \(E_8\).
For \(U \oplus U(3)\), the action is described in \Cref{subsec:eis_cusp} and for \(E_8\) it is described in \cite[Ch~II~\S~2.6]{con.slo:99}.

\subsection{Baily--Borel cusps} \label{subsec:BB_01}
Let \(F_\rho^{\rm sep} = (\D_{\rho} \setminus \Delta_\rho)/\Gamma_{\rho}\) be the period domain for \(\rho\)-markable K3 surfaces as described in \Cref{subsec:period_map}.
From \cite[Theorem~5.9]{cas.jen.laz:12}, 
see also \Cref{thm:t11-t21cusps},
the cusps of \(\overline{\D_{\rho}/\Gamma_{\rho}}^{\rm BB}\) are classified by the root sublattices of \(J^\perp_{T_\rho}/J\), which is one of
\begin{enumerate}
  \item\label{cusp:e6e6a2a2}   \(E_6^{\oplus 2} \oplus A_2^{\oplus 2}\),
  \item\label{cusp:e8e6a2} \(E_8 \oplus E_6 \oplus A_2\), or
  \item\label{cusp:e8e8} \(E_8^{\oplus 2}\).
\end{enumerate}

The \(\rho\)-action on \(J^{\perp}_{T_\rho}/J\)  respects the direct sum decomposition  of the root sublattice \cite[Lemma~5.5]{cas.jen.laz:12}.
On each summand, up to the action of the Weyl group, it is the unique automoprhism of order 3 that has no non-zero fixed vectors.

\subsection{Kulikov models} \label{subsec:Kulikov_01}
We use the notation in \Cref{subsec:ell_fib_trigonal}.
In particular, we let \(\mathcal H\) be the moduli space of triple covers \(\phi \colon C \to \P^1\) where \(C\) is a smooth genus 4 curve and \(\overline {\mathcal H}\) its compactification by admissible covers.
For \(i = 1,2,3\), we consider the boundary divisors \(\Delta_i \subset \overline {\mathcal H}\) whose generic point parametrise a degenerate triple cover with the dual graphs \(\Gamma_i\) shown in \Cref{fig:divisors}.

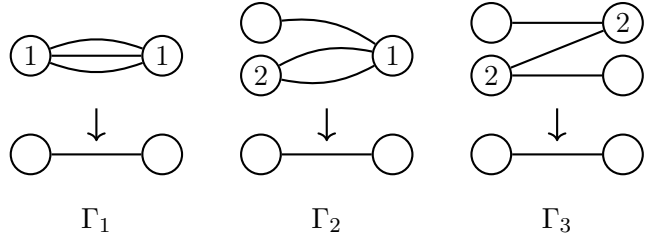
\begin{figure}
  \centering
  \begin{tikzpicture}[thick, scale=1.75, every node/.style={inner sep=0.2em}]
    \draw (0,0) node[circle, draw] (P1) {\phantom{0}} (1,0) node[circle, draw](P2) {\phantom{0}} (P1) edge (P2);
    \draw (0,0.75) node[circle, draw] (E1) {1} (1,0.75) node[circle, draw](E2) {1}
    (E1) edge[bend left=20] (E2) (E1) edge (E2) (E1) edge[bend right=20] (E2);
    \draw[->]        (0.5,0.35)   -- (0.5,0.1);
    \draw (.5,-.5) node {\(\Gamma_1\)};

    \begin{scope}[xshift=1.75cm]
      \draw (0,0) node[circle, draw] (P1) {\phantom{0}} (1,0) node[circle, draw](P2) {\phantom{0}} (P1) edge (P2);
      \draw
      (0,0.6) node[circle, draw] (C1) {2}
      (0,1.0) node[circle, draw] (C2) {\phantom{0}}
      (1,0.75) node[circle, draw](E2) {1}
      (C2) edge[bend left=20] (E2) (C1) edge[bend left=20] (E2) (C1) edge[bend right=20] (E2);
      \draw[->]        (0.5,0.35)   -- (0.5,0.1);
      \draw (.5,-.5) node {\(\Gamma_2\)};
    \end{scope}
    \begin{scope}[xshift=3.5cm]
      \draw (0,0) node[circle, draw] (P1) {\phantom{0}} (1,0) node[circle, draw](P2) {\phantom{0}} (P1) edge (P2);
      \draw
      (0,0.6) node[circle, draw] (C1) {2}
      (0,1.0) node[circle, draw] (C2) {\phantom{0}}
      (1,0.6) node[circle, draw](E1) {\phantom{0}}
      (1,1.0) node[circle, draw](E2) {2}
      (C2) edge (E2) (C1) edge (E2) (C1) edge (E1);
      \draw[->]        (0.5,0.35)   -- (0.5,0.1);
      \draw (.5,-.5) node {\(\Gamma_3\)};
    \end{scope}
    %  \begin{scope}[xshift=5.25cm]
    %   \draw (0,0) node[circle, draw] (P1) {\phantom{0}} (1,0) node[circle, draw](P2) {\phantom{0}} (P1) edge (P2);
    %   \draw
    %   (0,0.75) node[circle, draw] (C1) {2}
    %   (1,0.75) node[circle, draw](C2) {2}
    %   (C1) edge node[above]{3} (C2);
    %   \draw (.5,-.5) node {\(D_4\)};
    % \end{scope}
  \end{tikzpicture}
  \caption{The dual graphs of the admissible covers that give Kulikov degenerations in the case \(n = 0\) and \(k = 1\).}
  \label{fig:divisors}
\end{figure}

\begin{proposition} \label{prop:period_01}
  The extended period map \(\overline {\mathcal H} \dashrightarrow \overline{\D_{\rho}/\Gamma_{\rho}}^{\rm BB}\) maps \(\Delta_i\) to the \(i\)-th cusp in \Cref{subsec:BB_01}.
\end{proposition}
\begin{proof}
  The generic point of \(\Delta_i\) corresponds to a Kulikov surface \(X = X_0 \cup X_1\) and a 2-plane \(J \subset T_{\rho}\).
  Let \(\Lambda\) be the reduced lattice of numerically Cartier divisors on \(X\).
  By \eqref{eqn:JperpJrho}, we have an isometry
  \begin{equation}\label{eqn:primJ}
    \Lambda^{\rm prim} \simeq J^\perp_{T_\rho} / J.
  \end{equation}
  \Cref{prop:root-sublattice-splits} implies that the root sub-lattice of \(\Lambda^{\rm prim}\) is the direct sum of the root sublattices of the lattices \(\Lambda_0\) and \(\Lambda_i\) associated to \(X_0\) and \(X_1\).
  \Cref{prop:lattices} identifies these lattices.
  As a result, we see that the root sublattice of \(J^\perp_{T_\rho} / J\) is:
  \begin{enumerate}
  \item for \(i = 1\), the direct sum of two copies of \(E_6 \oplus A_2\),
  \item for \(i = 2\), the direct sum of \(E_8\) and \(E_6 \oplus A_2\), and
  \item for \(i = 3\), the direct sum of two copies of \(E_8\).
  \end{enumerate}
  The statement follows.
\end{proof}

\subsection{Stable models and the KSBA semifan} \label{subsec:stable_model_01}
We now identify the KSBA semifan \(\mathfrak F\) that gives the KSBA compactification \(\overline F_{\rho}^{\rm KSBA}\).
For each cusp of \(\overline{\D_{\rho}/\Gamma_\rho}^{\rm BB}\) corresponding to an isotropic rank 2 lattice \(J \subset T_{\rho}\), we must specify a sublattice \(\mathfrak F_J \subset J^{\perp}_{T_{\rho}}/J\).
We use the superscript \(\rm{root}\) to denote the root sublattice of a lattice and the superscript \(\rm{sat}\) to denote the saturation of a sub-lattice in an ambient lattice.
\begin{theorem}\label{thm:ksba01}
  The space \(\overline{F}_{\rho}^{\rm KSBA}\) is isomorphic to the semi-toroidal compactification for the following semisfan \(\mathfrak F_J\):
  \[
    \mathfrak F_J = 
    \begin{cases}
      \langle  A_2^{\oplus 2} \rangle^{\rm sat} & \text{ for the cusp with } (J^{\perp}_{T_{\rho}}/J)^{\rm root} = E_6^{\oplus 2} \oplus A_2^{\oplus 2}\\
      \langle  A_2 \rangle^{\rm sat} & \text{ for the cusp with } (J^{\perp}_{T_{\rho}}/J)^{\rm root} = E_8 \oplus E_{6} \oplus A_2 \\
      0 & \text{ for the cusp with } (J^{\perp}_{T_{\rho}}/J)^{\rm root} = E_8^{\oplus 2}.
    \end{cases}
  \]
\end{theorem}
We devote the rest of \Cref{subsec:stable_model_01} to the proof of \Cref{thm:ksba01}.
For \(i = 1, 2, 3\), let \(\delta_i \subset \overline{\D_{\rho}/\Gamma_\rho}^{\rm tor}\) is the divisor lying over the \(i\)-th cusp.
Let \(J = J_i \subset T_{\rho}\) be the rank 2 isotropic sublattice corresponding to this cusp.
For each \(i\), we follow the strategy outlined in \Cref{subsec:stable_model_02}.
We let \(\phi \colon C_0 \cup C_1 \to P_0 \cup P_{1}\) be the degenerate triple cover corresponding to a generic point of \(\Delta_i\) and \(X = X_0 \cup X_1\) the associated Kulikov surface.
In this case, we have \(R = C \subset X\); write \(R = R_0 \cup R_{1}\) with \(R_i \subset X_i\).

\subsubsection{The case \(i = 1\)}\label{sec:01i1}
In this case, both \(C_{0}\) and \(C_1\) are smooth curves of genus \(1\).
From \Cref{sec:k0-13}, we know that \(X_j\) is the blow up of \(\overline X_j\), which is a \(\delpezzo_3\), in three points that form a \((\Z/3\Z)\)-orbit.
The curve \(R_j\) is big and nef and disjoint from the exceptional divisors of the blow-up.
On \(\overline X_j\), the image of \(R_j\) is anti-canonical, and hence ample.
Therefore, the KSBA stable surface associated to \(X\) is \(\overline X = \overline X_0 \cup \overline X_1\) (see \Cref{fig:kulstab011}).
Observe that (up to a finite choice), the \((\Z/3\Z)\)-isomorphism class of \(\overline X\) is determined by the \((\Z/3\Z)\)-isomorphism classes of the anti-canonical pairs \((\overline X_0,D)\) and \((\overline X_1,D)\).
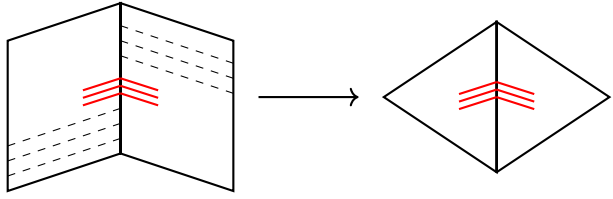
\begin{figure}[ht]
  \centering
  \begin{tikzpicture}
    \draw[thick]  (0,0) -- (0,2) -- (1.5,1.5) -- (1.5,-0.5) -- (0,0) -- (0,2) -- (-1.5,1.5) -- (-1.5,-0.5) -- (0,0);
    \draw[dashed] (0,1.7) -- (1.5,1.2) (0,1.5) -- (1.5,1.0) (0,1.3) -- (1.5,0.8);
    \draw[dashed] (0,0.2) -- (-1.5,-0.3) (0,0.4) -- (-1.5,-0.1) (0,0.6) -- (-1.5,0.1) ;
    \draw[color=red, thick] (0,0.8) -- (0.5,0.64) (0,0.9) -- (0.5,0.74) (0,1.0) -- (0.5,0.84);
    \draw[color=red, thick] (0,0.8) -- (-0.5,0.64) (0,0.9) -- (-0.5,0.74) (0,1.0) -- (-0.5,0.84);
    \draw (1.7,0.75) node (K) {};
    \begin{scope}[xshift=5cm, yshift=-0.25cm]
      \draw[thick]  (0,0) -- (1.5,1) -- (0,2) -- (0,0) -- (-1.5,1) -- (0,2) -- (0,0);
      \draw[color=red, thick] (0,1.2) -- (0.5,1.04) (0,1.1) -- (0.5,0.94) (0,1.0) -- (0.5,0.84);
      \draw[color=red, thick] (0,1.2) -- (-0.5,1.04) (0,1.1) -- (-0.5,0.94) (0,1.0) -- (-0.5,0.84);
      \draw (-1.7, 1) node (S) {};      
    \end{scope}
    \draw[thick, ->] (K) edge (S);
  \end{tikzpicture}
  \caption{The stable model in the case \(i = 1\) is obtained by contracting a \((\Z/3\Z)\)-orbit of \((-1)\)-curves on each component (dashed) disjoint from the divisor (red).}\label{fig:kulstab011}
\end{figure}

We identify \(A_2 \oplus E_6\) with \(H^2(X_j, \Z)^{\rm prim}\) following \Cref{prop:lattices}.
We then get an identification
\[ A_2^{\oplus 2} \oplus E_6^{\oplus 2} = H^2(X_0, \Z)^{\rm prim} \oplus H^2(X_1, \Z)^{\rm prim}.\]
For \(j = 0,1\), let \(\psi_j\) be the restriction of the period point \(\psi\) of \(X\) to the first and the second \(E_6\)-summand.
Then \(\psi_j\) is simply the period point of the anti-canonical pair \((\overline{X}_j,D)\).
By the Torelli theorem for anti-canonical del Pezzo pairs (\Cref{prop:dptor}), we conclude that two Kulikov surface \(X\) and \(X'\) whose periods \(\psi\) and \(\psi'\) agree on the \(E_6\) summands have isomorphic stable models (possibly up to a finite choice).
Therefore, the map \(\Delta_1 \dashrightarrow \overline F^{\rm KSBA}_{\rho}\) contracts precisely the traslates of \(A_2^{\oplus 2} \otimes_{\Z[\zeta_3]} E \subset \mathcal{A}_J\).
We conclude that \(\mathfrak F_J \subset J^{\perp}_{T_{\rho}}/J\) is the saturation of the \(A_2^{\oplus 2}\) summand.

\subsubsection{The case \(i = 2\)}\label{sec:01i2}
In this case, \(C_0 = P_0 \sqcup Q_0\), where \(Q_0\) is smooth of genus \(2\), and \(C_1\) is irreducible of genus \(1\).
From \Cref{sec:k0-01-22}, we know that \(X_0\) is the blow up of \(\overline X_0\), which is a \(\delpezzo_1\), in a \((\Z/3\Z)\)-fixed point.
From \Cref{sec:k0-13}, we know that \(X_1\) is the blow-up of \(\overline X_1\), which is a \(\delpezzo_3\), in three points that form a \((\Z/3\Z)\)-orbit.

Note that \(R_0 = P_0 + Q_0\) is negative on \(P_0 \subset X_0\).
Let \(D = X_0 \cap X_1\) be the double curve of \(X\).
We perform an M1 modification on \(X\) by contracting \(P_0 \subset X_0\) and blowing-up \(P_0 \cap D\) on \(X_1\).
Call the resulting surface \(X' = X'_0 \cup X'_1\) and let \(R' \subset X'\) be the proper transform of \(R\).
Then \(R' \subset X'\) is nef and trivial only on the exceptional divisors of \(X_{1} \to \overline X_1\) (which are unaffected by the M1 modification).
Let \(X'_1 \to \overline X'_1\) be the contraction of these three \(-1\) curves.
Let \(\overline X'_0 = X'_0\).
Then \(\overline X' = \overline X'_0 \cup \overline X'_1\) is the stable model (see \Cref{fig:kulstab012}).
Note that (up to a finite choice) the \((\Z/3\Z)\)-isomorphism class of \(\overline X'\) is determined by the \((\Z/3\Z)\)-isomorphism classes of the anti-canonical pairs \((\overline X_0', D)\) and \((\overline X_1, D)\).
\begin{figure}[ht]
  \centering
  \begin{tikzpicture}
    \draw[thick]  (0,0) -- (0,2) -- (1.5,1.5) -- (1.5,-0.5) -- (0,0) -- (0,2) -- (-1.5,1.5) -- (-1.5,-0.5) -- (0,0);
    \draw[dashed] (0,1.7) -- (1.5,1.2) (0,1.5) -- (1.5,1.0) (0,1.3) -- (1.5,0.8);
    \draw[color=red, thick] (0,0.6) -- (0.5,0.44) (0,0.9) -- (0.5,0.74)  (0,1.0) -- (0.5,0.84);
    \draw[color=red, thick]  (0,0.9) -- (-0.5,0.74) (0,1.0) -- (-0.5,0.84);
    \draw[color=red, thick, dashed] (0,0.6) -- (-0.5,0.44);
    \draw (1.7,0.75) node (K) {};
    \begin{scope}[xshift=5cm, yshift=-0.25cm]
      \draw[thick]  (0,0) -- (1.5,0) -- (1.5,2) -- (0,2) -- (0,0) -- (-1.5,1) -- (0,2) -- (0,0);
      \draw[color=red, thick] (0,1.2) -- (0.5,1.04) (0,1.1) -- (0.5,0.94); 
      \draw[color=red, thick] (0,1.2) -- (-0.5,1.04) (0,1.1) -- (-0.5,0.94);
      \draw[color=red, thick]  plot[smooth, tension=1] coordinates {(0.1,0.6) (0.2,0.8) (0.5,0.8)};
      \draw[color=blue, dashed, thick] (0,0.8) -- (0.5,0.64);
      \draw (-1.7, 1) node (S) {};      
    \end{scope}
    \draw[thick, dashed, ->] (K) edge (S);
  \end{tikzpicture}
  \caption{The stable model in the case \(i = 2\) is obtained by an M1 modification along a \((-1)\)-curve on the left (dashed red) followed by contracting a \((\Z/3\Z)\)-orbit of \((-1)\)-curves on the right (dashed).  The divisor is shown in red (dashed and solid).}\label{fig:kulstab012}
\end{figure}
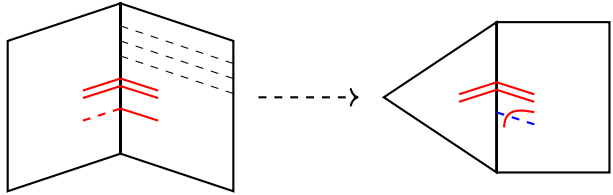

We identify \(E_8\) with \(H^2(X_0, \Z)^{\rm prim}\) and \(E_6 \oplus A_2\) with \(H^2(X_1, \Z)^{\rm prim}\) following \Cref{prop:lattices}.
We then get an identification
\[ E_8 \oplus E_{6} \oplus A_2 = H^2(X_0, \Z)^{\rm prim} \oplus H^2(X_1, \Z)^{\rm prim}.\]
Let \(\psi\) be the period point of \(X\).
Then \(\psi\) induces a homomorphism
\[ \psi \colon E_8 \oplus E_6 \oplus A_2 \to E\]
The restriction \(E_8 \to E\) is the period of the pair \((\overline X_0', D)\) and the restriction \(E_6 \to E\) is the period of the pair \((\overline X_1, D)\).
By the Torelli theorem for anti-canonical del Pezzo pairs (\Cref{prop:dptor}), we conclude that two Kulikov surface \(X\) and \(X'\) whose periods \(\psi\) and \(\psi'\) agree on the \(E_8\) and \(E_6\) summands have isomorphic stable models (possibly up to a finite choice).
Therefore, the map \(\Delta_2 \dashrightarrow \overline F^{\rm KSBA}_{\rho}\) contracts precisely the traslates of \(A_2 \otimes_{\Z[\zeta_3]} E \subset \mathcal{A}_J\).
We conclude that \(\mathfrak F_J \subset J^{\perp}_{T_{\rho}}/J\) is the saturation of the \(A_2\) summand.

\subsubsection{The case \(i = 3\)}\label{sec:01i3}
In this case, \(C_j = P_j \cup Q_j\), where \(Q_j\) is smooth of genus \(2\).
From \Cref{sec:k0-01-22}, we know that \(X_j\) is the blow up of \(\overline X_j\), which is a \(\delpezzo_1\), in a \((\Z/3\Z)\)-fixed point.

Note that \(R_j = P_j + Q_j\) is not nef on \(X_j\).
Let \(D = X_0 \cap X_1\) be the double curve of \(X\).
We perform two M1 modifications: contract \(P_0 \subset X_0\) and blow-up \(P_0 \cap D\) on \(X_1\) and likewise conract \(P_1 \subset X_1\) and blow-up \(P_1 \cap D\) on \(X_{0}\).
Let \(X' = X_0' \cup X_1'\) be the resulting surface and \(R' \subset X'\) the proper transform of \(R\).
Then \(R'\) is ample on \(X'\) and \(X'\) is the stable model (see \Cref{fig:kulstab013}).
Up to a finite choice, the \((\Z/3\Z)\)-isomorphism class of \(X'\) is determined by the \((\Z/3\Z)\)-isomorphism classes of the pairs \((\overline X_0, D)\) and \((\overline X_1, D)\).
It follows that no translates in \(E_8^{\oplus 2} \otimes_{\Z[\zeta_3]} E = \mathcal{A}_J\) are contracted by \(\Delta_3 \dashrightarrow \overline F_{\rho}^{\rm KSBA}\).
We conclude that \(\mathfrak F_J = 0\).
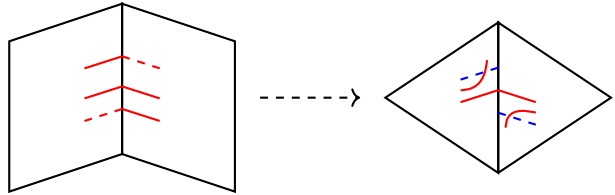
\begin{figure}[ht]
  \centering
  \begin{tikzpicture}
    \draw[thick]  (0,0) -- (0,2) -- (1.5,1.5) -- (1.5,-0.5) -- (0,0) -- (0,2) -- (-1.5,1.5) -- (-1.5,-0.5) -- (0,0);
    \draw[color=red, thick] (0,0.6) -- (0.5,0.44) (0,0.9) -- (0.5,0.74);
    \draw[color=red, thick]  (0,0.9) -- (-0.5,0.74) (0,1.3) -- (-0.5,1.14);
    \draw[color=red, thick, dashed] (0,0.6) -- (-0.5,0.44)   (0,1.3) -- (0.5,1.14);
    \draw (1.7,0.75) node (K) {};
    \begin{scope}[xshift=5cm, yshift=-0.25cm]
      \draw[thick]  (0,0) -- (1.5,1) -- (0,2)-- (0,0) -- (-1.5,1) -- (0,2) -- (0,0);
      \draw[color=red, thick] (0,1.1) -- (0.5,0.94); 
      \draw[color=red, thick] (0,1.1) -- (-0.5,0.94);
      \draw[color=red, thick]  plot[smooth, tension=1] coordinates {(0.1,0.6) (0.2,0.8) (0.5,0.8)};
      \draw[color=red, thick]  plot[smooth, tension=1] coordinates {(-0.15,1.5) (-0.25,1.2) (-0.5,1.1)};
      \draw[color=blue, dashed, thick] (0,0.8) -- (0.5,0.64) (0,1.4) -- (-0.5,1.24) ;
      \draw (-1.7, 1) node (S) {};      
    \end{scope}
    \draw[thick, dashed, ->] (K) edge (S);
  \end{tikzpicture}
  \caption{The stable model in the case \(i = 3\) is obtained by M1 modifications along the \((-1)\)-curves on the left and the right (dashed red).  The divisor is shown in red (dashed and solid).}\label{fig:kulstab013}
\end{figure}

\begin{proof}[Conclusion of the proof of \Cref{thm:ksba01}]
  Having settled the cases of the three cusps in \Cref{sec:01i1}, \Cref{sec:01i2}, and \Cref{sec:01i3}, the proof of \Cref{thm:ksba01} is complete.
\end{proof}
\subsection{Connection with the moduli space of stable log quadrics} \label{subsec:compare_DH21}
Fix a small positive number \(\epsilon\).
Let \(\Y_{\rho}\) be the KSBA compactification of weighted pairs \((Y, (\frac{2}{3}+\epsilon)B)\) where \(Y \simeq \P^1 \times \P^1\), and \(B \subset Y\) is a curve of bi-degree \((3,3)\).
We call \(\Y_{\rho}\) the space of stable log quadrics.
By \Cref{thm:stable_quotient_pair}, the map \(X \mapsto X/\sigma\) induces an isomorphism
\( \overline{F}_{\rho}^{\rm KSBA} \to \Y_{\rho}\).
The geometry of \(\Y_{\rho}\) is described in detail in \cite{deo.han:21}.
In particular, we know that \(\Y_{\rho}\) is the coarse space of a smooth stack, and hence already normal.

A Zariski open subset of \(\Y_{\rho}\) parametrises \((Y, B)\) whose triple cover \(X\) is either a smooth K3 surface or a type I degeneration.
There are three disjoint irreducible Zariski closed subsets  of \(\Y_{\rho}\) where the triple covers come from type II degenerations.
These map to the three cusps of \(\overline{\D/\Gamma_{\rho}}^{\rm BB}\).
We recall their description.
\begin{proposition}
  The locus of points \((Y, B)\) of \(\Y_{\rho}\) with reducible \(Y\) is the union of three irreducible components \(Z_1, Z_2, Z_3\) of dimension \(6\), \(7\), and \(8\), respectively.
  Their generic points correspond to the following \((Y, B)\).
\begin{description}
\item [\(Z_1\)] \(Y = \P^2 \cup \P^2\) glued along a line and \(B = E_0 \cup E_1\) where \(E_i \subset \P^2\) are cubic curves.
\item [\(Z_2\)] \(Y = T \cup \P(3,1,1)\), where \(T\) is the \(\Q\)-Gorenstein smoothing of the \(A_1\)-singularity of \(\P(3,1,2)\). 
In this case, \(B = E \cup C_1\) where \(E \subset T\) is an elliptic curve, \(C_1 \subset \P(3,1,1)\) is a genus \(2\) curve with \(\mathcal{O}_{\P(3,1,1)}(C_1)=\mathcal{O}_{\P(3,1,1)}(2)\), and \(E \cap C_1\) consists of two distinct points. 
\item [\(Z_3\)] \(Y\) is the coarse space of \(\P(O(4/3,5/3) \oplus O(5/3,4/3))\) over an orbinodal curve \(P\) with two irreducible components glued along a node of stabilizer group \(\mu_3\).
The restriction of \(B\) to the two components of \(Y\) are smooth curves of genus \(2\) which intersect the double curve transversally at a single non-Weierstrass point.
\end{description}
\end{proposition}
\begin{proof}
  Follows from the description of the points of \(\Y_{\rho}\) given in \cite[Table~1]{deo.han:21}.
\end{proof}

We now determine which \(Z_{i}\) map to which cusp.
\begin{theorem}\label{thm:cuspmap}
  Enumerate the cusps as in \Cref{subsec:BB_01}.
  Then the map \(\Y_{\rho} \to \overline{\D_{\rho}/\Gamma_{\rho}}^{\rm BB}\) sends \(Z_i\) to the \(i\)-th cusp.
\end{theorem}
\begin{proof}
  We use the same Hurwitz space \(\mathcal H\), but its compactification \(\overline{\mathcal H}\) as the moduli space of weighted admissible covers, where we assign weight \(1/6+\epsilon\) to every branch point.
  This space was constructed in \cite{deo:14} and used extensively in \cite{deo.han:21}.
  By \cite[Theorem~3]{deo.han:21}, we have a regular surjective morphism \[\overline{\mathcal H} \to \Y_{\rho}.\]
  From \cite[Table~1]{deo.han:21}, we see that \(Z_i\) is the image of the boundary divisor of \(\overline{\mathcal H}\) corresponding to the dual graph \(\Gamma_i\).
  The statement now follows from \Cref{prop:period_01}.
\end{proof}

\section{Moduli space 3: \texorpdfstring{\(g=3\)}{g=3} and \texorpdfstring{\((n,k) = (1,1)\)}{(n,k) = (1,1)}}\label{sec:n1k1}
In this section, we identify the KSBA compactification ofthe moduli space of K3 surfaces with a non-symplectic automorphism of order 3 with \(1\) isolated fixed points and \(1\) fixed curve (so \(n = 1\) and \(k = 1\)).

Let \((X,\sigma)\) be generic K3 surface with an automorphism of order \(3\) with \(n = 1\) and \(k = 1\).
From \Cref{prop:triple-tschirnhausen-dominant}, we know that \((X,\sigma)\) arises from the pinched triple Tschirnhausen construction applied to \((\phi \colon C \to \P^1, p)\), where \(C\) is a smooth curve of genus \(3\).
For a generic \(\phi\), the Tschirnhausen embedding identifies \(C\) as a divisor in \(\F_1\) of class \(3 \sigma + 4 f\).
In the literature, there are two other explicit constructions of \(X\): \cite[Proposition~4.9]{art.sar:08} and \cite[Example~4.10]{ma.oha.tak:15}, both of which are equivalent to our construction.

% In \cite[Proposition~4.9]{art.sar:08}, Artebani--Sarti give explicit projective equations for \(X\), which are easy to 

% \begin{remark}

%   In \(\P^3\) with homogeneous coordinates \(x_0,\dots,x_3\), it is cut out by a quartic of the form
%   \begin{equation}\label{eqn:artebani-sarti-11}
%     F_4(x_0,x_1,x_2) + F_1(x_0,x_1,x_2)x_3^3 = 0,
%   \end{equation}
%   where \(F_i\) is a generic form of degree \(i\).
%   Let us reconcile this description with our description.
%   Let \(Y\) be a surface cut out by \eqref{eqn:artebani-sarti-11} and let \(X\) be obtained from a marked genus 3 triple cover by the pinched triple Tschirnhausen construction.
%   It suffices to exhibit a birational isomorphism between \(X\) and \(Y\).
%   Since both are K3 surfaces, a birational isomorphism must be biregular.
%   After multiplying the equation of \(Y\) by \(F_1^2\) and setting \(z = x_3F_1\), we see that \(Y\) is birational to the affine surface
%   \[ z^3 = f_4(x,y,z) f_1(x,y,z)^2.\]
%   On the other hand, by construction \(X\) is birational to the cyclic triple cover of \(\F_1\) with branch divisor \(C + 2f\).
%   Under the blow-down \(\F_1 \to \P^2\), the curve \(C\) maps to a quartic and \(f\) to a line.
%   Therefore, \(X\) is also birational to an affine surface with the same equation as above.
%   Finally, observe that the isomorphism \(X \cong Y\) thus obtained is \((\Z/3\Z)\)-equivariant.
% \end{remark}

Fix an isometry of \(H^2(X, \Z)\) with the K3 lattice \(L\) and let \(\rho\) be the automorphism of \(L\) induced by \(\sigma\).
We recall the Hodge type of \(\rho\) from \Cref{tab:ASlattices}.
The lattice \(S_{\rho} \subset H^2(X,\Z)\) of \(\rho\)-fixed vectors is given by
\[ S_{\rho} = U(3) \oplus A_2.\]
Its orthogonal complement \(T_{\rho}\) is
\[ T_{\rho} = U \oplus U(3) \oplus E_6 \oplus E_8.\]
The automorphism \(\rho\) acts on the \(U \oplus U(3)\) summand as described in \Cref{subsec:eis_cusp}, and on the \(E_6\) and \(E_8\) summands by the unique order 3 automorphism without non-zero fixed vectors \cite[Ch~II~\S~2.6]{con.slo:99}.

\begin{remark}\label{rem:srhobasis11}
  We describe a basis for \(S_{\rho} = U(3) \oplus A_2\).
  Let \(\widehat \F_1\) be the blow-up of \(\F_1\) at the three points of \(C\) over \(p\).
  Recall that \(X\) is obtained by taking a normalised triple cover of \(\F_1\) and then blowing down the pre-image of the proper transform of the fiber over \(p\).
  Let \(\sigma \subset \widehat \F_1\) be the directrix of \(\F_1\) and let \(E_1,E_2,E_3 \subset \widehat \F_1\) be the exceptional divisors of the blow-up.
  For a curve \(\alpha \subset \widehat \F_1\), use \(\widetilde{\alpha}\) to denote the image in \(X\) of the pre-image of \(\alpha\) in the triple cover.
  Then \[S_{\rho} = \langle \widetilde{\sigma}, \widetilde{E_1}, \widetilde{E_2}, \widetilde{E_3}\rangle.\]
  Set \(e = \widetilde{E_1}+ \widetilde{E_2}+ \widetilde{E_3}\) and \(f = \widetilde{\sigma} + \widetilde{E_1} + \widetilde{E_2}\).
  Then \(e^2 = f^2 = 0\) and \(e \cdot f = 3\).
  Changing the basis of \(S_{\rho}\) to \((e,f,\widetilde{E}_1, \widetilde{E}_2)\) gives an isomorphism \(S_{\rho} \cong U(3) \oplus A_2\).
\end{remark}
\subsection{Baily--Borel cusps} \label{subsec:BB_11}
Let \(F_{\rho}^{\rm sep} = \left(\D_{\rho} \setminus \Delta_{\rho}\right) / \Gamma_{\rho}\) be the period domain for \(\rho\)-markable K3 surfaces as described in \Cref{subsec:period_map}. 
Recall from \Cref{thm:t11-t21cusps}, the classification of cusps for the Baily--Borel compactification \(\overline{\D_{\rho}/\Gamma_{\rho}}^{\rm BB}\) by the root sublattice of \(J^{\perp}_{T_{\rho}}/J\), which is one of:
\begin{enumerate}
  \item \(E_6 \oplus A_2^{\oplus 4}\),
  \item \(E_8 \oplus A_2^{\oplus 3}\),
  \item \(E_6^{\oplus 2} \oplus A_2\), 
  \item \(E_8\oplus E_6\).
\end{enumerate}

\subsection{Kulikov models}
We use the notation in \Cref{subsec:ell_fib_trigonal}.
In particular, we let \(\mathcal H\) be the moduli space of marked triple covers \((\phi \colon C \to \P^1, p)\) where \(C\) is a smooth genus 3 curve and \(\phi\) is \'etale over \(p\).
Let \(\overline {\mathcal H}\) be its compactification by marked admissible covers.
For \(i = 1,\dots,4\), we consider the boundary divisor \(\Delta_i \subset \overline {\mathcal H}\) whose generic point parametrizes a degenerate triple cover with the dual graph \(\Gamma_i\) shown in \Cref{fig:dualgraphs11}  (we indicate the marked point \(p\) by a half edge).
\begin{figure}
  \centering
  \begin{tikzpicture}[thick, scale=1.75, every node/.style={inner sep=0.2em}]
    \draw (0,0) node[circle, draw] (P1) {\phantom{0}} (1,0) node[circle, draw](P2) {\phantom{0}} (P1) edge (P2);
    \draw node[below=.25cm of P2] (P2p) {} (P2) edge (P2p);    
    \draw (0,0.75) node[circle, draw] (E1) {1} (1,0.75) node[circle, draw](E2) {\phantom{0}}
    (E1) edge[bend left=20] (E2) (E1) edge (E2) (E1) edge[bend right=20] (E2);
    \draw[->]        (0.5,0.35)   -- (0.5,0.1);
    \draw (.5,-.5) node {\(\Gamma_{1}\)};

    \begin{scope}[xshift=1.75cm]
      \draw (0,0) node[circle, draw] (P1) {\phantom{0}} (1,0) node[circle, draw](P2) {\phantom{0}} (P1) edge (P2);
      \draw node[below=.25cm of P2] (P2p) {} (P2) edge (P2p);    
      \draw
      (0,0.6) node[circle, draw] (C1) {2}
      (0,1.0) node[circle, draw] (C2) {\phantom{0}}
      (1,0.75) node[circle, draw](E2) {\phantom{0}}
      (C2) edge[bend left=20] (E2) (C1) edge[bend left=20] (E2) (C1) edge[bend right=20] (E2);
      \draw[->]        (0.5,0.35)   -- (0.5,0.1);
      \draw (.5,-.5) node {\(\Gamma_{2}\)};
    \end{scope}
    \begin{scope}[xshift=3.5cm]
      \draw (0,0) node[circle, draw] (P1) {\phantom{0}} (1,0) node[circle, draw](P2) {\phantom{0}} (P1) edge (P2);
      \draw node[below=.25cm of P2] (P2p) {} (P2) edge (P2p);
      \draw
      (1,0.6) node[circle, draw] (C2) {\phantom{0}}
      (1,1.0) node[circle, draw] (C1) {1}
      (0,0.75) node[circle, draw](E1) {1}
      (C2) edge[bend left=20] (E1) (C1) edge[bend left=20] (E1) (C1) edge[bend right=20] (E1);
      \draw[->]        (0.5,0.35)   -- (0.5,0.1);
      \draw (.5,-.5) node {\(\Gamma_{3}\)};
    \end{scope}
    \begin{scope}[xshift=5.25cm]
      \draw (0,0) node[circle, draw] (P1) {\phantom{0}} (1,0) node[circle, draw](P2) {\phantom{0}} (P1) edge (P2);
      \draw node[below=.25cm of P2] (P2p) {} (P2) edge (P2p);
      \draw
      (1,0.6) node[circle, draw] (C2) {\phantom{0}}
      (1,1.0) node[circle, draw] (C1) {1}
      (0,0.6) node[circle, draw] (E1) {2}
      (0,1.0) node[circle, draw] (E2) {\phantom{0}}
      (C2) edge (E1) (C1) edge (E1) (C1) edge (E2);
      \draw[->]        (0.5,0.35)   -- (0.5,0.1);
      \draw (.5,-.5) node {\(\Gamma_{4}\)};
    \end{scope}
  \end{tikzpicture}
  \caption{The dual graphs of the admissible covers that give Kulikov degenerations in the case \(n = 1\) and \(k = 1\).}
  \label{fig:dualgraphs11}
\end{figure}
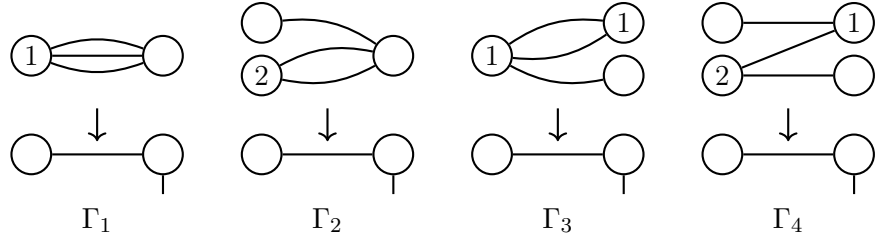

\begin{proposition}\label{prop:period_11}
  The extended period map \(\overline {\mathcal H} \dashrightarrow \overline{\D_{\rho}/\Gamma_{\rho}}^{\rm BB}\) maps \(\Delta_i\) to the \(i\)-th cusp in \Cref{subsec:BB_11}.
\end{proposition}
\begin{proof}
  The proof is analogous to the proof of \Cref{prop:period_01}, based on the calculation of the primitive Picard lattices in \Cref{prop:lattices}.
\end{proof}

\subsection{Stable models and the KSBA semifan}\label{subsec:stable_model_11}
We now identify the KSBA semifan \(\mathfrak F\) that gives the KSBA compactification \(\overline{F}^{\rm KSBA}_{\rho}\).
We follow the notation of \Cref{thm:ksba01}.
\begin{theorem}\label{thm:ksba11}
    The space \(\overline{F}_{\rho}^{\rm KSBA}\) is isomorphic to the semi-toroidal compactification for the following semisfan \(\mathfrak F_J\):
  \[
    \mathfrak F_J = 
    \begin{cases}
      \langle  A_2^{\oplus 4} \rangle^{\rm sat} & \text{ for the cusp with } (J^{\perp}_{T_{\rho}}/J)^{\rm root} = E_6 \oplus A_2^{\oplus 4}\\
      \langle  A_2^{\oplus 3} \rangle^{\rm sat} & \text{ for the cusp with } (J^{\perp}_{T_{\rho}}/J)^{\rm root} = E_8 \oplus A_2^{\oplus 3} \\
      \langle A_2\rangle^{\rm sat} & \text{ for the cusp with } (J^{\perp}_{T_{\rho}}/J)^{\rm root} = E_6^{\oplus 2} \oplus A_2 \\
      0 &\text{ for the cusp with } (J^{\perp}_{T_{\rho}}/J)^{\rm root} = E_8 \oplus E_6.
    \end{cases}
  \]  
\end{theorem}
We devote the rest of \Cref{subsec:stable_model_11} to the proof of \Cref{thm:ksba11}.
We use notation analogous to that introduced after \Cref{thm:ksba01}.
\subsubsection{The case \(i = 1\)}\label{sec:11i1}
In this case, \(C_0\) and \(C_1\) are smooth curves of genus \(1\) and \(0\), respectively.
From \Cref{sec:k0-13}, we recall that \(X_0\) is the blow-up of \(\overline X_0\), a \(\delpezzo_3\), in 3 points of a \((\Z/3\Z)\)-orbit.
From \Cref{sec:k1-03}, we recall that \(X_1\) is the blow-up of \(\overline X_1 = \P^2\) in 9 points of three \((\Z/3\Z)\)-orbits.
We also recall that \(R_0 \subset X_0\) is nef and it is trivial precisely on the exceptional curves of \(X_0 \to \overline X_0\).
Similarly, \(R_1 \subset X_1\) is also nef and it is trivial precisely on the exceptional curves of \(X_1 \to \overline X_1\).
Therefore, \(\overline X = \overline X_{0} \cup \overline X_1\) is the stable model (see \Cref{fig:kulstab111}).
By the Torelli theorem for anti-canonical pairs, the moduli of \(\overline X\) is determined (possibly up to a finite choice) by the restriction of \(\overline\psi\) to the \(E_6\)-summand.
It follows that the translates of \(A_2^{\oplus 4} \otimes_{\Z[\zeta_3]}E \subset \mathcal{A}_J\) are contracted in \(\overline F_{\rho}^{\rm KSBA}\).
Therefore, \(\mathfrak F_J \subset J^{\perp}_{T_{\rho}}/J\) is the saturation of the \(A_2^{\oplus 4}\) summand.
\begin{figure}[ht]
  \centering
  \begin{tikzpicture}
    \draw[thick]  (0,0) -- (0,2) -- (1.5,1.5) -- (1.5,-0.5) -- (0,0) -- (0,2) -- (-1.5,1.5) -- (-1.5,-0.5) -- (0,0);
    \draw[dashed] (0,1.75) -- (1.5,1.25) (0,1.70) -- (1.5,1.20) (0,1.65) -- (1.5,1.15)
    (0,1.45) -- (1.5,0.95)     (0,1.40) -- (1.5,0.90) (0,1.35) -- (1.5,0.85)
    (0,1.15) -- (1.5,0.65) (0,1.1) -- (1.5,0.6)     (0,1.05) -- (1.5,0.55);
    \draw[dashed] (0,0.2) -- (-1.5,-0.3) (0,0.3) -- (-1.5,-0.2) (0,0.4) -- (-1.5,-0.1) ;
    \begin{scope}[yshift=-0.5em]
      \draw[color=red, thick] (0,0.8) -- (0.5,0.64) (0,0.9) -- (0.5,0.74) (0,1.0) -- (0.5,0.84);
      \draw[color=red, thick] (0,0.8) -- (-0.5,0.64) (0,0.9) -- (-0.5,0.74) (0,1.0) -- (-0.5,0.84);
    \end{scope}
    \draw (1.7,0.75) node (K) {};
    \begin{scope}[xshift=5cm, yshift=-0.25cm]
      \draw[thick]  (0,0) -- (1.5,1) -- (0,2) -- (0,0) -- (-1.5,1) -- (0,2) -- (0,0);
      \draw[color=red, thick] (0,1.2) -- (0.5,1.04) (0,1.1) -- (0.5,0.94) (0,1.0) -- (0.5,0.84);
      \draw[color=red, thick] (0,1.2) -- (-0.5,1.04) (0,1.1) -- (-0.5,0.94) (0,1.0) -- (-0.5,0.84);
      \draw (-1.7, 1) node (S) {};      
    \end{scope}
    \draw[thick, ->] (K) edge (S);
  \end{tikzpicture}
  \caption{The stable model in the case \(i = 1\) is obtained by contracting a \((\Z/3\Z)\)-orbit of \((-1)\)-curves on the left (dashed) and three such \((\Z/3\Z)\)-orbits on the right (dashed).  All are disjoint from the divisor (red).}\label{fig:kulstab111}
\end{figure}
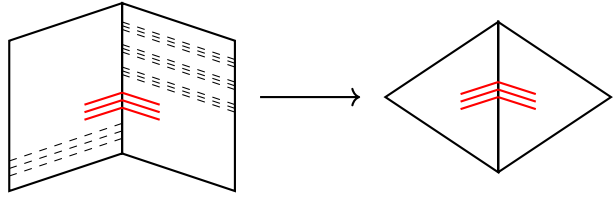
\subsubsection{The case \(i = 2\)}\label{sec:11i2}
In this case, \(C_0 = P_0 \sqcup Q_0\), where \(Q_0\) is a smooth curve of genus 2, and \(C_1\) is a smooth curve of genus \(0\).
From \Cref{sec:k0-01-22}, we know that \(X_0\) is the blow-up \(\overline X_0\), which is a \(\delpezzo_1\), in a \((\Z/3\Z)\)-fixed point.
From \Cref{sec:k1-03}, we know that \(X_1\) is the blow-up of \(\overline X_1 \cong \P^2\) in 9 points which form three \((\Z/3\Z)\)-orbits.
The divisor \(R_0 \subset X_0\) is negative on \(P_0\), which is a \((-1)\)-curve on \(X_0\).
The divisor \(R_1 \subset X_1\) is nef (and has self-intersection \(1\)).

We perform an M1 modification along \(P_0\), obtaining \(X' = X_0' \cup X_1'\).
Let \(R' \subset X'\) be the proper transform of \(R\).
Then \(R'\) is nef, so we obtain the stable model by taking the image of \(X'\) under the map given by (large multiples of) \(R'\).

Finding the stable image is a bit subtle.
The divisor \(R'_0 \subset X'_0 = \overline X_0\) is ample.
But the divisor \(R'_1 \subset X'_1\) is nef but \emph{not} big.
Indeed, its self intersection is \(0\).
It induces a contraction \(X'_1 \to \P^1\), which restricts to a degree 2 morphism on the double curve \(D = X'_0 \cap X'_1\).
Hence, the entire component \(X'_1\) is contracted to a curve in the stable model.
Owing to the contraction of \(X'_1\), the points in the fiber of the degree 2 map \(D \to \P^1\) are identified in the stable image.
The stable image is thus a non-normal surface \(\overline X\), whose normalisation is \(X'_0 = \overline X_0\); the map \(\overline X_0 \to \overline X\) is an isomorphism away from \(D\), but it folds \(D\) to a \(\P^1\) by a degree 2 map (see \Cref{fig:kulstab112}).
The smooth curve \(R'_0\) of genus \(2\) maps to a nodal curve of arithmetic genus \(2\) on \(\overline X\); its two intersection points with \(D\) are identified to form a node.

\begin{figure}[ht]
  \centering
  \begin{tikzpicture}
    \draw[thick]  (0,0) -- (0,2) -- (1.5,1.5) -- (1.5,-0.5) -- (0,0) -- (0,2) -- (-1.5,1.5) -- (-1.5,-0.5) -- (0,0);
    \draw[dashed] (0,1.75) -- (1.5,1.25) (0,1.70) -- (1.5,1.20) (0,1.65) -- (1.5,1.15)
    (0,1.45) -- (1.5,0.95)     (0,1.40) -- (1.5,0.90) (0,1.35) -- (1.5,0.85)
    (0,1.15) -- (1.5,0.65) (0,1.1) -- (1.5,0.6)     (0,1.05) -- (1.5,0.55);
    \begin{scope}[yshift=-0.5em]
      \draw[color=red, thick] (0,0.6) -- (0.5,0.44) (0,0.9) -- (0.5,0.74)  (0,1.0) -- (0.5,0.84);
      \draw[color=red, thick]  (0,0.9) -- (-0.5,0.74) (0,1.0) -- (-0.5,0.84);
      \draw[color=red, thick, dashed] (0,0.6) -- (-0.5,0.44);
    \end{scope}
    \draw (1.7,0.75) node (K) {};
    \begin{scope}[xshift=5cm, yshift=-0.25cm]
      \draw[thick]  (0,0) -- (1.5,0) -- (1.5,2) -- (0,2) -- (0,0) -- (-1.5,1) -- (0,2) -- (0,0);
      \begin{scope}[yshift=-0.5em]
        \draw[color=red, thick] (0,1.2) -- (0.5,1.04) (0,1.1) -- (0.5,0.94); 
        \draw[color=red, thick] (0,1.2) -- (-0.5,1.04) (0,1.1) -- (-0.5,0.94);
        \draw[color=red, thick]  plot[smooth, tension=1] coordinates {(0.1,0.6) (0.2,0.8) (0.5,0.8)};
        \draw[color=blue, dashed, thick] (0,0.8) -- (0.5,0.64);
      \end{scope}
      \draw (-1.7, 1) node (S) {};     
      \draw (1.7,1) node (S2) {};
    \end{scope}
    \begin{scope}[xshift=8.5cm, xscale=0.75]
     \draw[thick] plot[smooth, tension=2] coordinates {(0,0) (0.5,1.5) (1,-.25)};
     \draw[thick] (3,1.5) -- (3,0);
     \draw[thick] (1,-.25) -- (3,0);
     \draw[dashed] (0,0) -- (3,0);
     \draw[thick] (0,0) -- (1,0);
     \draw[thick] (0.5,1.5) -- (3,1.5);
     \draw[thick, red] (2,0.75) -- (3,0.75) (2,0.5) -- (3,0.75);
     \draw (-0.5,0.75) node (T) {};
   \end{scope}
    \draw[thick, dashed, ->] (K) edge (S);
    \draw[thick, ->] (S2) edge (T);
  \end{tikzpicture}
  \caption{The stable model in the case \(i = 2\) is obtained by an M1 modification along a \((-1)\)-curve on the left (dashed red) followed by a morphism that contrats the entire right-hand component, and folds the double curve into a \(\P^1\), resulting in a non-normal surface.}\label{fig:kulstab112}
\end{figure}
By the Torelli theorem for anti-canonical pairs (\Cref{prop:dptor}), the isomorphism class of the stable model is equivalent (up to a finite choice) to the isomorphism class of the anti-canonical pair \((\overline X_0, D)\), whose period is the restriction of \(\psi\) to the \(E_8\)-summand.
It follows that translates of \(A_2^{\oplus 3} \otimes_{\Z[\zeta_3]} E \subset \mathcal{A}_J\) are contracted in \(\overline F_{\rho}^{\rm KSBA}\).
Therefore, \(\mathfrak F_J \subset J^{\perp}_{T_{\rho}}/J\) is the saturation of the \(A_2^{\oplus 3}\) summand.

\subsubsection{The case \(i = 3\)}\label{sec:11i3}
In this case, \(C_0\) is a smooth curve of genus \(1\) and \(C_1 = P_1 \cup Q_1\), where \(Q_1\) is a smooth curve of genus \(1\).
From \Cref{sec:k0-13}, we know that \(X_0\) is the blow up of \(\overline X_0\), which is a \(\delpezzo_3\), in 3 points forming a \((\Z/3\Z)\)-orbit.
From \Cref{sec:k1-01-12}, we know that \(X_1\) is the blow up of \(\overline X_1\), which is a \(\delpezzo_3\).
The exceptional locus of the blow-up is a \((-1,-2,-2)\)-chain of which the first curve is \(P_1\).
Let \(E_1\) and \(E_2\) be the next two curves.
The divisor \(R_1 \subset X_1\) is negative on \(P_1\).
We do an M1 modification of \(X\) along \(P_1\), obtaining \(X' = X_0' \cup X_1'\).
The proper transform \(R' \subset X'\) is nef and trivial on the exceptional divisors of \(X_0 \to \overline X_0\) (which are unaffected by the M1 modification) and on the image \(E_1'\) of \(E_1\), which is now a \((-1)\)-curve.
Contracting these \((-1)\)-curves gives the stable model \(\overline X' = \overline X_0' \cup \overline X_1'\) (see \Cref{fig:kulstab113}).

\begin{figure}[ht]
  \centering
  \begin{tikzpicture}
    \draw[thick]  (0,0) -- (0,2) -- (1.5,1.5) -- (1.5,-0.5) -- (0,0) -- (0,2) -- (-1.5,1.5) -- (-1.5,-0.5) -- (0,0);
    \draw[dashed] (0,1.2) -- (-1.5,0.7) (0,1.3) -- (-1.5,0.8) (0,1.4) -- (-1.5,0.9) ;
    \begin{scope}[yshift=-0.3em]
      \draw[color=red, thick, dashed] (0,0.6) -- (0.5,0.44);
      \draw[dashed, thick] plot[smooth, tension=1] coordinates {(0.2,0.4) (0.5,0.55) (0.9,0.5)};
      \draw[color=red, thick] (0,0.9) -- (0.5,0.74)  (0,1.0) -- (0.5,0.84);
      \draw[color=red, thick]  (0,0.9) -- (-0.5,0.74) (0,1.0) -- (-0.5,0.84);
      \draw[color=red, thick] (0,0.6) -- (-0.5,0.44);
    \end{scope}
    \draw (1.7,0.75) node (K) {};
    \begin{scope}[xshift=5cm, yshift=-0.25cm]
      \draw[thick]  (0,0) -- (1.5,1) -- (0,2) -- (0,0) -- (-1.5,1) -- (0,2) -- (0,0);
        \draw[color=red, thick] (0,1.2) -- (0.5,1.04) (0,1.1) -- (0.5,0.94); 
        \draw[color=red, thick] (0,1.2) -- (-0.5,1.04) (0,1.1) -- (-0.5,0.94);
        \draw[color=red, thick]  plot[smooth, tension=1] coordinates {(-0.1,0.6) (-0.2,0.8) (-0.5,0.8)};
        \draw[color=blue, dashed, thick] (0,0.8) -- (-0.5,0.64);
        \draw (-1.7, 1) node (S) {};      
    \end{scope}
    \draw[thick, dashed, ->] (K) edge (S);
  \end{tikzpicture}
  \caption{The stable model in the case \(i = 3\) is obtained by an M1 modification along a \((-1)\)-curve on the right (dashed red) followed by contracting a \((-1)\)-curve on the right (dashed) and three \((-1)\)-curves on the left (dashed). The divisor is shown in red (dashed and solid).}\label{fig:kulstab113}
\end{figure}
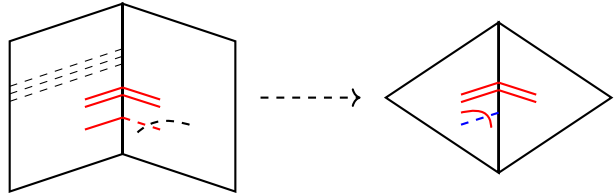
By the Torelli theorem for anti-canonical pairs (\Cref{prop:dptor}), the moduli of the stable model is determined (possibly up to a finite choice) by the restriction of \(\psi\) to the \(E_6\)-summands.
It follows that the translates of \(A_2 \otimes_{\Z[\zeta_3]} E\) are contracted in \(\overline F_{\rho}^{\rm KSBA}\).
Therefore, \(\mathfrak F_J \subset J^{\perp}_{T_{\rho}}/J\) is the saturation of the \(A_2\) summand.

\subsubsection{The case \(i = 4\)}\label{sec:11i4}
In this case, \(C_0 = P_0 \sqcup Q_0\) and \(C_1 = P_1 \sqcup Q_1\), where \(Q_0\) is a smooth curve of genus 2 and \(Q_1\) is a smooth curve of genus \(1\).
From \Cref{sec:k0-01-22}, we know that \(X_0\) is the blow-up of \(\overline X_0\), which is a \(\delpezzo_1\), in a \((\Z/3\Z)\)-fixed point.
From \Cref{sec:k1-01-12}, we know that \(X_1\) is the blow-up of \(\overline X_1\), which is a \(\delpezzo_3\), with an exceptional divisor forming a \((-1,-2,-2)\)-chain of which the first curve is \(P_1\).
The divisor \(R_0\) is negative on \(P_0\) and \(R_1\) is negative on \(P_1\).
We perform M1 modifications along \(P_0\) and \(P_1\), obtaining \(X' = X_0' \cup X_1'\).
Let \(R'\) be the proper transform of \(R\).
Then \(R' \subset X'\) is nef, and zero only on the image of the middle curve in the \((-1,-2,-2)\)-chain, which is now a \((-1)\)-curve.
Contracting this curve gives the stable model \(\overline X' = \overline X_0' \cup \overline X_1'\) (see \Cref{fig:kulstab114}).
\begin{figure}[ht]
  \centering
  \begin{tikzpicture}
    \draw[thick]  (0,0) -- (0,2) -- (1.5,1.5) -- (1.5,-0.5) -- (0,0) -- (0,2) -- (-1.5,1.5) -- (-1.5,-0.5) -- (0,0);
    \draw[color=red, thick, dashed] (0,0.6) -- (0.5,0.44);
    \draw[color=red, thick] (0,0.9) -- (0.5,0.74);
    \draw[color=red, thick]  (0,0.9) -- (-0.5,0.74) ;
    \draw[color=red, thick, dashed] (0,1.3) -- (-0.5,1.14);
    \draw[color=red, thick] (0,0.6) -- (-0.5,0.44)   (0,1.3) -- (0.5,1.14);
    \draw[dashed, thick] plot[smooth, tension=1] coordinates {(0.2,0.4) (0.5,0.55) (0.9,0.5)};
    \draw (1.7,0.75) node (K) {};
    \begin{scope}[xshift=5cm, yshift=-0.25cm]
      \draw[thick]  (0,0) -- (1.5,1) -- (0,2)-- (0,0) -- (-1.5,1) -- (0,2) -- (0,0);
      \draw[color=red, thick] (0,1.1) -- (0.5,0.94); 
      \draw[color=red, thick] (0,1.1) -- (-0.5,0.94);
      \draw[color=red, thick]  plot[smooth, tension=1] coordinates {(0.15,1.5) (0.25,1.2) (0.5,1.1)};
      \draw[color=red, thick]  plot[smooth, tension=1] coordinates {(-0.15,.6) (-0.25,.8) (-0.5,.8)};
      \draw[color=blue, dashed, thick] (0,1.4) -- (0.5,1.24) (0,0.8) -- (-0.5,0.64) ;
      \draw (-1.7, 1) node (S) {};      
    \end{scope}
    \draw[thick, dashed, ->] (K) edge (S);
  \end{tikzpicture}
  \caption{The stable model in the case \(i = 4\) is obtained by M1 modifications along the \((-1)\)-curves on the left and the right (dashed red) followed by the contraction of a \((-1)\)-curve on the right (dashed).  The divisor is shown in red (dashed and solid).}\label{fig:kulstab114}
\end{figure}
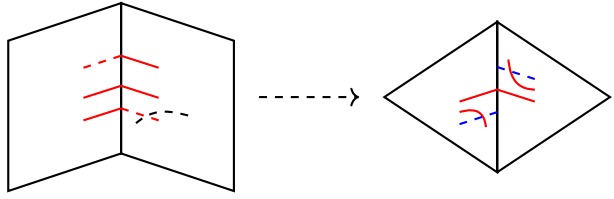
By the Torelli theorem for anti-canonical pairs (\Cref{prop:dptor}), the moduli of the stable model is determined by \(\psi\).
So there is no contraction while mapping to \(\overline F_{\rho}^{\rm KSBA}\) and we get \(\mathfrak F_J = 0\).

\begin{proof}[Conclusion of the proof of \Cref{thm:ksba11}]
  Having settled the cases of the four cusps in \Cref{sec:11i1}, \Cref{sec:11i2}, \Cref{sec:11i3}, and \Cref{sec:11i4}, the proof of \Cref{thm:ksba11} is complete.
\end{proof}

\section{Moduli space 4: \texorpdfstring{\(g=2\)}{g = 2} and \texorpdfstring{\((n,k) = (2,1)\)}{(n,k) = (2,1)}} \label{sec:n2k1}
In this section, we identify the KSBA compactification ofthe moduli space of K3 surfaces with a non-symplectic automorphism of order 3 with \(2\) isolated fixed points and \(1\) fixed curve (so \(n = 2\) and \(k = 1\)).

Let \((X,\sigma)\) be generic K3 surface with an automorphism of order \(3\) with \(n = 2\) and \(k = 1\).
From \Cref{prop:triple-tschirnhausen-dominant}, we know that \((X,\sigma)\) arises from the pinched triple Tschirnhausen construction applied to \((\phi \colon C \to \P^1, p_1,p_2)\), where \(C\) is a smooth curve of genus \(2\).
For a generic \(\phi\), the Tschirnhausen embedding identifies \(C\) as a divisor in \(\P^1 \times \P^1\) of class \((3,2)\).

\begin{remark}\label{rem:as21}
  In \cite[Proposition~4.11]{art.sar:08}, Artebani--Sarti give explicit projective equations for \(X\).
  It is the double cover of \(\P^2\) branched along a smooth plane sextic of the form
  \begin{equation}\label{eqn:artebani-sarti-21}
    F_6(x_0,x_1) + F_3(x_0,x_1) x_2^3 + b x_2^6,
  \end{equation}
  where \(F_i\) is a generic form of degree \(i\).
  We now reconcile this description with ours.
  Consider the \(\P(1,1,3)\) obtained by taking the quotient of \(\P^2\) by the order 3 automorphism
  \[ [x_0:x_1:x_2] \mapsto [x_0:x_1:\zeta_3x_2].\]
  Let \(h\) be the positive generator of the Weil divisor class group of \(\P(1,1,3)\).
  Then the cubic \eqref{eqn:artebani-sarti-21} is the pull-back of a curve of class \(6h\).
  The branch divisor of the triple cover \(\P^2 \to \P(1,1,3)\) is of class \(3h\).
  We have the diagram
  \[X \xrightarrow{2:1} \P^2 \xrightarrow{3:1} \P(1,1,3).\]
  We swap the triple and double covers (see \Cref{fig:as21} for reference).
  Let \(Y\) be the double cover of \(\P(1,1,3)\) branched along a curve of class \(6h\).
  Then \(X\) is the triple cover of \(Y\) branched along a curve that is the pull-back of a curve of class \(3h\); that is, 
  \[ X \xrightarrow{3:1} Y \xrightarrow{2:1} \P(1,1,3).\]
  To understand \(Y\), let \(\F_3 \to \P(1,1,3)\) be the minimal resolution of the \(\frac{1}{3}(1,1)\) singularity.
  Observe that the pull-back of the Cartier divisor \(3h\) to \(\F_3\) is \(\sigma + 3f\).
  Set \(\widetilde Y = \F_3 \times_{\P(1,1,3)} Y\).
  Then \(\widetilde Y \to \F_3\) is a double cover branched along a curve \(B \subset \F_3\) of class \(2\sigma+6f\) (the curve \(B\) is drawn in red in \Cref{fig:as21}).
  The surface \(\widetilde Y\) contains two disjoint \(-3\) curves, say \(\widetilde{\sigma}_1\) and \(\widetilde{\sigma}_2\); these are the pre-images of the directrix \(\sigma \subset \F_3\).
  The map \(\widetilde Y \to Y\) contracts them.
  The curve \(B\) is hyperelliptic of genus 2; its projection \(B \to \P^1\) has 6 branch points.
  Therefore, the composite \(\widetilde Y \to \F_3 \to \P^1\) is generically a \(\P^1\)-bundle with 6 nodal fibers of the form \(\P^1 \cup \P^1\).
  Let \(L \subset \P(1,1,3)\) be a general curve of class \(3h\) .
  Then its pre-image in \(\F_3\) is a general curve of class \(\sigma + 3f\) (drawn in blue in \Cref{fig:as21}).
  It intersects \(B\) in 6 points (distinct from the 6 Weierstrass points of \(B\)).
  The pre-image \(\widetilde C\) of \(L\) in \(\widetilde Y\) (also blue) is a hyperelliptic curve of genus 2, which is disjoint from the nodes of the 6 reducible fibers of \(\widetilde Y \to \P^1\).
  Contract a rational component of each of the 6 reducible fibers of \(\widetilde Y \to \P^1\) so that three of the contracted components meet \(\widetilde{\sigma}_1\) and three meet \(\widetilde{\sigma}_2\) (drawn as dashed lines in \Cref{fig:as21}).
  For \(i = 1,2\), let \(\sigma_i\) be the images of \(\widetilde \sigma_i\).
  Observe that the contracted surface is a \(\P^1\)-bundle over \(\P^1\) with sections \(\sigma_i\) satisfying \(\sigma_i^2 = 0\).
  Therefore, it is \(\P^1 \times \P^1\).
  The image of \(C\) of \(\widetilde C\) in \(\P^1 \times \P^1\) is a curve of class \(3 \sigma_i + 2 f\).
  Let \(C \to \P^1\) be the degree 3 projection and let \(p_{i} \in \P^1\) be the image of \(\sigma_i\).
  Then we see that \(X\) is obtained from \((C \to \P^1, p_1,p_2)\) by the triple Tschirnhausen construction pinched at \(p_1\) and \(p_2\).
\end{remark}
\begin{figure}[ht] 
  \begin{tikzpicture}[yscale=2, xscale=2]
    \draw (0.6,0) node (F3)[below] {\(\F_3\)};
    \draw(0.6,-1) node (P113) {\(\P(1,1,3)\)};
    \draw[thick]
    (0,0) -- (0,1) -- (1.2,1) -- (1.2,0) -- (0,0);
    \draw (0,0.8) -- (1.2,0.8) node[right] {\tiny\(\sigma\)};
    \draw[red] plot[smooth, tension=2] coordinates {(0,.5) (0.1, .4) (0,.3)};
    \draw[red] (0.2,0.4) circle (.1) (0.4,0.4) circle (0.1) (0.6,0.4) circle (0.1) (0.8,0.4) circle (0.1) (1.0,0.4) circle (0.1);
    \draw[red] plot[smooth, tension=2] coordinates {(1.2,.5) (1.1, .4) (1.2,.3)};
    \draw (1.2, 0.4) node[right] {\tiny \(B\)};
    \draw[blue] plot[smooth] coordinates {(0,0.1) (0.2,0.4) (0.4,0.1) (0.6,0.4) (0.8,0.1) (1.0,0.4) (1.2,0.1)};
    \draw (1.2,0.1) node[right] {\tiny \(L\)};
    \begin{scope}[xshift=-2cm]
      \draw(0.6,0) node (Yt) [below] {\(\widetilde Y\)};
      \draw(0.6,-1) node (Y) {\(Y\)};
      \draw[thick]
      (0,0) -- (0,1) -- (1.2,1) -- (1.2,0) -- (0,0);
      \draw (0,0.8) -- (1.2,0.8) node[right] {\tiny\(\widetilde{\sigma}_1\)};
      \draw (0,0.2) -- (1.2,0.2) node[right] {\tiny\(\widetilde{\sigma}_2\)};
      \draw[dashed](0.08,1.0) -- (0.15,.45);
      \draw (0.15,0.55) -- (0.08,0.0);
      \draw[blue] plot[smooth, tension=2] coordinates {(0,0.6) (.05,0.5) (0,0.4)};
      \draw[blue] (0.15, 0.5) circle (0.1);
      \begin{scope}[xshift=0.2cm]
        \draw[dashed](0.08,1.0) -- (0.15,.45);
        \draw (0.15,0.55) -- (0.08,0.0);
        \draw[blue] (0.15, 0.5) circle (0.1);
      \end{scope}
      \begin{scope}[xshift=0.4cm]
        \draw[dashed](0.08,1.0) -- (0.15,.45);
        \draw (0.15,0.55) -- (0.08,0.0);
        \draw[blue] (0.15, 0.5) circle (0.1);
      \end{scope}
      \begin{scope}[xshift=0.6cm]
        \draw(0.08,1.0) -- (0.15,.45);
        \draw[dashed](0.15,0.55) -- (0.08,0.0);
        \draw[blue] (0.15, 0.5) circle (0.1);
      \end{scope}
      \begin{scope}[xshift=0.8cm]
        \draw(0.08,1.0) -- (0.15,.45);
        \draw[dashed](0.15,0.55) -- (0.08,0.0);
           \draw[blue] (0.15, 0.5) circle (0.1);
      \end{scope}
      \begin{scope}[xshift=1.0cm]
        \draw(0.08,1.0) -- (0.15,.45);
        \draw[dashed](0.15,0.55) -- (0.08,0.0);
      \end{scope}
      \draw[blue] plot[smooth, tension=2] coordinates {(1.2,0.6) (1.05,0.5) (1.2,0.4)};
      \draw (1.2,0.5) node[right] {\tiny\(\widetilde C\)};
    \end{scope}
    \begin{scope}[xshift=-4cm]
          \draw[thick]
          (0,0) -- (0,1) -- (1.2,1) -- (1.2,0) -- (0,0);
      \draw[blue] plot[smooth, tension=2] coordinates {(0,0.6) (.05,0.5) (0,0.4)};
      \draw[blue] (0.15, 0.5) circle (0.1);
      \draw[blue] (0.35, 0.5) circle (0.1);
      \draw[blue] (0.55, 0.5) circle (0.1);
      \draw[blue] (0.75, 0.5) circle (0.1);
      \draw[blue] (0.95, 0.5) circle (0.1);                  
      \draw[blue] plot[smooth, tension=2] coordinates {(1.2,0.6) (1.05,0.5) (1.2,0.4)};
      \draw[fill] (0.15,0.6) circle (0.02) edge [->, dashed, bend right=30,shorten <=.1cm, shorten >=.1cm] (-0.5,0.6);
      \draw[fill] (0.35,0.6) circle (0.02) edge [->, dashed, bend right=30,shorten <=.1cm, shorten >=.1cm] (-0.5,0.6);
      \draw[fill] (0.55,0.6) circle (0.02) edge [->, dashed, bend right=30,shorten <=.1cm, shorten >=.1cm] (-0.5,0.6);
      \draw[fill] (0.75,0.4) circle (0.02) edge [->, dashed, bend right=-30,shorten <=.1cm, shorten >=.1cm] (-0.5,0.4);
      \draw[fill] (0.95,0.4) circle (0.02) edge [->, dashed, bend right=-30,shorten <=.1cm, shorten >=.1cm] (-0.5,0.4);
      \draw[fill] (1.15,0.4) circle (0.02) edge [->, dashed, bend right=-30,shorten <=.1cm, shorten >=.1cm] (-0.5,0.4);
      \draw (1.2,0.5) node[right] {\tiny\(C\)};
      \draw[thick] (-0.5,1) -- (-0.5,0) node[below] {\(\P^1\)} ;
      \draw[fill] (-0.5,0.6) circle (0.02) node (p1) [left] {\tiny\(p_1\)};
      \draw[fill] (-0.5,0.4) circle (0.02) node (p2) [left] {\tiny\(p_2\)};
      \draw (0.6, 0.0) node (P1P1) [below] {\(\P^1 \times \P^1\)};
       \draw(0.6,-1.0) node (X) {\(X\)};
     \end{scope}
     \draw(X) edge[->] node[above]{\tiny \(3:1\)} (Y)
     (Y) edge[->] node[above]{\tiny\(2:1\)} (P113);
     \draw(F3) edge[->] (P113);
     \draw(Yt) edge[->] node[below]{\tiny \(2:1\)} (F3);
     \draw(Yt) edge[->] (Y);
     \draw(Yt) edge[->] node[below]{\tiny contract dashed}(P1P1);
     \draw(X) edge[->, dashed] node[left]{\tiny \(3:1\)} (P1P1);
   \end{tikzpicture}
   \caption{The figure above reconciles Artebani--Sarti's description of K3s of type \((2,1)\) with the triple Tschirnhausen construction.  See \Cref{rem:as21}.}\label{fig:as21}
\end{figure}
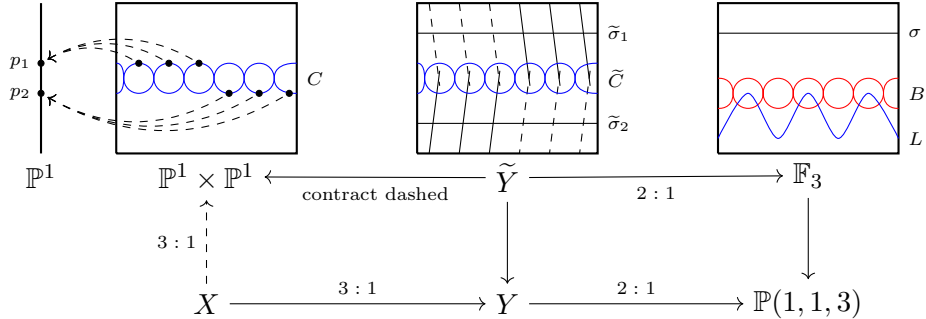
\begin{remark}
  Contracting the six \((-1)\)-curves on \(\widetilde Y\) intersecting \(\widetilde\sigma_2\) yields an \(\F_3\).
  This transformation reconciles our construction with \cite[\S~8.1]{ma.oha.tak:15}.
\end{remark}

Fix an isometry of \(H^2(X, \Z)\) with the K3 lattice \(L\) and let \(\rho\) be the automorphism of \(L\) induced by \(\sigma\).
We recall the Hodge type of \(\rho\) from \Cref{tab:ASlattices}.
The lattice \(S_{\rho} \subset H^2(X, \Z)\) of \(\rho\)-fixed vectors is given by
\[ S_{\rho} = U(3) \oplus A_2^{\oplus 2}. \]
Its orthogonal complement \(T_{\rho}\) is
\[ T_{\rho} = U \oplus U(3) \oplus E_{6}^{\oplus 2}.\]
The automorphism \(\rho\) acts on \(U \oplus U(3)\) as described in \Cref{subsec:eis_cusp} and on the \(E_6\)-summand by the unique order 3 automorphism without non-zero fixed vectors \cite[Ch~II~\S~2.6]{con.slo:99}.

\begin{remark}\label{rem:srhobasis21}
  We describe a basis for \(S_{\rho} = U(3) \oplus A_2^{\oplus 2}\) analogous to \Cref{rem:srhobasis11}.
  Let \(\widehat Y \to \P^1 \times \P^1\) be the blow-up in the six points points of \(C\) that map to \(p_1\) or \(p_2\) (this notation is consistent with \Cref{fig:as21}).
  Let \(E_1,E_2,E_3 \subset \widetilde Y\) be the exceptional curves over \(p_1\) and \(F_1,F_2,F_3 \subset \widetilde Y\) over \(p_2\).
  Recall that \(X\) is obtained from \(\widetilde Y\) by taking a triple cover and blowing down two \(-1\) curves.
  For a curve \(\alpha \in \widetilde Y\), let \(\widetilde \alpha\) be the image in \(X\) of the pre-image of \(\alpha\) in the triple cover of \(\widetilde Y\).
  Let \(x \in \P^1 \times \P^1\) be the image of \(E_1\).
  There are two ruling lines of \(\P^1 \times \P^1\) passing through \(x\): one maps to \(p_1\) and the complementary one.
  Let \(\sigma \subset \widetilde Y\) be the proper transform of the complementary one.
  Then \(S_{\rho}\) is spanned by  \(\widetilde \sigma, \widetilde E_1, \widetilde E_2, \widetilde E_3, \widetilde F_1, \widetilde F_2, \widetilde F_3\) modulo the relation \(\sum \widetilde E_i = \sum \widetilde F_i\).
  Set
  \begin{align*}
    e &= \widetilde E_1 + \widetilde E_2 + \widetilde E_3, \text{ and }\\
    f &= \widetilde \sigma + \widetilde F_1 + \widetilde F_2. 
  \end{align*}
  Then we have \[S_{\rho} = \langle e,f \rangle \oplus \langle \widetilde E_2, \widetilde E_3 \rangle \oplus \langle \widetilde F_2, \widetilde{F}_3 \rangle \cong U(3) \oplus A_2 \oplus A_2.\]
\end{remark}
\subsection{Baily--Borel cusps} \label{subsec:BB_21}
Let \(F_{\rho}^{\rm sep} = \left(\D_{\rho} \setminus \Delta_{\rho}\right) / \Gamma_{\rho}\) be the period domain for \(\rho\)-markable K3 surfaces as described in \Cref{subsec:period_map}. 
Recall from \Cref{thm:t11-t21cusps} the classification of cusps for the Baily--Borel compactification \(\overline{\D_{\rho}/\Gamma_{\rho}}^{\rm BB}\) by the root sublattice of \(J^{\perp}_{T_{\rho}}/J\), which is one of:
\begin{enumerate}
   \item \(A_2^{\oplus 6}\),
   \item \(E_6 \oplus A_2^{\oplus 3}\),
   \item \(E_6^{\oplus 2}\), 
   \item \(E_8 \oplus A_2\).
\end{enumerate}

\subsection{Kulikov models}\label{subsec:kulikov_21}
We use the notation in \Cref{subsec:ell_fib_trigonal}.
In particular, we let \(\mathcal H\) be the moduli space of marked triple covers \((\phi \colon C \to \P^1, p_1,p_2)\), where \(C\) is a smooth curve of genus 2 and \(\phi\) is \'etale over \(p_1\) and \(p_2\).
Let \(\overline{\mathcal H}\) be its compactification by marked admissible covers.
For \(i = 1, \dots, 5\), we consider the boundary divisor \(\Delta_i \subset \overline{\mathcal H}\) whose generic point parametrises a degenerate triple cover with the dual graph \(\Gamma_i\) shown in \Cref{fig:dualgraphs21}.

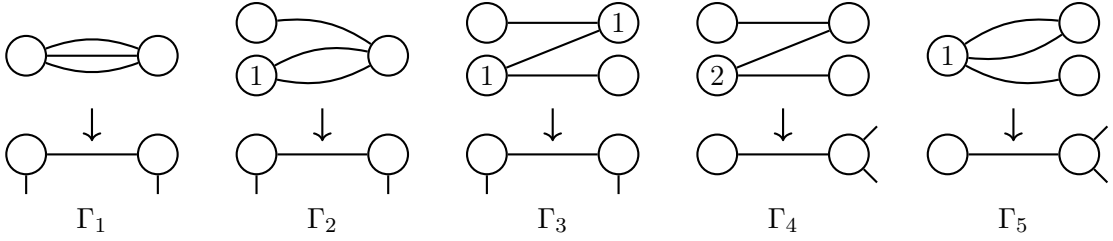
\begin{figure}
  \centering
  \begin{tikzpicture}[thick, scale=1.75, every node/.style={inner sep=0.2em}]
    \draw (0,0) node[circle, draw] (P1) {\phantom{0}} (1,0) node[circle, draw](P2) {\phantom{0}} (P1) edge (P2);
    \draw node[below=.25cm of P2] (P2p) {} (P2) edge (P2p);
    \draw node[below=.25cm of P1] (P1p) {} (P1) edge (P1p);    
    \draw (0,0.75) node[circle, draw] (E1) {\phantom{0}} (1,0.75) node[circle, draw](E2) {\phantom{0}}
    (E1) edge[bend left=20] (E2) (E1) edge (E2) (E1) edge[bend right=20] (E2);
    \draw[->]        (0.5,0.35)   -- (0.5,0.1);
    \draw (.5,-.5) node {\(\Gamma_{1}\)};

    \begin{scope}[xshift=1.75cm]
      \draw (0,0) node[circle, draw] (P1) {\phantom{0}} (1,0) node[circle, draw](P2) {\phantom{0}} (P1) edge (P2);
      \draw node[below=.25cm of P2] (P2p) {} (P2) edge (P2p);
      \draw node[below=.25cm of P1] (P1p) {} (P1) edge (P1p);    
      \draw
      (0,0.6) node[circle, draw] (C1) {1}
      (0,1.0) node[circle, draw] (C2) {\phantom{0}}
      (1,0.75) node[circle, draw](E2) {\phantom{0}}
      (C2) edge[bend left=20] (E2) (C1) edge[bend left=20] (E2) (C1) edge[bend right=20] (E2);
      \draw[->]        (0.5,0.35)   -- (0.5,0.1);
      \draw (.5,-.5) node {\(\Gamma_{2}\)};
    \end{scope}
    \begin{scope}[xshift=3.5cm]
      \draw (0,0) node[circle, draw] (P1) {\phantom{0}} (1,0) node[circle, draw](P2) {\phantom{0}} (P1) edge (P2);
      \draw node[below=.25cm of P2] (P2p) {} (P2) edge (P2p);
      \draw node[below=.25cm of P1] (P1p) {} (P1) edge (P1p);    
      \draw
      (1,0.6) node[circle, draw] (C2) {\phantom{0}}
      (1,1.0) node[circle, draw] (C1) {1}
      (0,0.6) node[circle, draw] (E1) {1}
      (0,1.0) node[circle, draw] (E2) {\phantom{0}}
      (C2) edge (E1) (C1) edge (E1) (C1) edge (E2);
      \draw[->]        (0.5,0.35)   -- (0.5,0.1);
      \draw (.5,-.5) node {\(\Gamma_{3}\)};
    \end{scope}
    \begin{scope}[xshift=5.25cm]
      \draw (0,0) node[circle, draw] (P1) {\phantom{0}} (1,0) node[circle, draw](P2) {\phantom{0}} (P1) edge (P2);
      \draw node[below right=.25cm of P2] (P2p1) {} (P2) edge (P2p1);
      \draw node[above right=.25cm of P2] (P2p2) {} (P2) edge (P2p2);
      \draw
      (1,0.6) node[circle, draw] (C2) {\phantom{0}}
      (1,1.0) node[circle, draw] (C1) {\phantom{0}}
      (0,0.6) node[circle, draw] (E1) {2}
      (0,1.0) node[circle, draw] (E2) {\phantom{0}}
      (C2) edge (E1) (C1) edge (E1) (C1) edge (E2);
      \draw[->]        (0.5,0.35)   -- (0.5,0.1);
      \draw (.5,-.5) node {\(\Gamma_{4}\)};
    \end{scope}
    \begin{scope}[xshift=7.0cm]
      \draw (0,0) node[circle, draw] (P1) {\phantom{0}} (1,0) node[circle, draw](P2) {\phantom{0}} (P1) edge (P2);
      \draw node[below right=.25cm of P2] (P2p1) {} (P2) edge (P2p1);
      \draw node[above right=.25cm of P2] (P2p2) {} (P2) edge (P2p2);
      \draw
      (1,0.6) node[circle, draw] (C2) {\phantom{0}}
      (1,1.0) node[circle, draw] (C1) {\phantom{0}}
      (0,0.75) node[circle, draw](E1) {1}
      (C2) edge[bend left=20] (E1) (C1) edge[bend left=20] (E1) (C1) edge[bend right=20] (E1);
      \draw[->]        (0.5,0.35)   -- (0.5,0.1);
      \draw (.5,-.5) node {\(\Gamma_{5}\)};
    \end{scope}
  \end{tikzpicture}
  \caption{The dual graphs of the admissible covers that give Kulikov degenerations in the case \(n = 2\) and \(k = 1\).}
  \label{fig:dualgraphs21}
\end{figure}

\begin{proposition}\label{prop:period_21}
  The extended period map \(\overline {\mathcal H} \dashrightarrow \overline{\D_{\rho}/\Gamma_{\rho}}^{\rm BB}\) maps \(\Delta_i\) to the \(i\)-th cusp in \Cref{subsec:BB_21} for \(i = 1, \dots, 4\).
  It maps \(\Delta_5\) to the 2nd cusp.
\end{proposition}
\begin{proof}
  The proof is analogous to the proof of \Cref{prop:period_01}, based on the calculation of the primitive Picard lattices in \Cref{prop:lattices}.
\end{proof}

\subsection{Stable models and the KSBA semifan}\label{subsec:stable_model_21}
We now identify the KSBA semifan \(\mathfrak F\) that gives the KSBA compactification \(\overline{F}^{\rm KSBA}_{\rho}\).
We follow the notation of \Cref{thm:ksba01}.
\begin{theorem}\label{thm:ksba21}
    The space \(\overline{F}_{\rho}^{\rm KSBA}\) is isomorphic to the semi-toroidal compactification for the following semisfan \(\mathfrak F_J\):
  \[
    \mathfrak F_J = 
    \begin{cases}
      \langle  A_2^{\oplus 6} \rangle^{\rm sat} & \text{ for the cusp with } (J^{\perp}_{T_{\rho}}/J)^{\rm root} = A_2^{\oplus 6}\\
      \langle  A_2^{\oplus 3} \rangle^{\rm sat} & \text{ for the cusp with } (J^{\perp}_{T_{\rho}}/J)^{\rm root} = E_6 \oplus A_2^{\oplus 3} \\
      0 & \text{ for the cusp with } (J^{\perp}_{T_{\rho}}/J)^{\rm root} = E_6^{\oplus 2}\\
      \langle A_2 \rangle^{\rm sat} &\text{ for the cusp with } (J^{\perp}_{T_{\rho}}/J)^{\rm root} = E_8 \oplus A_2.
    \end{cases}
  \]  
\end{theorem}
We devote the rest of \Cref{subsec:stable_model_21} to the proof of \Cref{thm:ksba21}.
We use notation analogous to that introduced after \Cref{thm:ksba01}.
Many of the following arguments are analogous to those in \Cref{subsec:stable_model_01} and \Cref{subsec:stable_model_11}, so we will be even more brief.

\subsubsection{The case \(i = 1\)}\label{sec:21i1}
From \Cref{sec:k1-03}, we know that each \(X_j \to \P^2\) is a blow-up in 9 points that form three \(\Z/3Z\)-orbits.
The divisor \(R\) is nef and it contracts the 9 exceptional curves on each \(X_j\), yielding \(\P^2 \cup \P^2\) as the stable model (see \Cref{fig:kulstab211}).
Note that the stable model has no moduli, and hence the entire divisor \(\delta_1\) is contracted to a point.
\begin{figure}[ht]
  \centering
  \begin{tikzpicture}
    \draw[thick]  (0,0) -- (0,2) -- (1.5,1.5) -- (1.5,-0.5) -- (0,0) -- (0,2) -- (-1.5,1.5) -- (-1.5,-0.5) -- (0,0);
    \draw[dashed] (0,1.75) -- (1.5,1.25) (0,1.70) -- (1.5,1.20) (0,1.65) -- (1.5,1.15)
    (1.2,1.0) node {\tiny \(\times 3\)};
    \draw[dashed] (0,0.3) -- (-1.5,-0.2) (0,0.35) -- (-1.5,-0.15) (0,0.4) -- (-1.5,-0.1)
    (-1.2, .3) node {\tiny \(\times 3\)};
    \begin{scope}
      \draw[color=red, thick] (0,0.8) -- (0.5,0.64) (0,0.9) -- (0.5,0.74) (0,1.0) -- (0.5,0.84);
      \draw[color=red, thick] (0,0.8) -- (-0.5,0.64) (0,0.9) -- (-0.5,0.74) (0,1.0) -- (-0.5,0.84);
    \end{scope}
    \draw (1.7,0.75) node (K) {};
    \begin{scope}[xshift=5cm, yshift=-0.25cm]
      \draw[thick]  (0,0) -- (1.5,1) -- (0,2) -- (0,0) -- (-1.5,1) -- (0,2) -- (0,0);
      \draw[color=red, thick] (0,1.2) -- (0.5,1.04) (0,1.1) -- (0.5,0.94) (0,1.0) -- (0.5,0.84);
      \draw[color=red, thick] (0,1.2) -- (-0.5,1.04) (0,1.1) -- (-0.5,0.94) (0,1.0) -- (-0.5,0.84);
      \draw (-1.7, 1) node (S) {};      
    \end{scope}
    \draw[thick, ->] (K) edge (S);
  \end{tikzpicture}
  \caption{The stable model in the case \(i = 1\) is obtained by contracting three \((\Z/3\Z)\)-orbits of \((-1)\)-curves on each side (dashed).  All are disjoint from the divisor (red).}\label{fig:kulstab211}
\end{figure}
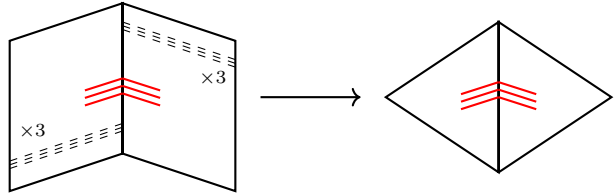
\subsubsection{The case \(i = 2\)}\label{sec:21i2}
Write \(C_0 = P_0 \cup Q_0\), where \(Q_0\) is a smooth curve of genus \(1\).
From \Cref{sec:k1-01-12}, we know that \(X_0\) is the blow-up of \(\overline X_0\), which is a \(\delpezzo_3\), with an exceptional divisor forming a \((-1,-2,-2)\)-chain of which the first curve is \(P_0\).
From \Cref{sec:k1-03}, we know that \(X_1\) is the blow-up of \(\P^2\) at 9 points forming three \(\Z/3\Z\)-orbits.
The divisor \(R_0 \subset X_0\) is negative on \(P_0\).
The divisor \(R_1 \subset X_1\) is nef (and has self-intersection 1).

We perform an M1 modification along M1, obtaining \(X' = X'_0 \cup X'_1\) (see \Cref{fig:kulstab212}).
Let \(R'\) be the proper transform of \(R\).
Then \(R'\) is nef on \(X'\), so we obtain the stable model by taking the image of \(X'\) under the map given by (large multiples of) \(R'\).

Finding the stable image is a bit subtle.
The divisor \(R'_0 \subset X'_0\) is big and nef; it is trivial on the second curve in the \((-1,-2,-2)\)-chain, which is now a \(-1\) curve (drawn as a dashed curve on the left surface in the middle of \Cref{fig:kulstab212}), and positive on all other curves.
This \((-1)\)-curve is contracted in the stable model.
Let \(X'_0 \to \widetilde X_0\) be this contraction.
Note that \(\widetilde X_0\) is the blow-up of the \(\delpezzo_3\) surface \(\overline X_0\) at a \((\Z/3\Z)\)-fixed point on the (anti-canonical) double curve.
The divisor \(R'_1 \subset X'_1\) is nef but \emph{not} big.
Indeed, its self intersection is \(0\).
It induces a contraction \(X'_1 \to \P^1\), which restricts to a degree 2 morphism on the double curve \(D\).
Hence, the entire component \(X'_1\) is contracted to a curve in the stable model.
Owing to the contraction of \(X'_1\), the points in the fiber of the degree 2 map \(D \to \P^1\) are identified in the stable image.
The stable image is thus a non-normal surface \(\overline X\), whose normalization is \(\widetilde X_0\); the map \(\widetilde X_0 \to \overline X\) is an isomorphism away from \(D\), but it folds \(D\) to a \(\P^1\) by a degree 2 map.
The curve \(C_0 = R'_0\) maps to a nodal curve of arithmetic genus \(2\) on \(\overline X\); its two intersection points with \(D\) are identified to form a node.

Up to a finite choice, the isomorphism class of the stable model is equivalent to the isomorphism class of the anti-canonical pair \((\overline X_0, D)\), whose period is the restriction of \(\psi\) to the \(E_6\)-summand.
It follows that translates of \(A_2^{\oplus 3} \otimes_{\Z[\zeta_3]} E\) are contracted in \(\overline F_{\rho}^{\rm KSBA}\).
Therefore, \(\mathfrak F_J \subset J^{\perp}_{T_{\rho}}/J\) is the saturation of the \(A_2^{\oplus 3}\) summand.

\begin{figure}[ht]
  \centering
  \begin{tikzpicture}
    \draw[thick]  (0,0) -- (0,2) -- (1.5,1.5) -- (1.5,-0.5) -- (0,0) -- (0,2) -- (-1.5,1.5) -- (-1.5,-0.5) -- (0,0);
    \draw[dashed] (0,1.75) -- (1.5,1.25) (0,1.70) -- (1.5,1.20) (0,1.65) -- (1.5,1.15)
    (1.2,1.0) node {\tiny \(\times 3\)};
    \begin{scope}
      \draw[color=red, thick] (0,0.6) -- (0.5,0.44) (0,0.7) -- (0.5,0.54) (0,1.0) -- (0.5,0.84);
      \draw[color=red, thick] (0,0.6) -- (-0.5,0.44) (0,0.7) -- (-0.5,0.54);
      \draw[color=red, thick, dashed] (0,1.0) -- (-0.5,0.84);
      \draw[dashed, thick] plot[smooth, tension=1] coordinates {(-0.2,0.8) (-0.5,0.95) (-0.9,0.9)};
    \end{scope}
    \draw (1.7,0.75) node (K) {};
    \begin{scope}[xshift=5cm]
    \draw[thick]  (0,0) -- (0,2) -- (1.5,1.5) -- (1.5,-0.5) -- (0,0) -- (0,2) -- (-1.5,1.5) -- (-1.5,-0.5) -- (0,0);
    \draw[color=red, thick]  plot[smooth, tension=1] coordinates {(0.2,1.1) (0.3,0.8) (0.5,0.7)};
    \draw[dashed] (0,1.75) -- (1.5,1.25) (0,1.70) -- (1.5,1.20) (0,1.65) -- (1.5,1.15)
    (1.2,1.0) node {\tiny \(\times 3\)};
    \begin{scope}
      \draw[color=red, thick] (0,0.6) -- (0.5,0.44) (0,0.7) -- (0.5,0.54);
      \draw[color=blue, thick, dashed](0,1.0) -- (0.5,0.84);
      \draw[color=red, thick] (0,0.6) -- (-0.5,0.44) (0,0.7) -- (-0.5,0.54);
      \draw[color=black, thick, dashed] (0,1.0) -- (-0.5,0.84);
    \end{scope}
    \draw (-1.7,0.75) node (S) {};
    \draw (1.7,0.75) node (S2) {};
  \end{scope}
   \begin{scope}[xshift=8.5cm, xscale=0.75]
     \draw[thick] plot[smooth, tension=2] coordinates {(0,0) (0.5,1.5) (1,-.25)};
     \draw[thick] (3,1.5) -- (3,0);
     \draw[thick] (1,-.25) -- (3,0);
     \draw[dashed] (0,0) -- (3,0);
     \draw[thick] (0,0) -- (1,0);
     \draw[thick] (0.5,1.5) -- (3,1.5);
     \draw[thick, red] (2,0.75) -- (3,0.75) (2,0.5) -- (3,0.75);
     \draw (-0.5,0.75) node (T) {};
   \end{scope}
   \draw[thick, ->, dashed] (K) edge (S);
   \draw[thick, ->] (S2) edge (T);
 \end{tikzpicture}
  \caption{The stable model in the case \(i = 2\) is obtained by an M1 modification along a \((-1)\)-curve on the left (dashed red) followed by a morphism that contracts the entire right-hand component, contracts a \((-1)\)-curve on the left (dashed), and folds the double curve into a \(\P^1\), resulting in a non-normal surface.}\label{fig:kulstab212}
\end{figure}
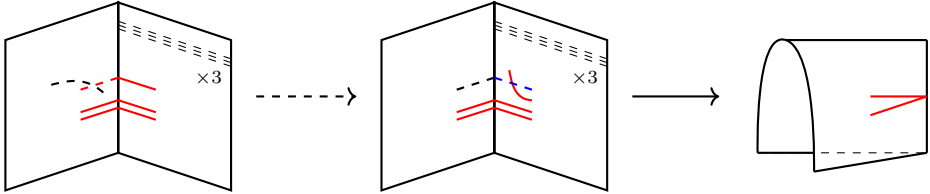
\subsubsection{The case \(i = 3\)}\label{sec:21i3}
Write \(C_0 = P_0 \cup Q_0\) and \(C_1 = P_1 \cup Q_1\), where \(Q_j\) are smooth curves of genus \(1\).
By \Cref{sec:k1-01-12}, we know that each \(X_j\) is a blow-up of \(\overline X_j\), a \(\delpezzo_3\), with an exceptional divisor forming a \((-1,-2,-2)\)-chain of which the first curve is \(P_j\).
The divisor \(R_j\) is negative on \(P_j\).

We make M1 modifications along \(P_0\) and \(P_1\), resulting in \(X'\) with nef \(R'\).
On each side, the divisor \(R'\) is negative only on the middle component of the chain, which is now a \((-1)\)-curve.
Contracting them yields the stable model (see \Cref{fig:kulstab213}).
Passing to the stable model does not lose any information.
Therefore, \(\mathfrak F_J = 0\).
\begin{figure}[ht]
  \centering
  \begin{tikzpicture}
    \draw[thick]  (0,0) -- (0,2) -- (1.5,1.5) -- (1.5,-0.5) -- (0,0) -- (0,2) -- (-1.5,1.5) -- (-1.5,-0.5) -- (0,0);
    \draw[color=red, thick, dashed] (0,0.6) -- (0.5,0.44);
    \draw[color=red, thick] (0,0.9) -- (0.5,0.74);
    \draw[color=red, thick]  (0,0.9) -- (-0.5,0.74) ;
    \draw[color=red, thick, dashed] (0,1.3) -- (-0.5,1.14);
    \draw[color=red, thick] (0,0.6) -- (-0.5,0.44)   (0,1.3) -- (0.5,1.14);
    \draw[dashed, thick] plot[smooth, tension=1] coordinates {(0.2,0.4) (0.5,0.55) (0.9,0.5)};
    \draw[dashed, thick] plot[smooth, tension=1] coordinates {(-0.2,1.35) (-0.5,1.05) (-0.8,1.0)};
    \draw (1.7,0.75) node (K) {};
    \begin{scope}[xshift=5cm, yshift=-0.25cm]
      \draw[thick]  (0,0) -- (1.5,1) -- (0,2)-- (0,0) -- (-1.5,1) -- (0,2) -- (0,0);
      \draw[color=red, thick] (0,1.1) -- (0.5,0.94); 
      \draw[color=red, thick] (0,1.1) -- (-0.5,0.94);
      \draw[color=red, thick]  plot[smooth, tension=1] coordinates {(0.15,1.5) (0.25,1.2) (0.5,1.1)};
      \draw[color=red, thick]  plot[smooth, tension=1] coordinates {(-0.15,.6) (-0.25,.8) (-0.5,.8)};
      \draw[color=blue, dashed, thick] (0,1.4) -- (0.5,1.24) (0,0.8) -- (-0.5,0.64) ;
      \draw (-1.7, 1) node (S) {};      
    \end{scope}
    \draw[thick, dashed, ->] (K) edge (S);
  \end{tikzpicture}
  \caption{The stable model in the case \(i = 3\) is obtained by M1 modifications along the \((-1)\)-curves (dashed red) followed by the contraction of the \((-1)\)-curves (dashed).  The divisor is shown in red (dashed and solid).}\label{fig:kulstab213}
\end{figure}
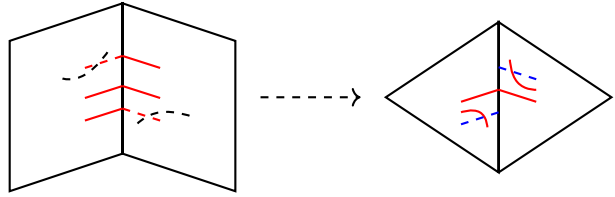
\subsubsection{The case \(i = 4\)}\label{sec:21i4}
Write \(C_0 = P_0 \cup Q_0\), where \(Q_0\) is a smooth curve of genus \(2\), and \(C_1 = P_1 \cup Q_1\), where \(Q_1\) is a smooth curve of genus 0.
From \Cref{sec:k0-01-22}, we know that \(X_0\) is the blow-up of \(\overline X_0\), a \(\delpezzo_1\), whose exceptional curve is \(P_0\).
From \Cref{sec:k2-01-02}, we know that \(X_1\) is the blow-up of \(\P^2\) with exceptional locus consisting of a \((-1,-2,-2)\)-chain (whose first curve is \(P_1\)) and six other \((-1)\)-curves.
The divisor \(R_0\) is negative on \(P_0\) and \(R_1\) is negative on \(P_1\).
We make M1 modifications on \(P_0\) and \(P_1\), resulting in \((X', R')\) (see \Cref{fig:kulstab214}).
Now \(R_0' \subset X'_0\) is ample but \(R_1' \subset X_1'\) is still negative on \(Q_1'\) (which is now a \((-1)\)-curve).
We make another M1 modification along \(Q_1'\), resulting in \((X'',R'')\).
Now \(R_0'' \subset X_0''\) is nef and trivial only on the double curve.
On the other hand \(R_1'' = 0\).
It follows that the stable model \(\overline X\) is obtained by contracting the entire \(X''_1\) component to a point.
In other words, we have a map \(X''_0 \to \overline X\) that contracts the double curve (of self-intersection \(-1\)) to a point (so \(\overline X\) has an $\widetilde E_8$ elliptic singularity).
The stable model only retains the moduli of \((\overline X_0, D)\), captured by the \(E_8\) component of the period \(\psi\) of \(X\).
Therefore, \(\mathfrak F_{J} \subset J^{\perp}_{T_{\rho}}/J\) is the saturation of the complementary \(A_2\) summand.
\begin{figure}[ht]
  \centering
  \begin{tikzpicture}[xscale=0.9]
    \draw[thick]  (0,0) -- (0,2) -- (1.5,1.5) -- (1.5,-0.5) -- (0,0) -- (0,2) -- (-1.5,1.5) -- (-1.5,-0.5) -- (0,0);
    \draw[color=red, thick, dashed] (0,0.6) -- (0.5,0.44);
    \draw[color=red, thick] plot[smooth, tension=2] coordinates {(0,0.9) (0.5,1.1) (0,1.3)};
    \draw[color=red, thick, dashed] (0,1.3) -- (-0.5,1.14);
    \draw[color=red, thick] (0,0.6) -- (-0.5,0.44) (0,0.9) -- (-0.5, 0.74);
    \draw (1.7,0.75) node (K) {};
    \begin{scope}[xshift=4.5cm]
      \draw[thick]  (0,0) -- (0,2) -- (1.5,1.5) -- (1.5,-0.5) -- (0,0) -- (0,2) -- (-1.5,1.5) -- (-1.5,-0.5) -- (0,0);
      \draw[color=red, thick] (0,0.9) -- (-0.5, 0.74);
      \draw[color=red, thick, dashed]  plot[smooth, tension=1] coordinates {(0.15,1.4) (0.5,0.9) (0,0.9)};
      \draw[color=red, thick]  plot[smooth, tension=1] coordinates {(-0.15,.3) (-0.25,.6) (-0.5,.6)};
      \draw[color=blue, dashed, thick] (0,1.3) -- (0.5,1.14) (0,0.6) -- (-0.5,0.44) ;
      \draw (-1.7, 0.75) node (S1) {};
      \draw (1.7, 0.75) node (S2) {};      
    \end{scope}
    \begin{scope}[xshift=9cm]
      \draw[thick]  (0,0) -- (1.5,1) -- (0,2)-- (0,0) -- (0,2) -- (-1.5,1.5) -- (-1.5,-0.5) -- (0,0);
      \draw[color=red, thick]  plot[smooth, tension=1] coordinates {(-0.15,1.2) (-0.25,1.0) (-0.5,0.9)};
      \draw[color=red, thick]  plot[smooth, tension=1] coordinates {(-0.15,.4) (-0.25,.6) (-0.5,.6)};
      \draw[color=blue, dashed, thick] (0,0.6) -- (-0.5,0.44) (0,1.1) -- (-0.5,1.1) ;
      \draw (-1.6, 0.75) node (T1) {};
      \draw (1.7, 0.75) node (T2) {};      
    \end{scope}
    \begin{scope}[xshift=12cm, yshift=-0.25cm]
      \draw[thick]  (0,0) -- (0,2);
      \draw[thick] plot[smooth,tension=1] coordinates {(0,0) (1,0.3) (1.4,0.9) (1.5,1)};
      \draw[thick] plot[smooth,tension=1] coordinates {(0,2) (1,1.7) (1.4,1.1) (1.5,1)};
      \draw[thick, red] (1.5,1) -- (0.5,1.5) (1.5,1) -- (0.5,0.5);
      \draw (-0.2, 1) node (U) {};
    \end{scope}
    \draw[thick, dashed, ->] (K) edge (S1);
    \draw[thick, dashed, ->] (S2) edge (T1);
    \draw[thick, ->] (T2) edge (U);
  \end{tikzpicture}
  \caption{The stable model in the case \(i = 4\) is obtained by two M1 modifications along the \((-1)\)-curves (dashed red) followed by an M1 modification along another \((-1)\)-curve (dashed red in the middle), and finally the contraction of the right component.  It has an elliptic singularity. The divisor is shown in red.}\label{fig:kulstab214}
\end{figure}
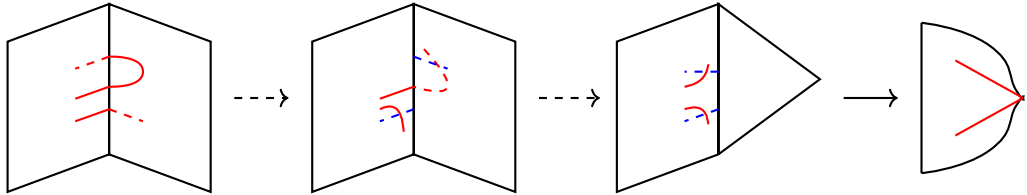
\subsubsection{The case \(i = 5\)}\label{sec:21i5}
Write \(C_1 = P_1 \cup Q_1\), where \(Q_1\) is a curve of genus \(0\).
From \Cref{sec:k0-13}, we know that \(X_0\) is the blow-up of \(\overline X_0\), a \(\delpezzo_3\), in 3 points that form a \((\Z/3\Z)\)-orbit.
From \Cref{sec:k2-01-02}, we know that \(X_1\) is the blow-up of \(\P^2\) with exceptional locus consisting of a \((-1,-2,-2)\)-chain (of which the first curve is \(P_1\)) and 6 points that form two \((\Z/3\Z)\)-orbits.
This \(X\) is related to the Kulikov surface from the case \(i = 2\) by the sequence of following modifications: three successive M1 modifications along the components of the \((-1,-2,-2)\)-chain on \(X_1\) followed by three M1 modifications along the exceptional curves of \(X_0 \to \overline X_0\).
We get the same stable model as the case \(i = 2\).

\begin{proof}[Conclusion of the proof of \Cref{thm:ksba21}]
  Having settled the cases of the four cusps in \Cref{sec:21i1}, \Cref{sec:21i2}, \Cref{sec:21i3}, and \Cref{sec:21i4}, the proof of \Cref{thm:ksba21} is complete.
  The fifth case treated in \Cref{sec:21i5} gives the same information as the second case treated in \Cref{sec:21i2}.
\end{proof}

Observe that for \(i = 1,2,3,4\), the generic stable pair corresponding to the \(i\)-th cusp in \Cref{subsec:stable_model_21} matches with the boundary component numbered \(1,4,2,5\), respectively, of \(\overline{\mathcal{P}}_2\) from \cite[Table 1]{laz:16}.

\bibliographystyle{amsalpha} % Valery's eis.tex uses this format
\bibliography{references}

\end{document}